\documentclass[11pt,reqno]{amsart}

\usepackage{amsmath}
\usepackage{amsthm} 
  
\usepackage{graphicx}
\usepackage{color}
\usepackage{geometry}
\usepackage{fullpage} 
\usepackage[all,color]{xy}
\usepackage{multirow}
\usepackage{booktabs}
\usepackage{tikz}
\usepackage{tikz-cd}
\usetikzlibrary{decorations.pathreplacing}
\usetikzlibrary{arrows}

\usepackage{calrsfs}
\usepackage[OT2,T1]{fontenc}
\usepackage[T1]{fontenc} 
\usepackage{textcomp}
\usepackage{times}
\usepackage[scaled=0.90]{helvet} 

\usepackage{tensor}
\usepackage[e]{esvect} 

\usepackage[hidelinks]{hyperref} 
\usepackage{etoolbox}
\gappto{\UrlBreaks}{\UrlOrds}
\patchcmd{\section}{\scshape}{\bfseries}{}{}
\makeatletter
\renewcommand{\@secnumfont}{\bfseries}
\makeatother

\usepackage{enumitem}

\newtheorem{introtheorem}{Theorem}

\newtheorem{introcorollary}[introtheorem]{Corollary}
\newtheorem{introprop}[introtheorem]{Proposition}
\theoremstyle{definition}
\newtheorem{introexample/}[introtheorem]{Example}
\newenvironment{introexample}
  {%
   \pushQED{\qed}\begin{introexample/}}
  {\popQED\end{introexample/}}

\theoremstyle{plain}
\newtheorem{theorem}{Theorem}[subsection]
\newtheorem{proposition}[theorem]{Proposition}

\newtheorem{lemma}[theorem]{Lemma}
\newtheorem{corollary}[theorem]{Corollary}

\theoremstyle{definition}
\newtheorem{definition}[theorem]{Definition}

\newtheorem*{definition*}{Definition}
\newtheorem{example/}[theorem]{Example}
\newenvironment{example}
  {%
   \pushQED{\qed}\begin{example/}}
  {\popQED\end{example/}}
\newtheorem{examples/}[theorem]{Examples}

\newtheorem*{example*}{Example}
\newtheorem{remark}[theorem]{Remark}
\newtheorem*{remark*}{Remark}

\theoremstyle{plain}

\usepackage{amsfonts}
\usepackage{amssymb}
\usepackage{hyperref}
\usepackage{mathtools}
\usepackage{mathabx}

\renewcommand{\bar}{\overline} % looks better?
\renewcommand{\hat}{\widehat} % looks better?
\renewcommand{\tilde}{\widetilde}
\renewcommand{\check}{\widecheck} 

\DeclareMathOperator{\Z}{\mathbf{Z}}

\DeclareMathOperator{\Q}{\mathbf{Q}}
\DeclareMathOperator{\R}{\mathbf{R}}
\DeclareMathOperator{\C}{\mathbf{C}}
\DeclareMathOperator{\F}{\mathbf{F}}

\DeclareMathOperator{\Ho}{H}

\DeclareMathOperator{\Hom}{Hom}
\DeclareMathOperator{\Aut}{Aut}
\DeclareMathOperator{\tr}{tr}

\DeclareMathOperator{\Un}{U}

\DeclareMathOperator{\trans}{\underline{t}}

\DeclareMathOperator{\ab}{ab}
\DeclareMathOperator{\conj}{conj}

\DeclareMathOperator{\Sp}{\sigma}
\DeclareMathOperator{\Irr}{Irr}
\DeclareMathOperator{\Ind}{Ind}
\DeclareMathOperator{\Res}{Res}
\DeclareMathOperator{\one}{\mathbf{1}}

\DeclareMathOperator{\fp}{\mathfrak{p}}

\usepackage{hyphenat}

\begin{document}

\date{\today\ (version 1.0)} 
\title{Twisted isospectrality, homological wideness and isometry}
\author[G.~Cornelissen]{Gunther Cornelissen}
\address{\normalfont Mathematisch Instituut, Universiteit Utrecht, Postbus 80.010, 3508 TA Utrecht, Nederland}
\email{g.cornelissen@uu.nl}
\author[N.~Peyerimhoff]{Norbert Peyerimhoff}
\address{\normalfont Department of Mathematical Sciences, Durham University, DH1 3LE, Great Britain}
\email{norbert.peyerimhoff@durham.ac.uk}
\thanks{GC thanks Durham University and NP thanks Utrecht University for hospitality. The authors thank Jens Funke for sparking this project by bringing them together at the 2017 BMC, and Alex Bartel, Stefan Bechtluft-Sachs, Nigel Higson, Gerhard Knieper, Matthias Lesch, Jeroen Sijsling and Harry Smit for patiently answering some of their questions and helping out with some of the proofs.
}


\begin{abstract} \noindent Given a manifold (or, more generally, a developable orbifold) $M_0$ and two closed Riemannian manifolds $M_1$ and $M_2$ with a finite covering map to $M_0$, we give a spectral characterisation of when they are equivalent Riemannian covers (in particular, isometric), assuming a representation-theoretic condition of \emph{homological wideness}: if $M$ is a common finite cover of $M_1$ and $M_2$ and $G$ is the covering group of $M$ over $M_0$, the condition involves the action of $G$ on the first homology group of $M$ (it holds, for example, when there exists a rational homology class on $M$ whose orbit under $G$ consists of $|G|$ linearly independent homology classes). 
We prove that, under this condition, Riemannian covering equivalence is the same as isospectrality of finitely many twisted Laplacians on the manifolds, acting on sections of flat bundles corresponding to specific representations of the fundamental groups of the manifolds involved. 
Using the same methods, we provide spectral criteria for weak conjugacy and strong isospectrality. In the negative curvature case, we formulate an analogue of our result for the length spectrum. The proofs are inspired by number-theoretical analogues. 
We study examples where the representation theoretic condition does and does not hold. For example, when $M_1$ and $M_2$ are commensurable non-arithmetic closed Riemann surfaces of negative Euler characteristic, there is always such an $M_0$, and the condition of homological wideness always holds. 
 \end{abstract}

\maketitle

%\begin{flushright} {\small 
%\emph{Come on let's twist again} \\ 
%--- ``Let's twist again'', written by  Kal Mann and Dave Appell}
%\end{flushright} 

\tableofcontents

\newpage

\section*{Introduction}

Let $M_1$ and $M_2$ be a pair  of connected closed oriented smooth Riemannian manifolds. There exist such $M_1$ and $M_2$ that are not isometric, but \emph{isospectral}, i.e., they have the same Laplace spectrum with multiplicities. This leaves open the question whether equality of spectra of other geometrically defined operators on $M_1$ and $M_2$ is equivalent to $M_1$ and $M_2$ being isometric. In this paper we investigate the use of twisted Laplacians (acting on sections of flat bundles constructed from representations of fundamental groups) in answering this question. 
We will see that under two conditions on $M_1$ and $M_2$, equality of finitely many suitably twisted Laplace spectra implies isometry of the manifolds. That at least some condition is necessary for such a result to hold is illustrated by the existence of simply connected isospectral non-isometric manifolds \cite{SchuethAnn}.  

\medskip

The \textbf{first condition} is the following: we suppose that the manifolds $M_1$ and $M_2$ are finite Riemannian coverings of a developable Riemannian orbifold $M_0$ (meaning that the universal covering of $M_0$ is a manifold), expressed through a diagram 
\begin{equation} \label{m0}
  \xymatrix@C=2mm@R=4mm{ M_1\ar[dr]_{p_1} & & M_2\ar[dl]^{p_2} \\ & M_0 & }  
\end{equation} 
Such a diagram may be extended to a diagram of finite coverings
\begin{equation} \label{sunadasetup}
  \xymatrix@C=6mm@R=6mm{ & M \ar[dl]^{H_1}_{q_1} \ar[dd]^(0.56)G_(0.58){q} \ar[dr]_{H_2}^{q_2} & \\ M_1\ar[dr]_{p_1} & & M_2\ar[dl]^{p_2} \\ & M_0 & }
\end{equation} 
where $M$ is a connected closed smooth Riemannian manifold $M$ with $q_1 \colon M \twoheadrightarrow M_1:=H_1 \backslash M$, $q_2 \colon M \twoheadrightarrow  M_2:=H_2 \backslash M$ and $q \colon M  \twoheadrightarrow M_0:=G\backslash M$ Galois covers (Proposition \ref{fromdiagramtodiagram}); in particular, $M_1$ and $M_2$ are commensurable. Commensurability is in general weaker than the existence of a diagram (\ref{m0}) (see Proposition \ref{prop:commdiag}). However, if $M_1$ and $M_2$ are commensurable hyperbolic manifolds, then a diagram (\ref{m0}) (with an orbifold $M_0$) exists if the corresponding lattices are not arithmetic (Proposition \ref{2to1}).  

\medskip

The \textbf{second condition} is related to the action of $G$ on the first homology group of $M$. In terms of the data in diagram  (\ref{sunadasetup}), we require that for some prime number $\ell$ not dividing $|G|$, the $\F_\ell[G]$-module $\Ho_1(M,\F_\ell)$ contains the permutation representation of $G$ acting by left multiplication on the cosets $G/H_i$ for either $i=1$ or $i=2$ (or both), i.e., $(\Ind_{H_i}^G \one) \otimes_{\Z} \F_\ell$ is an $\F_\ell[G]$-submodule of $\Ho_1(M,\F_\ell)$. Concretely, this means that if there are $n$ cosets, then the $\F_\ell$-vector space $\Ho_1(M,\F_\ell)$ contains $n$ linear independent vectors that are permuted in the same way as those $m$ cosets under the action of any $g \in G$. 

This second condition is implied by the stronger requirement that the $G$-action is \emph{$\F_\ell$\hyp{}homologically wide}, where, for a general field $K$, we say that the $G$-action is  \emph{$K$\hyp{}homologically wide} if  the regular representation $K[G]$ of $G$ occurs in the homology representation $G \rightarrow \Aut(\Ho_1(M,K))$ (Lemma \ref{hwstars}). This condition has the advantage of no longer referring to $M_1$ and $M_2$ (or $H_1$ and $H_2$). 

An even stronger, geometrically tangible, condition is that the action of $G$ on $M$ is $\Q$\hyp{}homologically wide; this means that there is a free homology class $\omega$ on $M$ such that the elements $\{g\omega \colon g \in G\}$ are linearly independent in $\Ho_1(M,\Q)$. This is diametric to the condition of homologically trivial group actions found more frequently in the literature (see, e.g., \cite{Weinberger}). 
A $\Q$\hyp{}homologically wide action is $\F_\ell$\hyp{}homologically wide for any $\ell$ coprime to $|G|$ (Lemma \ref{fromQtoell}). 

In Section \ref{exhomwide}, we discuss $\Q$\hyp{}homological wideness  for certain low dimensional manifolds and locally symmetric spaces.  For non-trivial fixed-point free group actions on orientable surfaces, $\Q$\hyp{}homologically wideness is equivalent to any of the spaces $M,M_0,M_1$ or $M_2$ having negative Euler characteristic (Proposition \ref{prop:homwidecritriemsurf}). A (not necessarily fixed point free) holomorphic group action on a Riemann surface $M$ is $\Q$\hyp{}homologically wide if the Euler characteristic of the quotient surface $M_0$ satisfies $\chi_{M_0}<0$ (Proposition \ref{prop:homwidecritriemorb}). In higher dimension, the picture can vary widely: for any $n \geq 3$, we use standard surgery methods to construct an $n$-manifolds with a free action by any given non-trivial group that is or is not $\Q$\hyp{}homologically wide (Proposition \ref{hwdim>3} and Corollary \ref{nhwdim>3}).  Since locally symmetric spaces of rank $\geq 2$ have trivial rational homology, no non-trivial group action on them can be $\Q$\hyp{}homologically wide (see \S \ref{lssr}). On the other hand, in rank one, the condition relates to  decomposition results for automorphic representations (see \S \ref{hwls1}). 
A disadvantage of $\Q$\hyp{}homological wideness is that it discards torsion homology; in \S\ref{torsionex} we give an example of some non-trivial $\F_5$-homologically wide group actions on the Seifert--Weber dodecahedral space (that has homology $\Ho_1(M,\Z) \cong \left({\Z}/{5}{\Z}\right)^3$). 

\medskip

Our results involve spectra of \emph{twisted Laplace operators} 
 $\Delta_\rho$ corresponding to unitary representations $\rho \colon \pi_1(M) \rightarrow \Un_n(\C)$; these  are symmetric second order elliptic differential operators acting on sections of flat bundles $E_\rho$ over $M$; the sections are conveniently described as smooth vector valued functions on the universal cover that are equivariant  with respect to the representation, and on these, $\Delta_\rho$ acts (componentwise) like the usual Laplacian of the universal cover (cf.\ Section \ref{prelim}). In fact, our representations will factor through a specific finite group, allowing for a very concrete description of the operators (cf.\ Subsections \ref{tl} \& \ref{tlf}). Denote the spectrum of such an operator by $\Sp_M(\rho)$, where the index $M$ indicates that the Laplacian is defined on sections over the space $M$. 

\medskip

Returning to the setup in diagram \eqref{m0},  
we call $M_1$ and $M_2$ \emph{equivalent Riemannian covers of $M_0$} if there is an isometry between them induced by a conjugacy of the fundamental groups of $M_1$ and $M_2$ inside that of $M_0$ (since $M_0$ is developable, its universal covering is a manifold and we mean the subgroup of its isometry group that fixes $M_0$ pointwise, see Lemma \ref{over}). Our main result states that in this situation isometry can be detected via the spectra of finitely many specific twisted Laplacians. 

\begin{introtheorem} \label{main} Suppose we have a diagram \textup{(\ref{m0})} and suppose that the action of $G$ on $M$ in the extended diagram \textup{(\ref{sunadasetup})} is $\F_\ell$\hyp{}homologically wide for some prime $\ell$ coprime to $|G|$. 
Then $M_1$ and $M_2$ are equivalent Riemannian covers of $M_0$ \textup{(}in particular, isometric\textup{)} if and only if the multiplicity of zero in the spectra of a finite number of specific twisted Laplacians on $M_1$ and $M_2$ is equal.  
In fact, at most $2 \ell |\mathrm{Hom}(H_2,\C^*)|$ equalities suffice. 
\end{introtheorem}

For a precise formulation of the required twists and a more technical condition (denoted $(\ast)$) that is weaker than homological wideness, the reader is encouraged to glance at Theorem \ref{maindetail}, which is a more detailed formulation of Theorem \ref{main}, our main result. The detailed formulation reveals that the representations occurring in the theorem are constructed explicitly using induction and restriction of special characters (termed ``solitary'' below) on groups corresponding to a specific finite Riemannian covering of $M$, whose existence is guaranteed by the assumption of homological wideness (or the weaker requirement $(\ast)$). Riemannian equivalence over $M_0$ is the same as conjugacy of $H_1$ and $H_2$ in $G$; a merit of the theorem is to show that this can be verified via a spectral geometric criterion using twists,  where the twists on $M_1$ are constructed using information from $M_2$ and vice versa.

Sunada \cite{Sunada} showed that if we are given a diagram of the form (\ref{sunadasetup}), and $H_1$ and $H_2$ are \emph{weakly conjugate} (meaning that the permutation representations given by the action of $G$ by left multiplication on the cosets $G/H_1$ and $G/H_2$ are isomorphic),  $M_1$ and $M_2$ are isospectral for the Laplace operators. The following example illustrates our theorem in such a situation. 

\begin{introexample}  
 We provide the explicit data for what is maybe the oldest example of weakly conjugate subgroups, due to Ga{\ss}mann \cite{Gassmann}: $G=S_6$, and $H_1 = \langle (12)(34), (13)(24) \rangle$, $H_2=\langle (12)(34), (12)(56) \rangle$, both isomorphic to the Klein four-group, but not conjugate inside $S_6$ \cite[pp.\ 674--675]{Gassmann}. As in \cite[p.\ 174]{Sunada}, choose a compact Riemann surface $M_0$ of genus $2$ and a surjective group homomorphism $ \pi_1(M_0) \twoheadrightarrow G$. This leads to a diagram of the form
(\ref{sunadasetup}) \cite[\S 2]{Sunada}, and homological wideness is immediate from Proposition \ref{prop:homwidecritriemsurf}.  In this case, $M_1$ and $M_2$ are inequivalent covers of $M_0$, but they have the same Laplace spectrum \cite[\S 2]{Sunada}. We can set $\ell=7$ in the main theorem, and, with the group $G$ having order $720$, inequivalence of $M_1$ and $M_2$ may be verified purely spectrally by checking $56$ equalities of multiplicities of zero in the spectrum of twisted Laplacians corresponding to representations of dimension $6!/4 = 180$. 
\end{introexample} 

To show the scope of our result, we discuss several more examples in Section \ref{obstructions}. Our result should be contrasted with \cite[Thm.\ 1.1]{Arapura}, where it is shown that large non-arithmetic hyperbolic manifolds $M$ admit arbitrary large sets of strongly isospectral but pairwise non-isometric finite Riemann coverings. 

Our method of proof for Theorem \ref{main}, presented in Section \ref{setup}--\ref{sec:covers}, is based on a similar construction of Solomatin \cite{Solomatin} for algebraic function fields, which in turn is based on number theoretical work of Bart de Smit in \cite{Lseries}. The analogy between number theory and spectral differential geometry was pioneered by Sunada  \cite{Sunada} (see also the survey \cite{SunadaSurvey}), and the importance of representation theoretical techniques was pointed out early on by Sunada \cite{SunadaNT} and Pesce \cite{PesceJFA}. We have given a self-contained presentation, with references to the number theory literature when appropriate. Our construction uses a certain wreath product of $G$ with a cyclic group; in  Subsection \ref{wreathuniversal}, we describe a universality property of such wreath products that should make their appearance look less suprising. 

We have formulated our results using spectra, but the analogy to number theory becomes most apparent by instead using spectral zeta functions of (twisted) Laplacians; that this is equivalent is explained in Subsection \ref{specz} and Proposition \ref{fixm0}, pointing to some subtleties concerning the multiplicity of zero in the spectrum of more general operators. 

From our earlier brief discussion of the two conditions in the theorem, we find the following specific result in dimension two. 
\begin{introcorollary} \label{maincor} Let $M_1,M_2$ be two
  commensurable non-arithmetic closed Riemann surfaces. Then they
  admit a diagram \textup{(\ref{m0})} and, assuming the corresponding
  orbifold $M_0$ satisfies $\chi_{M_0} < 0$, isometry of $M_1$ and
  $M_2$ can be checked by computing the multiplicity of zero in
at most
$4 (\chi_{M_1}\chi_{M_2}/(\chi^{\mathrm{orb}}_{M_0})^2)!)^2 $ twisted Laplace spectra, where
 $\chi^{\mathrm{orb}}_{M_0}$ is the orbifold Euler characteristic, defined in \eqref{deforbeuler}. 
\end{introcorollary}

The detailed formulation and proof can be found in Corollary \ref{maincordetail}. If one believes in a positive answer to the (open since several decades) question whether commensurability of such Riemann surfaces is implied by their isospectrality, then a purely spectral formulation of the corollary is possible, replacing ``commensurable'' by ``isospectral''. Intriguingly, the only cases where a positive answer to the open question is known are arithmetic \cite{Reid}, precisely the ones excluded in the corollary.

 In Section \ref{length}, we study the analogue of Theorem \ref{main}  for the \emph{length spectrum} on negatively curved manifolds. The statement about agreement of multiplicities of zero in certain spectra is changed into equality of the pole order of certain $L$-series  (details are found  in Theorem \ref{mainL}).  
 
 \begin{introtheorem} \label{intromainL} Suppose we have a diagram \textup{(\ref{m0})} of negatively curved Riemannian manifolds with (common) volume entropy $h$, and suppose that the action of $G$ on $M$ in the extended diagram \textup{(\ref{sunadasetup})} is $\F_\ell$\hyp{}homologically wide for some $\ell$ coprime to $|G|$. 
Then $M_1$ and $M_2$ are equivalent Riemannian covers of $M_0$ if and only if the pole order at $s=h$ of a finite number of specific $L$-series of representations on $M_1$ and $M_2$ is equal.  
\end{introtheorem}

 Sunada's result \cite{Sunada} quoted above says that weak conjugacy implies isospectrality, but the converse is not necessarily true; this leaves open the question to characterise weak conjugacy of $H_1$ and $H_2$ in $G$ in a spectral way using the associated manifolds $M_1$ and $M_2$. 
One of our intermediate results answers this question using induction and restriction from the trivial representation $\one$.
\begin{introprop} \label{mainprop} 
If we have a diagram \textup{(\ref{sunadasetup})}, then $H_1$ and $H_2$ are weakly conjugate if and only if the multiplicity of the zero eigenvalue in $\Sp_{M_i}(\Res_{H_i}^G \Ind_{H_j}^G \one)$ is independent of $i,j=1,2$. 
 \end{introprop}

This result, reformulated in Proposition \ref{mainpropdet}, is proven by an adaptation of a number theoretical argument of Nagata \cite{Nagata}. The crucial differential geometric ingredient is the spectral characterisation of the multiplicity of the trivial representation in any given representation (Lemma \ref{mult}). As a corollary, we get a spectral characterisation of \emph{strong isospectrality} (cf.\ Definition \ref{strongiso} and Corollary \ref{strong}). 

The most pressing question that remains open concerns the case where there is not necessarily a common finite orbifold quotient: are \emph{twisted isospectral} manifolds $M_1$ and $M_2$ (meaning that there is a bijection of all unitary representations of their fundamental groups such that the spectra of the corresponding twisted Laplacians are equal) with \emph{large} fundamental groups (i.e., containing a finite index subgroup with a non-abelian free group as quotient) isometric?  Another natural question that occurs in relation to homological wideness is the following: given a hyperbolic manifold $ \Gamma \backslash \mathbb{H}^n$ for a cocompact discrete group $\Gamma$ of isometries of hyperbolic $n$-space $\mathbb{H}^n$ ($n \geq 3$), determine the structure of its first homology $\Gamma^{\mathrm{ab}}$ as a $\Z[\mathrm{Out}(\Gamma)]$-module (compare \S \ref{sechyp}).  

Since our methods combine concepts and tools from different fields, we have been slightly more discursive on (folklore) details, background, recaps and examples. 

\section*{``Leitfaden''} 

The dependencies between sections is indicated in the graph below. In boldface is the minimal path to reach the main results (Theorems \ref{main} = \ref{maindetail} and Theorem \ref{intromainL} = \ref{mainL}), in grey a geometric instead of purely group-theoretical construction, and in regular font material on homological wideness and examples. 

$$
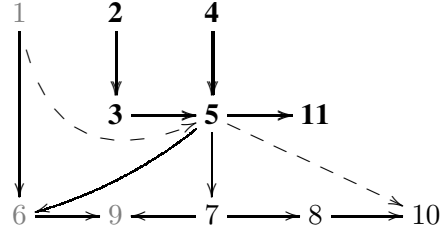
\xymatrix{ 
\color{gray}{{\ref{setup}}} \ar@[gray][dd] \ar@{-->}@[gray]@/_2.3pc/[drr]%_{\color{gray}{{\tiny 1.4.1}}} 
& \textbf{\ref{prelim}}  \ar[d] \ar@<0.2mm>[d] &  \textbf{\ref{sec:uniqmonstruc}} \ar[d] \ar@<0.2mm>[d]  & &  \\
& \textbf{\ref{detecting}} \ar[r] \ar@<0.2mm>[r]  &  \textbf{\ref{sec:covers}} \ar@[gray]@/^0.4pc/[dll] \ar[d] \ar[r] \ar@<0.2mm>[r]   \ar@{-->}[drr] & \textbf{\ref{length}}  &  \\ 
\color{gray}{{\ref{geocon}}} \ar@[gray][r] &  \color{gray}{{\ref{homwideNT}}} & \ref{homwide} \ar@[gray][l] \ar[r] & \ref{exhomwide} \ar[r]  & \ref{obstructions} 
} 
$$

\section{Manifold and orbifold constructions} \label{setup} 

In this section, we collect some information on various constructions of manifolds, orbifolds, and their covers. Most of it is straightforward, but since we could not always find convenient references, we include some of the proofs. 

\subsection{Setup} We fix the following notation for the remaining of the paper: we assume that  $M_1$ and $M_2$ are connected smooth oriented closed (i.e., compact with empty boundary) Riemannian manifolds that are \emph{finite} Riemannian coverings of a developable Riemannian orbifold $M_0$, i.e., there is a diagram of the form (\ref{m0}). 
Recall that the connected space $M_0$ being a \emph{developable orbifold} means that it is the quotient of a connected and simply connected Riemannian manifold $\tilde M_0$ by a cocompact discrete group of isometries $\Gamma_0$, that is the (orbifold) fundamental group of $M_0$. For the theory of (good/developable) orbifolds, see, e.g., \cite[Ch.\ 13]{Thurston} \cite{Choi} \cite[\S 1]{Orbi}. Given this setup, we can write $M_i = \Gamma_i \backslash \tilde M_0$ for finite index subgroups $\Gamma_i$ of $\Gamma_0$ acting
without  fixed points on $\tilde M_0$.

\begin{remark} In dimension $2$, there is a relation with the theory of branched coverings of surfaces \cite[\S 3]{PePe} (compare \cite[Thm.\ 13.3.6]{Thurston}). Consider an orbifold cover $M_1 \rightarrow M_0$ of degree $n$, where $M_0$ is a closed developable $2$-orbifold with $r$ elliptic (cone) points $\{x_i\}$: the local orbifold structure is that of rotation over an angle $2 \pi/e_i$. Assume that $M_1$ is a closed $2$-manifold.
To such an orbifold cover corresponds a branched degree $n$ cover of the corresponding underlying closed surfaces $\Sigma_1 \rightarrow \Sigma_0$ branched over the $r$ given points such that $x_i$ has exactly $n/e_i$ points above it with ramification index $e_i$ (in particular, $e_i$ divides $n$).
A diagram (\ref{m0}) corresponds to a diagram of branched coverings of the corresponding surfaces
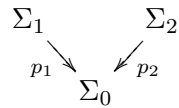
\begin{equation*} 
  \xymatrix@C=2mm@R=4mm{ \Sigma_1\ar[dr]_{p_1} & & \Sigma_2\ar[dl]^{p_2} \\ & \Sigma_0 & }  
\end{equation*} 
where, if $p_j$ has degree $d_j$ ($j=1,2$), we have the added data of $r$ branch points $\{x_i\}$ on $\Sigma_0$ and $r$ integers $e_i$ dividing $\mathrm{gcd}(d_1,d_2)$ such that there are exactly $d_j/e_i$ points above $x_i$ in $\Sigma_j$ for $j=1,2$ and $i=1,\dots,r$. 
\end{remark}

\begin{lemma} \label{over} Suppose we have a diagram of manifolds/orbifolds of the form \eqref{sunadasetup} and $\Gamma_0, \Gamma_1, \Gamma_2$ are as above. Suppose $\Gamma$ is such that $M = \Gamma \backslash \tilde M_0$. Then the following are equivalent: 
\begin{enumerate}
\item \label{over1} $H_1$ and $H_2$ are conjugate in $G$;
\item \label{over2} $\Gamma_1$ and $\Gamma_2$ are conjugate in $\Gamma_0$; 
\item \label{over3} $M_1$ and $M_2$ are equivalent Riemannian covers of $M_0$, i.e., there is an isometry $\varphi$ such that following diagram commutes
\begin{equation}
  \xymatrix@C=2mm@R=4mm{ M_1\ar[dr]_{p_1} \ar[rr]^\varphi & & M_2\ar[dl]^{p_2} \\ & M_0 & }  
\end{equation} 
\end{enumerate}
\end{lemma} 

\begin{proof} With $H_i =\Gamma_i / \Gamma$ and $\Gamma$ normal in $\Gamma_0$, we have that 
there exists a $g \in G$ satisfying $gH_1g^{-1} = H_2$ if and only if for any lift $\gamma_0 \in \Gamma_0$ of $g$, $\gamma_0 \Gamma_1 \gamma_0^{-1} = \Gamma_2$. This shows the equivalence of \eqref{over1} and \eqref{over2}.  To show the equivalence of  \eqref{over2} and \eqref{over3}, we argue as follows. If $gH_1g^{-1} = H_2$ for $g \in G$, a (finite) group of isometries of $M$, then $\varphi \colon M_1 \rightarrow M_2$ given by $\varphi(H_1 x) = H_2 g x$ (for $x \in \tilde M_0$) satisfies the requirements of \eqref{over3}. Conversely, if $M_1$ and $M_2$ are equivalent Riemannian covers of $M_0$, then the corresponding subgroups $\Gamma_1$ and $\Gamma_2$ of the orbifold fundamental group $\Gamma_0$ are conjugate. 
\end{proof} 

\subsection{Compositum} We set $M_{00}:=
\Gamma_{00} \backslash \tilde M_0$ for $\Gamma_{00}:=\Gamma_1 \cap \Gamma_2$. This is a finite connected common Riemannian cover of $M_1$ and $M_2$ (finiteness follows since $[\Gamma_0:\Gamma_1 \cap \Gamma_2] \leq [\Gamma_0:\Gamma_1] \cdot [\Gamma_0:\Gamma_2]$), which we call the \emph{compositum} $M_1 \bullet_{M_0} M_2$ of $p_1 \colon M_1 \rightarrow M_0$ and $p_2 \colon M_2 \rightarrow M_0$. Note that $M_{00}$ is a manifold since $\Gamma_i$ acts properly discontinuously without fixed points on $\tilde M_0$.  It is important to notice that the construction of $M_{00}$ and the covering maps $M_{00} \to M_i$, $i=1,2$, depend on the actual maps $p_1,p_2$, not just the spaces $M_0,M_1,M_2$, but it is customary to leave out the maps from the notation if they are clear. If necessary, we will write 
$M_1\tensor[_{p_1}]{\bullet}{_{M_0,p_2}} M_2$
 (or $M_1 \bullet_{M_0,p_2} M_2$ if $p_1$ is clear). 

\subsection{Fiber product} Given a diagram (\ref{m0}), we can also form the (orbifold) \emph{fiber product} $M_1 \times_{M_0} M_2$ (or, more precisely, $M_1 \tensor[_{p_1}]{\times}{_{M_0,p_2}} M_2$, where we use the same convention as before concerning including or leaving out the maps from the notation). We follow the construction of Thurston \cite[13.2.4]{Thurston} as explained in \cite[4.6.1]{Choi} (but where we are using left actions instead of right actions): 
$$M_1 \times_{M_0} M_2 =  \Gamma_0 \backslash \left( \tilde M_0 \times \Gamma_1 \backslash \Gamma_0 \times \Gamma_2 \backslash \Gamma_0 \right), $$ where the left action of $\gamma_0 \in \Gamma_0$ is given by $$ \gamma_0 (x,\Gamma_1 \gamma_0^{(1)}, \Gamma_2 \gamma_0^{(2)}) = (\gamma_0(x),\Gamma_1 \gamma_0^{(1)} 
\gamma_0^{-1}, \Gamma_2 \gamma_0^{(2)}\gamma_0^{-1} ) $$
for $x \in \tilde M_0, \gamma_0^{(i)} \in \Gamma_0 \ (i=1,2)$.  This has the universality property required for fiber products. Using the bijection $$\left( \Gamma_1 \backslash \Gamma_0 \times \Gamma_2 \backslash \Gamma_0 \right) / \Gamma_0 \rightarrow \Gamma_1 \backslash \Gamma_0 / \Gamma_2 \colon [(\alpha,\beta)] \rightarrow [\alpha \beta^{-1}],$$ we find an isometry 
$$ M_1 \times_{M_0} M_2 = \bigsqcup_{ \gamma \in  \Gamma_1 \backslash \Gamma_0/ \Gamma_2}  (\Gamma_1 \cap \gamma \Gamma_2 \gamma^{-1}) \backslash \tilde M_0. $$ 
In our situation, where $M_1$ and $M_2$ are manifolds, it follows that $ M_1 \times_{M_0} M_2 $ is also a  (possibly disconnected) manifold, since elements of $\Gamma_1 \cap \gamma \Gamma_2 \gamma^{-1}$ act without fixed points on $\tilde M_0$. The fiber product is a Riemannian manifold for the metric inherited from the universal covering (manifold) $\tilde M_0$. 

The fiber product contains the compositum as a connected component. However, it is not necessarily connected (similar to the tensor product of two number fields not always being isomorphic to their compositum); in fact, we see from the above that $ M_1 \times_{M_0} M_2$ has $|\Gamma_1 \backslash \Gamma_0/ \Gamma_2|$ connected components. The components need not all be isometric to the compositum, but if $M_i \rightarrow M_0$ are Galois covers, then they are (since then, $\gamma \Gamma_i \gamma^{-1} = \Gamma_i$ for all $\gamma \in \Gamma_0$, $i=1,2$). 

 There are two projections $$M_1 \times_{M_0} M_2  \twoheadrightarrow \Gamma_0 \backslash \left( \tilde M_0 \times \Gamma_i \backslash \Gamma_0 \right) \cong M_i,$$ where the latter isometry is given by $$(\{\gamma_0(x),\Gamma_i \gamma_0^{(i)} \gamma_0^{-1}\}_{\gamma_0 \in \Gamma_0}) \mapsto \Gamma_i \gamma_0^{(i)} x,$$ which is bijective, since $\Gamma_i$ has no fixed points on $\tilde M_0$. 

\begin{remark} There is a surjective map from the orbifold fiber product to the \emph{set-theoretic fiber product} \begin{align} \label{stfp} & M_1 \times_{M_0} M_2 \rightarrow \{(x_1,x_2) \in M_1 \times M_2 \colon p_1(x_1) = p_2(x_2) \}, \\ & (\{\gamma_0(x),\Gamma_1 \gamma_0^{(1)} \gamma_0^{-1},\Gamma_2 \gamma_0^{(2)} \gamma_0^{-1}\}_{\gamma_0 \in \Gamma_0}) \mapsto (\Gamma_1 \gamma_0^{(1)} x, \Gamma_2 \gamma_0^{(2)} x). \nonumber \end{align} If the action of $\Gamma_0$ on $\tilde M_0$ has fixed points, then this map is not necessarily injective, so the set-theoretic description of fiber product cannot be used in the orbifold setting, even if $M_1,M_2$ are manifolds. However, if $M_0$ is itself a manifold (so $\Gamma_0$ acts without fixed points), then the map in (\ref{stfp}) is an isometry of Riemannian manifolds (in particular, bijective), where the right hand side is a manifold since the projection maps are Riemannian submersions. 
\end{remark} 

We can complete the diagram (\ref{m0}) into diagrams of Riemannian coverings
\begin{equation*} 
  \xymatrix@C=2mm@R=4mm{ 
  & M_1 \times_{M_0} M_2 \ar[dl] \ar[dr] \ar[dd] & &  & M_1 \bullet_{M_0} M_2 \ar[dl] \ar[dr] \ar[dd] &\\ 
  M_1\ar[dr]_{p_1} & & M_2\ar[dl]^{p_2}  & \quad \quad M_1\ar[dr]_{p_1} & & M_2\ar[dl]^{p_2}  \\ 
  & M_0 & & & M_0 & } 
\end{equation*} 

\begin{lemma} \label{coprime}  
If $M_i \rightarrow M_0$ are Galois covers,  
then $M_1 \times_{M_0} M_2$ is connected and hence isometric to the compositum $M_1 \bullet_{M_0} M_2$ if and only if $\Gamma_0 = \langle \Gamma_1,\Gamma_2 \rangle$, the subgroup of $\Gamma_0$ generated by $\Gamma_1$ and $\Gamma_2$. This holds if the degrees of $M_i \rightarrow M_0$ are coprime. 
\end{lemma}
\begin{proof} The number of components of the compositum is the cardinality of the double coset space $\Gamma_1 \backslash \Gamma_0 /\Gamma_2$, and this is $1$ if and only if $\Gamma_0 = \Gamma_1 \Gamma_2$. Since $\Gamma_i$ are normal subgroups of $\Gamma_0$, the product $\Gamma_1 \Gamma_2$ is a subgroup; in fact, $\Gamma_1 \Gamma_2 =  \langle \Gamma_1,\Gamma_2 \rangle $. 
Hence the first statement holds. 
To see the second, note that we have a sequence of finite index group inclusions $$\Gamma_i \hookrightarrow \langle \Gamma_1,\Gamma_2 \rangle \hookrightarrow \Gamma_0,$$ so we find that $[\Gamma_0:\langle \Gamma_1,\Gamma_2 \rangle]$ divides $\mathrm{gcd}([\Gamma_0,\Gamma_1], [\Gamma_0,\Gamma_2])$. 
\end{proof} 

\begin{lemma} \label{lindis} 
If $M_i \rightarrow M_0$ are $G_i$-Galois of coprime degree, then \begin{enumerate} \item $M_1 \bullet_{M_0} M_2 \rightarrow M_0$ is $(G_1 \times G_2)$-Galois; \item 
$M_1 \bullet_{M_0} M_2 \rightarrow M_1$ is $G_2$-Galois and  $M_1 \bullet_{M_0} M_2 \rightarrow M_2$ is $G_1$-Galois. \end{enumerate} \end{lemma}

\begin{proof} 
We always have an exact sequence $$1 \rightarrow \Gamma_1 \cap \Gamma_2 \rightarrow \Gamma_0 \xrightarrow{\varphi} \Gamma_0/\Gamma_1 \times \Gamma_0/\Gamma_2.$$ Now since $[\Gamma_0:\Gamma_1 \cap \Gamma_2]$ is divisible by the coprime integers $[\Gamma_0:\Gamma_1]$ and  $[\Gamma_0:\Gamma_2]$, the map $\varphi$ is surjective. This proves the first statement. The second statement follows from $$\Gamma_1/(\Gamma_1 \cap \Gamma_2) \cong (\Gamma_0/(\Gamma_1 \cap \Gamma_2)) / (\Gamma_0/\Gamma_1) \cong \Gamma_0/\Gamma_2,$$ and similarly with the indices $1$ and $2$ interchanged. 
\end{proof}

\subsection{Normal closure} If $M_{00} \twoheadrightarrow M_0$ is a
general finite covering of connected spaces, then there exists
a connected space $M$ and a sequence of coverings
$M \twoheadrightarrow M_{00} \twoheadrightarrow M_0$ such that
$M \twoheadrightarrow M_0$ and $M \twoheadrightarrow M_{00}$ are
finite and Galois (i.e., the corresponding subgroup of the fundamental
group is normal), cf.\ \cite[Thm.\ 1]{Wolf}. We call the minimal such
$M$ the \emph{normal closure} of $M_{00} \twoheadrightarrow M_0$. In
terms of fundamental groups, if $M_{00}$ corresponds to the subgroup
$\Gamma_{00}$ of $\Gamma_0$, then $M$ corresponds to the subgroup $\Gamma$ of
$\Gamma_0$ given as the intersection of all $\Gamma_0$-conjugates of
$\Gamma_{00}$, the so-called \emph{(normal) core} of $\Gamma_{00}$ in
$\Gamma_0$.  Alternatively, the normal core $\Gamma$ is the kernel of the
action of $\Gamma_0$ by permutation on the cosets of $\Gamma_{00}$ in
$\Gamma_0$. In particular, since the index of $\Gamma_{00}$ in
$\Gamma_0$ is finite, so is the index of the normal core, and
$M \rightarrow M_0$ is indeed finite. In fact, by the above alternative description, the degree of $M \rightarrow M_0$, the index of the normal core in $\Gamma_0$, is bounded by the order of the group of permutations of the set $\Gamma_{00} \backslash \Gamma_0$. The number of elements of this set is the degree of the map $M_{00} \twoheadrightarrow M_0$, and hence we find a bound 
\begin{equation} \label{factorial} [\Gamma_0:\Gamma] \leq \deg(M_{00} \twoheadrightarrow M_0)! = [\Gamma_0:\Gamma_{00}]!. \end{equation} 

\begin{proposition} \label{fromdiagramtodiagram} Given manifolds $M_1$ and $M_2$ fitting into a diagram \textup{(\ref{m0})}, there exists a diagram of the form \textup{(\ref{sunadasetup})}. 
\end{proposition} 

\begin{proof} We start with a diagram (\ref{m0}) and add the normal closure of the compositum, to find a diagram of the form (\ref{sunadasetup}), where $M$ corresponds to the normal core of $\Gamma_1 \cap \Gamma_2$: 
 \[ M:=\Gamma \backslash \tilde M_0 \mbox{ for } \Gamma:= \bigcap\limits_{\gamma \in \Gamma_0 }\gamma (\Gamma_1 \cap \Gamma_2) \gamma^{-1}. \qedhere  \]
\end{proof}

\subsection{Commensurability} Two manifolds $M_1$ and $M_2$ are called \emph{commensurable} if they admit a common finite covering. Proposition \ref{fromdiagramtodiagram} implies that if $M_1$ and $M_2$ have a common finite (developable orbifold) quotient as in diagram (\ref{m0}), then they are commensurable. 

We briefly look at the converse statement. Assume that $M_1$ and $M_2$ are commensurable with common finite covering $M$, let $\tilde M$ denote the universal cover of $M$ with isometry group $\mathcal{I}$, and let $\Gamma$ denote the subgroup of $\mathcal I$ corresponding to $M$. Then $M_1$ and $M_2$ correspond to subgroups $\Gamma_1$ and $\Gamma_2$ of $\mathcal I$, and a diagram such as (\ref{m0}) exists if and only if $\Gamma_1$ and $\Gamma_2$ are of finite index in the subgroup of $\Gamma$ generated by $\Gamma_1$ and $\Gamma_2$. Indeed, if $M_0$ exists and corresponds to a subgroup $\Gamma_0$ of $\mathcal I$, then $\tilde M_0 = \tilde M$ and $\langle \Gamma_1, \Gamma_2 \rangle$ is a finite index subgroup of $\Gamma_0$, and by assumption $\Gamma_1$ and $\Gamma_2$ have finite index in $\Gamma_0$. Conversely, if $\langle \Gamma_1, \Gamma_2 \rangle$ is of finite index in $\Gamma$, we can set $M_0$ to be the corresponding orbifold $\langle \Gamma_1, \Gamma_2 \rangle \backslash \tilde M$.    

\begin{proposition} \label{2to1}
  If $\tilde M$ is a homogeneous space for a connected semisimple non-compact real Lie group with trivial center and no compact factors and $M_1$ and $M_2$ are quotients of $\tilde M$ corresponding to commensurable irreducible uniform lattices $\Gamma_1$ and $\Gamma_2$, then there exists a diagram of the form \textup{(\ref{m0})} if at least one of $\Gamma_1$ and $\Gamma_2$ is non-arithmetic. In this case
  both $\Gamma_1$ and $\Gamma_2$ are non-arithmetic.
\end{proposition} 

\begin{proof} Let $\mathcal I$ denote the isometry group of $\tilde M$. Since $\Gamma_1$ and $\Gamma_2$ are commensurable, their commensurator in $\mathcal I$, $$\mathcal C := \mathrm{Comm}_{\mathcal I}(\Gamma_i) = \{ g \in \mathcal I \colon [g\Gamma_i g^{-1}: \Gamma_i \cap g \Gamma_i g^{-1}] \cdot [\Gamma_i : \Gamma_i \cap g \Gamma_i g^{-1}] < \infty \}$$ is the same (indeed,
  if $\Gamma_1$ and $\Gamma_2$ are commensurable and $\Gamma_1$ is commensurable
  to $g \Gamma_1 g^{-1}$, then also $g \Gamma_1 g^{-1}$ and $g \Gamma_2 g^{-1}$ are commensurable and, therefore $\Gamma_2$ and $g \Gamma_2 g^{-1}$ are commensurable since commensurability is an equivalence relation).
  Margulis theorem \cite[Thm.\ (1), p.\ 2]{Margulis} (see also \cite[Props. 6.2.4, 6.2.5 and Thm. 6.2.6]{Zimmer}) states that either $\Gamma_i$ is not arithmetic and of finite index in $\mathcal C$; or $\Gamma_i$ is arithmetic and $\mathcal C$ is dense in $\mathcal I$ (to connect to the more general formulation in \cite{Margulis}: we consider a single semisimple group over the reals, and the ``anisotropy condition'' in loc.\ cit.\ is satisfied since we assume the group is not compact). This directly implies that either both lattices $\Gamma_1$ and $\Gamma_2$
are arithmetic or they are both not arithmetic. 

Consider the sequence of inclusions
\begin{equation} \label{seqcomm} \Gamma_i \hookrightarrow \langle \Gamma_1, \Gamma_2 \rangle \hookrightarrow \mathcal C \end{equation}  for $i=1,2$.

If neither of $\Gamma_i$ is arithmetic, the composed inclusion is of finite index. Hence the same holds for the first inclusion, and we can set $M_0:= \langle \Gamma_1, \Gamma_2 \rangle  \backslash \tilde M$. 
\end{proof} 
The proposition applies in particular to compact hyperbolic manifolds $M_i = \Gamma_i \backslash \mathbb{H}^n$, where $\mathbb{H}^n$ is hyperbolic $n$-space. 

\begin{proposition} \label{prop:commdiag}
There exist commensurable compact hyperbolic Riemann manifolds $M_1$ and $M_2$ of dimensions $2$ and $3$, for which a diagram of the form \textup{(\ref{m0})} does not exist. 
\end{proposition}  

\begin{proof} 
The uniform arithmetic isospectral, non-isometric lattices constructed by Vign\'eras satisfy this property \cite{Vigneras}; for commensurability, see 
\cite[Ch.\ IV]{Vignerasbook}. More information can be found in \cite[Prop. 3]{Chen}.\end{proof}

\begin{remark} More generally, Alan Reid has shown that isospectral manifolds corresponding to arithmetic lattices are commensurable (\cite{Reid}, compare \cite{Lola} for a quantitative statement). It is an open problem whether isospectral Riemann surfaces are always commensurable. Through work of Lubotzky, Samuels and Vishne, it is known that isospectrality does not  imply commensurability \cite{Lubotzky} in general.  
\end{remark} 

\section{Spectra, group representations and twisted Laplacians} \label{prelim} 

%\begin{flushright} {\small 
%\emph{We're gonna do the twist and it goes like this} \\ 
%--- ``Let's twist again'', written by  Kal Mann and Dave Appell} 
%\end{flushright} 

In this section, we review basic notions about spectra, group representations, and twisted Laplace operators (see, e.g., \cite{Isaacs}, \cite{Rosenberg}), \cite{SerreRep}).  

\subsection{Spectrum and spectral zeta function} \label{specz} Let $M=(M,g)$ denote a connected closed oriented smooth Riemannian manifold with Riemannian metric $g$.  Let $E$ denote a Hermitian line bundle on $M$ and $A$ a symmetric second order elliptic differential operator acting on smooth sections $C^\infty(M,E)$ of $E$ with non-negative eigenvalues. The operator extends to the corresponding space $L^2(M,E)$ of $L^2$-sections where it has a dense domain. The \emph{spectrum $\Sp_M(A)$ of $A$} (or $\Sp(A)$ is $M$ is fixed) is the multiset of eigenvalues of $A$, where the multiplicities of the elements in the set are given by the multiplicities of the eigenvalues. 

We make the following convenient notational conventions: if $S_1$ and $S_2$ are multisets, we let $S_1 \cup S_2$ denote the multiset consisting of elements of $S_1$ or $S_2$, where the multiplicity of an element is the sum of the multiplicities of that element in $S_1$ and $S_2$, and for a multiset $S$ and an integer $n$, we mean by $nS$ the multiset of elements of $S$ where all multiplicities are multiplied by $n$. 

The \emph{spectral zeta function} of $A$ as above  is defined as $$\zeta_{M,A}(s) = \zeta_A(s):= \sum\limits_{0 \neq \lambda \in \Sp(A)} \lambda^{-s}$$ with the sum not involving the zero eigenvalues. The function can be meromorphically continued to the entire complex plane \cite[Thm.\ 5.2]{Rosenberg}. Since $\zeta_A(s)$ is a Dirichlet series, it is determined by its values on a countable set with an accumulation point, e.g., by its values at all sufficiently large integers. We will formulate all results using the spectrum, rather than the spectral zeta function. For odd-dimensional manifolds, these give exactly the same information, as the following proposition shows.

\begin{proposition} \label{oddeq}  If $M$ is an odd-dimensional manifold, the
  multiset $\Sp_M(A)$ and the function $\zeta_{M,A}(s)$ mutually
  determine each other.
\end{proposition}

\begin{proof} It is clear that the function $\zeta_A(s)$ determines
  $\Sp(A)-\{0\}$, so we only need to show that if $M$ is of odd dimension, the multiplicity of zero in
  the spectrum is also determined by $\zeta_A$; this multiplicity is
  $\dim \ker A$, which equals $-\zeta_A(0)$ if $M$ has odd
  dimension (see \cite[Thm. 5.2]{Rosenberg}).
\end{proof} 

\begin{remark} \label{diracex}  The result in Proposition \ref{oddeq} does not hold in general if $M$ is of even dimension $n$, as the following example shows. Suppose $M$ is an even dimensional spin manifold with Dirac operator  $D = \left(\begin{smallmatrix} 0 & D^+ \\ D^- & 0 \end{smallmatrix} \right)$, where the spinor bundle is decomposed into eigenspaces for the chirality operator as $S^+ \oplus S^-$, with  $D^{+} \colon S^{+} \rightarrow S^{-}$ and $D^- \colon  S^{-} \rightarrow S^{+},$ $D^-=(D^+)^*$ adjoint to $D^+$. Then the second order operators $\Delta^{\pm} := D^{\mp} D^{\pm}$ have the same non-zero spectrum (this is true in general for the non-zero spectrum of the products $AB$ and $BA$ of two operators $A$ and $B$). 
On the other hand, $\ker \Delta^+ = \ker (D^+)^* D^+ = \ker D^+ $ since if $\Delta ^+\varphi = (D^+)^* D^+ \varphi =0$, then 
$|| D^+ \varphi||^2 = \langle D^+ \varphi, D^+ \varphi \rangle = \langle \varphi, (D^+)^* D^+ \varphi \rangle = 0$. Hence with  $m^{\pm}$ the multiplicity of $0$ in the spectrum of $\Delta^{\pm}$, we find that 
\begin{align*} m^+ - m^- & =
\dim \ker \Delta^+ - \dim \ker \Delta^- = 
 \dim \ker D^+ - \dim \ker D^- \\ &= \mathrm{index}\ D = \int_M \hat A(M)
 \end{align*} 
is, by the Atiyah--Singer index theorem, the $\hat A$-genus of $M$, which may be non-zero (compare \cite[3.4, 4.1]{BGV}). For example, if $M$ is a complex quartic surface (of real dimension $4$) in $\mathbf{CP}^3$, then $\mathrm{index}\ D = \int \hat A(M)= 2$ (see, e.g., \cite[p.\ 727]{Gilkey}). 
\end{remark}

Interestingly, our main results, such as Theorem \ref{main} and Proposition \ref{mainprop}, are formulated in their strongest possible form using \emph{precisely} the multiplicity of zero in the spectrum of certain twisted Laplacians. For these operators, it turns out that also in even dimension this multiplicity (and hence the zeta function) is fixed by the non-zero spectrum of the usual and the twisted Laplacian, cf.\ Proposition \ref{fixm0} below.

\subsection{Group representations} If $G$ is a finite group, let
$\check{G}=\Hom(G,\C^*)$ denote the group of linear characters of $G$,
and let $\Irr(G)$ denote the set of inequivalent irreducible unitary
representations of $G$. We consider complex representations as $\C[G]$-modules
$\mathcal{M}$ or group homomorphisms
$\rho \colon G \rightarrow \Aut(V) \cong \mathrm{GL}_N(\C)$ with $V = \C^N$ and freely mix these
concepts, writing expressions such as
``$\mathcal{M} \cong \rho$''. By further slight abuse of
notation, if $\rho_1$ and $\rho_2$ are representations of $G$, we
write
$$\langle \rho_1, \rho_2 \rangle = \langle \tr(\rho_1(-)),
\tr(\rho_2(-) \rangle = \frac{1}{|G|} \sum_{g \in G}
\tr(\rho_1(g))\overline{\tr(\rho_2(g))}$$ for the inner product of the
corresponding characters in the space of class functions. The
multiplicity of an irreducible representation $\rho' \in \Irr(G)$ in
the decomposition into irreducibles of a general representation $\rho$
of $G$ is then $\langle \rho, \rho' \rangle$.

The \emph{regular representation} $\rho_{G,\mathrm{reg}}$ corresponds to the $\C[G]$\hyp{}module $\C[G]$. It decomposes as $$\rho_{G,\mathrm{reg}} = \bigoplus\limits_{\rho_i \in \Irr(G)} \dim(\rho_i)  \rho_i.$$ 

If $H$ is a subgroup of $G$, then $\Ind_H^G \rho $ denotes the representation induced by $\rho$ from $H$ to $G$: if $\rho$ corresponds to the $\C[H]$-module $V$, then $\Ind_H^G \rho$ corresponds to the $\C[G]$ module $W:=\C[G] \otimes_{\C[H]} V$. 
In coordinates, this means the following: since $G$ permutes the cosets of $H$ in $G$, if we choose coset representatives $G/H = \{ g_1H = H, \dots g_nH \},$ then for any $g \in G$ we have $gg_i = g_{g(i)} h_{g,i}$ for some $h_{g,i} \in H$ and some permutation $i \mapsto g(i)$ of $\{1,\dots,n\}$. If we write $W = V^{G/H} = \bigoplus_{i=1}^n g_i V$ using $g_i$ as placeholder, then with $v_i \in V$, we have 
\begin{equation} \label{eq:indHG}
\Ind_H^G \rho(g) \left( \sum_i g_i v_i \right) = \sum_i g_{g(i)} \rho(h_{g,i})(v_i). 
\end{equation}
Let $\Res_H^G \rho$ denotes the restriction of $\rho$ from $G$ to $H$. If $\rho$ is a representation of $G$, $\bar \rho$ denote the complex conjugate representation. 
Recall the standard calculation rules $$\langle \rho_1 \otimes \rho_2, \rho_3 \rangle = \langle \rho_1, \rho_3 \otimes \overline{\rho_2}  \rangle \mbox{ and } \langle \Ind_H^G \rho_1, \rho_2 \rangle = \langle \rho_1, \Res_H^G \rho_2 \rangle,$$
(the latter is known as ``Frobenius reciprocity''). 

More generally, the above theory applies \emph{mutatis mutandis}, replacing $\C$ by an algebraically closed field of characteristic coprime to the order $|G|$ of the group $G$. For a non-algebraically closed field $K$ of characteristic coprime to $|G|$, an irreducible representation might decompose over the algebraic closure into a sum of irreducible (Galois-conjugate) representations, but the above theory remains valid, with the caveat that a rational character is not always the character of a rational representation, but a multiple is. If the characteristic of $K$ divides $|G|$, the category of $K[G]$-modules is not semisimple, so complementary modules for submodules do not always exist. For us, it will be important that the regular representation is defined over $\Q$, and induction and restriction turn $K$-representations into $K$-representations. 

\subsection{$G$-sets} If $G$ is a group, a \emph{$G$-set} is a set that admits a left $G$-action. An example is the left cosets $G/H$ of a subgroup $H$ with the action of left multiplication by $G$. A \emph{morphism} of $G$-sets is a $G$-equivariant map of the sets. We say a $G$-set is \emph{transitive} if $G$ acts transitively on it. If $X$ is a transitive $G$-set and $H$ the stabiliser of any point in $X$, then $X$ is isomorphism to $G/H$ as $G$-set. 

\subsection{(Weak) conjugacy} We let ``$\one=\one_G$'' denote the trivial representation of a group $G$. If, as before, $\{g_1,\dots,g_n\}$ is a set of representatives for the (left) $H$-cosets in $G$, then $\Ind_H^G \one$ is the permutation representation (i.e., the action of each $g \in G$ is given by a permutation matrix, a matrix having exactly one non-zero entry $1$ in each row and column) given by the action of $G$ on the vector space $$\C[G/H]:=\bigoplus\limits_{i=1}^n \C g_i H$$ spanned by the cosets of $H$ in $G$. We have the following. 
\begin{proposition} \label{propweakstrong} \mbox{ } Let $H_1$ and $H_2$ denote two subgroups of a finite group $G$. \begin{enumerate} \item \label{weak-conj-prop} 
The following properties are equivalent: 
\begin{enumerate} 
\item The representations $\Ind_{H_1}^G \one \cong \Ind_{H_2}^G \one$ are isomorphic.
\item Each conjugacy class $c$ of $G$ intersects $H_1$ and $H_2$ in the same number of elements.
\item There exists a set-theoretic bijection $\psi \colon H_1 \rightarrow H_2$ such that $h_1$ and $\psi(h_1)$ are conjugate in $G$ for any $h_1 \in H_1$. 
\end{enumerate}
If any of these holds, we say $H_1$ and $H_2$ are \emph{weakly conjugate} in $G$.
\item \label{conj-prop} 
The stronger property that the groups $H_1$ and $H_2$ are \emph{conjugate} in $G$ is equivalent to the cosets $G/H_1$ and $G/H_2$ being isomorphic as $G$-sets.   
\end{enumerate}
\end{proposition} 

Weak conjugacy is also known as ``Ga{\ss}mann equivalence'' in number theory, cf.\ \cite{Perlis}. 

\begin{proof} (\ref{weak-conj-prop}) Representation isomorphism is the
  same as isomorphism of characters, and the character of the
  representation $\Ind_{H}^G \one$ is
  $$\psi(g)=|[g]\cap H| \cdot |{C_G(g)}|/{|H|},$$ where $[g]$ is the conjugacy class of $g$ and $C_G(g)$
  is the centraliser of $g$ in $G$ (compare, e.g., \cite[\S
  1]{BrooksSurvey}). The final equivalent statement is proven in
  \cite[Lemma 2]{Chen}.  (\ref{conj-prop}) The existence of a
  $G$-isomorphism $\phi: G/H_1 \to G/H_2$ implies that the
  $G$-stabiliser $H_1$ of the coset $eH_1$ equals the
  $G$-stabiliser $gH_2g^{-1}$ of some coset
  $gH_2=\phi(eH_1)$, and hence $H_1 = gH_2
  g^{-1}$. Conversely, if $H_2 = g_0^{-1}H_1g_0$, the map $\phi: G/H_1 \to G/H_2$,
      $\phi(gH_1) = g g_0 H_2$ is a well defined isomorphism of $G$-sets. 
\end{proof} 

\begin{remark} 
If we have a diagram (\ref{m0}) and $M_1$ and $M_2$ have the same Laplace spectrum (viz., the same spectral zeta function) and the same dimension $n$, then they have the same volume (from Weyl's law, or, equivalently, from the residue of their zeta functions at $s=n/2$). Since the covering degree $\deg(p_i)$ of $p_i$ is $\mathrm{vol}(M_i)/\mathrm{vol}(M_0)$, these are also equal, and hence in diagram (\ref{sunadasetup}), we find from $\#H_i = |G|/\deg(p_i)$ that $\#H_1 = \#H_2$. So in this case, there is \emph{always} a set-theoretic bijection $\psi \colon H_1 \rightarrow H_2$. 
\end{remark}

\subsection{Twisted Laplacian} \label{tl} Suppose $\rho \colon \pi_1(M) \rightarrow \Un_{{N}}(\C)$ is a unitary representation of the fundamental group $\pi_1(M)$ of $M$. Let $\Pi \colon \tilde M \rightarrow M$ denote the universal covering of $M$, and set $E_\rho:= \tilde M \times_\rho \C^{{N}}$, where the subscript $\rho$ indicates equivalence classes for the relation $(z,v) \sim (\gamma z,\rho(\gamma)^{-1} v)$ for any $z \in \tilde M, v\in \C^{{N}}$ and $\gamma \in \pi_1(M)$. Now $E_\rho$ is a flat 
vector bundle {{of rank $N$}} over $M$, whose global sections $ f \in C^\infty(M,E_\rho)$ correspond bijectively to smooth $\rho$-equivariant vector-valued functions $\vv{f} \colon \tilde M \rightarrow \C^{{N}}$, i.e., functions with $\vv{f}(\gamma z) = \rho(\gamma)\vv{f}(z)$. The \emph{twisted Laplacian} $$ \Delta_{M,\rho}=\Delta_\rho \colon C^\infty(M,E_\rho) \rightarrow C^\infty(M,E_\rho)$$ is defined as $$\vv{\Delta_\rho(f)}:=\Delta_{\tilde M} \vv{f}(z).$$ 
{{Henceforth, we will also denote the spectrum $\Sp_M(\Delta_\rho)$ simply by $\Sp_M(\rho)$
or by $\Sp(\rho)$ if the underlying manifold $M$ is fixed.}}

Notice that when $\rho = \rho_1 \oplus \rho_2$, then $\Delta_\rho$ admits a block decomposition $\Delta_{\rho_1} \oplus \Delta_{\rho_2}$ on $E_\rho \cong E_{\rho_1} \oplus E_{\rho_2}$, and thus, the spectrum satisfies (as multisets) $$\Sp_M(\rho_1 \oplus \rho_2) = \Sp_M(\rho_1) \cup \Sp_M(\rho_2).$$

\subsection{Twisted Laplacians on finite covers} \label{tlf}  In case $\rho$ factors through a finite group $G$, there is no need to use the universal covering. Let $M' \to M$ denote a (fixed-point free) $G$-cover and $\rho: G \to U_{{N}}(\C)$ a unitary representation. The vector space $C^\infty(M,E_\rho)$ is canonically isomorphic to the vector space of smooth  $\rho$-equivariant vector-valued functions on $M'$ given by
$$ C^\infty_\rho(M',\C^{{N}}) := \{ \vv{f} \in C^\infty(M',\C^{{N}}) \mid \vv{f}(\gamma x) = \rho(\gamma) \vv{f}(x), \forall\, x \in M', \gamma \in G \}. $$
In this case, \begin{equation} \label{defdeltarho} \vv{\Delta_\rho f} = \Delta_{M'} \vv{f}, \end{equation}  where $\vv{f}$ is the $\rho$-equivariant function in $C^\infty_\rho(M',\C^{{N}})$ corresponding
to $f$. Note that $\Delta_{M'} \vv{f}$ is again a $\rho$-equivariant function in $C^\infty_\rho(M',\C^{{N}})$ and therefore represents an element in  $C^\infty(M,E_\rho)$.

The following lemma is stated in  \cite[Lemma 1]{SunadaNT}; we write a proof using our notation. 

\begin{lemma}  \label{SNT} If $M \rightarrow M_1 \rightarrow M_0$ is a tower of finite Riemannian coverings and $M \rightarrow M_0$ is Galois with group $G$, $M \rightarrow M_1$ with group $H$, and $\rho \colon H \rightarrow \Un_{{N}}(\C)$ a representation, then $$ \Sp_{M_0}(\Ind_{H}^G \rho) = \Sp_{M_1}(\rho). $$ 
\end{lemma}

\begin{proof} Write $\rho^*:=\Ind_H^G \rho \colon G \rightarrow \Un_{{Nn}}(\C)$ and let $$G/H=\{g_1H=H,\dots,g_{{n}}H\}$$ denote representatives for the distinct cosets of $H$ in $G$. 
Define two maps 
$$ 
  \xymatrix{ C_\rho^\infty(M,\C^{{N}})  \ar@/^1.5pc/[r]^\Phi  &    C_{\rho^*}^\infty(M,\C^{{Nn}}) \ar@/^1.5pc/[l] ^\Psi   } 
 $$
by $$\Phi(\vv{f})(x) =(\vv{f}(g_1^{-1}(x)),\dots, \vv{f}(g_{{n}}^{-1}(x)))$$ and $$\Psi(\vv{F})=\Psi((\vv{f_1},\dots,\vv{f_{{n}}})): = \vv{f_1}.$$ 
Recall that by definition $$\rho^*(g)((\vv{f_i}(x))_{i=1}^{{n}}) =(\rho(g_{g(i)}^{-1} g g_i) \vv{f}_{g(i)}(x) )_{i=1}^{{n}}$$ where $g(i)$ is given by $$g g_i H = g_{g(i)} H.$$ This allows one to check that $\Phi$ and $\Psi$ are well-defined and mutually inverse bijections. 
Recall that $\vv{\Delta_\rho(f)}=\Delta_M(\vv{f})$ and $\vv{\Delta_{\rho^*}(F)}=\Delta_M(\vv{F})$ with $\Delta_M$ applied componentwise. 
Since the $g_i$ are isometries, $\Phi$ is a unitary operator in $L^2$ and $\Delta_M \Phi = \Phi \Delta_M$, so that  we have the intertwining 
$$ \Phi \circ \Delta_\rho = \Delta_{\rho^*} \circ \Phi, $$
and the equality of spectra follows. 
\end{proof}

Decomposing a general representation $\rho: G \to U_{{N}}(\C)$ into irreducibles as $\rho = \bigoplus \langle \rho_i,\rho \rangle \rho_i$, we have
\begin{equation} \label{cup1} \Sp_M(\Delta_{\rho}) = \bigcup \langle \rho_i,\rho \rangle \Sp_M(\Delta_{\rho_i}). \end{equation}
Applied to the regular representation, we find a relation between the spectra of the usual Laplacian on the cover $M'$ and of the twisted Laplacians on the original manifold $M$, as follows: 
\begin{equation} \label{cup2} \Sp_{M'}(\Delta_{M'}) =\Sp_{M}(\Delta_{\rho_{G,\mathrm{reg}}}) = \bigcup \dim(\rho_i) \Sp_M(\Delta_{\rho_i}). \end{equation} 
(The first equality follows from Lemma \ref{SNT} since $\Ind_{\{1\}}^G \one = \C[G] =  \rho_{G,\mathrm{reg}}$). We see in particular that the eigenvalues of any twisted Laplacian are also eigenvalues of the usual Laplace operator of the corresponding cover. 

\begin{remark} \label{overlap} It appears to be an interesting problem how \emph{disjoint} the spectra in (\ref{cup2}) are, similar to the question how disjoint zeros of number theoretic $L$-series are (cf.\ \cite{RubinsteinSarnak}): the so-called \emph{grand simplicity hypothesis} says that the imaginary parts of the zeros of all Dirichlet $L$-series for primitive characters are linearly independent over $\Q$. From the above decomposition results, it is clear that if $\rho'$ is a subrepresentation of $\rho$, then $\Sp(\rho') \subseteq \Sp(\rho)$ (this even holds for infinite amenable groups if $\rho'$ is weakly contained in $\rho$, cf.\ \cite{Sunadaweak}); here, we are asking for a kind of converse result. 
\end{remark}

\begin{lemma} \label{mult} Let $G$ be a finite group acting by fixed-point free isometries on a closed connected Riemannian manifold $M'$ with quotient $M=G \backslash M'$.
If $\rho$ is any unitary representation of $G$, then the multiplicity $\langle \rho, \one \rangle$ of the trivial representation in the decomposition of $\rho$ into irreducibles equals  $\dim \ker \Delta_\rho$, the multiplicity of the zero eigenvalue in {{$\Sp_M(\rho)$}}. 
\end{lemma} 

\begin{proof}[First proof of Lemma \ref{mult}]
  Since $M$ and $M'$ are connected, the multiplicity of zero in $\Sp_{M'}(\Delta_{M'})$ and $\Sp_M(\Delta_M)$ is one. It follows from the decomposition of multisets \eqref{cup2} that for any irreducible representation $\rho_i \neq \one$ of $G$, $\Sp_M(\Delta_{\rho_i})$ does not contain zero. If we now decompose $
  \rho$ as a sum of irreducibles, the decomposition of multisets \eqref{cup1} implies that the multiplicity of zero in $\Sp_M(\Delta_\rho)$ is indeed the multiplicity with which $\one$ occurs in $\rho$. 
   \end{proof} 

We can also give a ``direct'' proof, as follows. 

\begin{proof}[Second proof of Lemma \ref{mult}]
  {{Let $\rho: G \to U_N(\C)$.}} A function
  $f \in \ker \Delta_\rho \in C^\infty(M,E_\rho)$ corresponds to a
  function $\vv{f}$ on $M'$ with $\Delta_{M'} \vv{f} = 0$ and
  $\vv{f}(\gamma z) = \rho(\gamma) \vv{f}(z)$ {{for all
      $\gamma \in \Gamma$}}. Since $M'$ is closed {{and}}
  connected, this implies that $\vv{f}=\vv{f}_0$
  is a constant vector {{in $\C^N$}}, and the equivariance condition translates into
  $$
  (\rho(\gamma)-1)\vv{f}_0 = 0 {{\quad \text{for all $\gamma \in G$}}}.
  $$
  Hence each such linearly independent vector
  $\vv{f_0} \in \C^N$ can be used to split off a
  one-dimensional invariant subspace in $\rho$, and we find the
  result.
\end{proof} 

In Remark \ref{diracex} we showed that, in general, on an even-dimensional manifold, knowledge of the spectrum is stronger than that of the spectral zeta function, i.e., of the non-zero spectrum. For our twisted Laplace operators, the situation is better, as the following proposition shows. 

\begin{proposition} \label{fixm0} 
Let $G$ be a finite group acting by fixed-point free isometries on a closed connected $n$-dimensional Riemannian manifold $M'$ with quotient $M=G \backslash M'$.
If $\rho$ is any unitary representation of $G$, then on the one hand the pair of multisets $\Sp_M(\Delta_\rho)$ and $\Sp_M(\Delta)$, and on the other hand the pair of zeta functions $\zeta_{M,\Delta_\rho}(s)$ and $\zeta_{M,\Delta}(s)$ mutually
  determine each other.
\end{proposition} 

\begin{proof} It suffices to prove that $\zeta_\Delta$ and $\zeta_{\Delta_\rho}$ determine the multiplicity of zero in the spectrum (i.e, the dimension of the kernel) of $\Delta_\rho$. 
If the dimension $n$ of $M$ is odd, this follows from the stronger Proposition \ref{oddeq}. For even $n$ and a general operator $A$ as in Subsection \ref{specz}, \begin{equation} \label{evenA} \zeta_A(0) = - \dim \ker A + (4 \pi)^{-n/2} \int_M u_{n/2}(A), \end{equation} where $ \tr(e^{-tA}) \sim (4 \pi t)^{-n/2} \sum_{k=0}^\infty \int_M u_k(A)  t^k$ is the asymptotic expansion of the heat kernel of $A$ as $t{\downarrow}0$ \cite[Thm. 5.2]{Rosenberg}. 

We apply this in our situation, with $A=\Delta_\rho = \bigoplus\limits_{i=1}^N \Delta_{M'}$ acting on $C^\infty(M,E_\rho) = C_\rho^\infty(M',\C^N)$. Denote by $1_N$ the identity matrix of size $N \times N$. Recall that the principal symbol $p(\Delta_M)$ of a Laplace operator $\Delta_M$ on a Riemannian manifold $(M,g)$ is determined by the metric tensor $g$ (more accurately, it is the quadratic form on the cotangent bundle dual to $g$). 
Since $M' \rightarrow M$ is a Riemannian covering, the metric tensor of $M$ pulls back to that of $M'$, and hence the principal symbol of $\Delta_{M'}$ is the same as that of $\Delta_M$. Therefore, $\Delta_\rho$ is a ``Laplace-style'' operator in the sense of \cite[1.2]{GilkeyAF}: it has (matrix) principal symbol the diagonal matrix $p(\Delta_\rho)=p(\Delta_M)\cdot 1_N$. Such operators have an invariant representation depending on a connection on the bundle and an endomorphism of the bundle as in \cite[Lemma 1.2.1]{Gilkey}, and in our situation, for  $E_\rho$, the bundle connection is flat (curvature $\Omega \equiv 0$) and the endomorphism $e$ is trivial.  

The coefficients $u_k(\Delta_\rho)(x)$ (as a function of $x \in M$) are of the form $\tr_{E_{\rho,x}}(e_k(\Delta_\rho)(x))$, where  $\tr_{E_{\rho,x}}$ denotes the fiberwise trace in the fibers $E_{\rho,x} \cong \C^N$, and where $e_k(\Delta_\rho)$ is a linear combination with universal coefficients (independent of the dimension $n$ of the manifold and the rank $N$ of the bundle) of covariant derivatives of $\underline{R}\cdot 1_N$ (where $\underline{R}$ is the covariant Riemann curvature tensor of $M$), the bundle curvature $\Omega$ and the endomorphism $e$ \cite[\S 3.1.8-3.1.9]{Gilkey}. Since the latter two are identically zero in our situation, we can write $e_k(\Delta_\rho) = P_k \cdot 1_N$ with $P_k$ only depending on the covariant derivatives of $\underline{R}$, in particular, not depending on $\rho$. We conclude that 
$$ (4 \pi)^{-n/2} \int_M u_{n/2} (\Delta_\rho)(x) = (4 \pi)^{-n/2} \int_M \tr_{E_{\rho,x}}(P_{n/2}(x) \cdot 1_N) = N U, $$
where $U = (4 \pi)^{-n/2} \int_M P_{n/2} $ is independent of $\rho$. 

Therefore, applying \eqref{evenA} to $\Delta$ and $\Delta_\rho$, we find 
$$ \dim \ker \Delta_\rho = N U -  \zeta_{\Delta_\rho} (0) =  N(\zeta_{\Delta} (0)+1) -  \zeta_{\Delta_\rho} (0). $$
We can compute the rank $N$ in terms of the first coefficients in the asymptotic expansions: using $e_0(\Delta_\rho) = 
1_N$, we find  $N = \int_M u_{0}(\Delta_\rho)/\int_M u_{0}(\Delta)$. On the other hand, $\zeta_{\Delta_\rho}(s)$ has a simple pole at $s=n/2$ with residue $\Gamma(n/2)^{-1} \int_M u_0(\Delta_\rho)$ (see, e.g., the proof of  \cite[Thm. 5.2]{Rosenberg}),  so that the function $\zeta_{\Delta_\rho}(s)/\zeta_{\Delta}(s)$ is holomorphic at $s=n/2$ and takes value $N$ there:
\begin{equation} \label{preweyl} 
N 
=  \left. \frac{\zeta_{\Delta_\rho}(s)}{\zeta_{\Delta}(s)}\right \rvert_{s=\frac{n}{2}}. 
\end{equation} 
 Hence 
$$ \dim \ker \Delta_\rho = (\zeta_{\Delta} (0)+1) \left. \frac{\zeta_{\Delta_\rho}(s)}{\zeta_{\Delta}(s)}\right \rvert_{s=\frac{n}{2}} -  \zeta_{\Delta_\rho} (0). $$
 This is the desired expression for the multiplicity of zero in the spectrum in terms of spectral zeta functions only. 
\end{proof}

\begin{remark} \mbox{ } 
\begin{enumerate}
\item The above argument also shows that for a twisted Laplacian $\Delta_\rho$ corresponding to a unitary representation on a fixed manifold $M$, the value $\dim \ker \Delta_\rho + \zeta_{\Delta_\rho}(0)$ only depends on the dimension of the representation $\rho$. 
\item Weyl's law for $\Delta_\rho$ says that if $\mathrm N(\Delta_\rho,X)$ denotes its number of eigenvalues $\leq  X$, then $$\lim_{X \rightarrow +\infty} \frac{\mathrm N(\Delta_\rho,X)}{X^{n/2}}  = N \cdot 
\frac{\mathrm{vol}(M)}{(4 \pi)^{n/2} \Gamma\left(\frac{n}{2}+1\right)},$$ 
so that on a fixed manifold, the dimension $N$ of the representation $\rho$ can be read off from the asymptotics of the spectra of $\Delta_\rho$ and $\Delta_M$: 
\begin{equation} \label{weyl} N= \lim_{X \rightarrow +\infty} \frac{\mathrm N(\Delta_\rho,X)}{\mathrm N(\Delta_M,X)};\end{equation} formulas (\ref{preweyl}) and (\ref{weyl}) are equivalent through Karamata's version of the Tauberian theorem (compare \cite[pp.~91--92]{BGV}). 
\end{enumerate}
\end{remark}

\section{Detecting representation isomorphism through twisted spectra} \label{detecting} 

The main result of this section is Proposition \ref{solo}, which is presently used to give a spectral characterisation of weak conjugacy, and will be used again later in the proof of the main theorem. 

\subsection{Spectral detection of isomorphism of induced representations} 
In this subsection, we assume that we have a diagram \eqref{sunadasetup} of finite coverings. We start with a proposition that allows us to detect isomorphism of representations induced from linear characters purely from spectral data. 

\begin{proposition} \label{solo} 
For two linear characters $\chi_1 \in \check{H}_1$ and $\chi_2 \in \check{H}_2$, the following are equivalent: 
\begin{enumerate}
\item \label{one} $\Ind_{H_1}^G \chi_1 \cong \Ind_{H_2}^G \chi_2$;
\item \label{two} The spectrum $\Sp_{M_i}(\bar \chi_i \otimes \Res_{H_i}^G \Ind_{H_j}^G \chi_j)$ is independent of $i,j=1,2$. 
\item[\textup{(\ref{two}')}] \label{twores} Condition \textup{(\ref{two})} holds for the pairs $(i,j)$ given by $(1,1),(2,1)$ and $(1,2),(2,2)$.
\item \label{three} The multiplicity of the zero eigenvalue in $\Sp_{M_i}(\bar \chi_i \otimes \Res_{H_i}^G \Ind_{H_j}^G \chi_j)$ is independent of $i,j=1,2$. 
\item[\textup{(\ref{three}')}] \label{threeres} Condition \textup{(\ref{three})} holds for the pairs $(i,j)$ given by $(1,1),(2,1)$ and $(1,2),(2,2)$.
\end{enumerate}
\end{proposition}

\begin{remark} \label{twoeq} Condition (\ref{two}')  is 
$$ \def\arraystretch{1.5} \left\{ \begin{array}{l} \Sp_{M_1}(\bar \chi_1 \otimes \Res_{H_1}^G \Ind_{H_1}^G \chi_1) = \Sp_{M_2}(\bar \chi_2 \otimes \Res_{H_2}^G \Ind_{H_1}^G \chi_1); \\
\Sp_{M_1}(\bar \chi_1 \otimes \Res_{H_1}^G \Ind_{H_2}^G \chi_2) = \Sp_{M_2}(\bar \chi_2 \otimes \Res_{H_2}^G \Ind_{H_2}^G \chi_2). \end{array} \right. $$
In this form, the statement of the proposition is similar to a number-theoretical result  of Solomatin \cite{Solomatin} that inspired our proof below.
\end{remark}

\begin{proof}[Proof of Proposition \ref{solo}] 

We start by proving that (\ref{one}) implies (\ref{two}). Let $\rho$ denote any irreducible representation of $G$; then for any $i=1,2$, we have 
\begin{align*} \langle \Ind_{H_i}^G \left(\bar \chi_i \otimes \Res_{H_i}^G \Ind_{H_i}^G \chi_i\right), \rho \rangle &= \langle \bar \chi_i \otimes \Res_{H_i}^G \Ind_{H_i}^G \chi_i, \Res_{H_i}^G \rho \rangle 
\\ 
& = \langle \bar \chi_i , \Res_{H_i}^G \rho \otimes  \overline{\Res_{H_i}^G\Ind_{H_i}^G \chi_i} \rangle \\ 
& = \langle \bar \chi_i , \Res_{H_i}^G \left (\rho \otimes \overline{\Ind_{H_i}^G \chi_i}\right) \rangle \\ 
& = \langle \overline{{\Ind_{H_i}^G \chi_i}}, \rho \otimes \overline{\Ind_{H_i}^G \chi_i} \rangle. 
\end{align*}
By assumption (\ref{one}), this final expression is independent of $i=1,2$, and hence the same holds for the initial expression. Since this holds for any $\rho$, we find that 
$$  \Ind_{H_1}^G \left(\bar \chi_1 \otimes \Res_{H_1}^G \Ind_{H_1}^G \chi_1\right) = \Ind_{H_2}^G \left(\bar \chi_2 \otimes \Res_{H_2}^G \Ind_{H_2}^G \chi_2\right). $$
By Lemma \ref{SNT}, we find 
\begin{align}  \Sp_{M_1}(\bar \chi_1 \otimes \Res_{H_1}^G \Ind_{H_1}^G \chi_1)  & = \Sp_{M_2} (\bar \chi_2 \otimes \Res_{H_2}^G \Ind_{H_2}^G \chi_2) \nonumber \\  
& =  \Sp_{M_2} (\bar \chi_2 \otimes \Res_{H_2}^G \Ind_{H_1}^G \chi_1), \label{effe} 
\end{align}
the last line again by assumption (\ref{one}). This is condition (\ref{two}) for $(i,j)=(1,1)$ and $(i,j)=(2,1)$.   Using assumption (\ref{one}), one may replace $\Ind_{H_1}^G \chi_1$ by $\Ind_{H_2}^G \chi_2$ in formula (\ref{effe}) on one or both sides, and this shows condition (\ref{two}) for all other choices of $i,j$.  

Passing from stronger to weaker statements, (\ref{two}) implies (\ref{two}') and (\ref{three}), and (\ref{three}), as well as (\ref{two}'), imply (\ref{three}'). 
Hence we only need to prove that (\ref{three}') implies (\ref{one}). Consider, for different $i,j$, 
\begin{equation} \label{aijsym} a_{i,j}:= \langle \bar \chi_i \otimes \Res_{H_i}^G \Ind_{H_j}^G \chi_j, \one \rangle =  \langle \Ind_{H_j}^G \chi_j,  \Ind_{H_i}^G \chi_i \rangle,  \end{equation} 
where the last equality follows by the usual calculation rules. 
Setting $\psi$ to be the class function $\psi:=  \Ind_{H_1}^G \chi_1 - \Ind_{H_2}^G \chi_2$, this allows us to compute that 
$$ \langle \psi, \psi \rangle = a_{1,1}+a_{2,2} - a_{1,2} - a_{2,1},$$ 
and since in (\ref{three}') we are assuming $a_{1,1} = a_{2,1}$ and $a_{1,2}=a_{2,2}$, it follows that $\langle \psi, \psi \rangle = 0$, so $\psi=0$, which is condition (\ref{one}). 
\end{proof} 

\subsection{Strong isospectrality and spectral detection of weak conjugacy} 

Isospectrality of manifolds $M_1$ and $M_2$ in a diagram of the form (\ref{sunadasetup}) does not in general imply that $H_1$ and $H_2$ are weakly conjugate. 
\begin{example} \label{lens} Consider the situation where $M=S^5$, and $M_1=L(11;1,2,3)$ and $M_2=L(11;1,2,4)$ are lens spaces  $L(q;s_1,s_2,s_3)$ defined as the quotient of $S^5$ by the block diagonal $6 \times 6$ matrix given by three $2 \times 2$ blocks representing planar rotations over respective angles $2\pi s_i/q$. In this case, the two groups $H_i \cong {\Z}/11{\Z}$ are not equal as subgroups of the isometry group of $S^5$; they commute, and we can set $M_0 = (H_1 \times H_2) \backslash M$. Ikeda has shown that  $M_1$ and $M_2$ are isospectral for the Laplace operator on functions (see \cite[p. 313]{Ikeda1980}, observing that since $8 = -3 \mbox{ mod } 11$, by \cite[Thm.\ 2.1]{Ikeda1980} $M_1$ is isometric to $L(11;1,2,8)$, where the latter parameters are the ones used by Ikeda). However, $H_1$ and $H_2$ are not weakly conjugate: since $G$ is abelian, conjugacy classes are singletons and $|\{g\} \cap H_i|$ is $0$ or $1$ depending on whether $g$ belongs to $H_i$ or not. 
\end{example} 

Sunada \cite[Lemma 1]{Sunada} proved that weak conjugacy of $H_1$ and $H_2$ implies \emph{strong isospectrality}, defined as follows in the sense of Pesce \cite[\S II]{PesceJFA}.  

\begin{definition} \label{strongiso} Two Riemannian manifolds $M_1$ and $M_2$ admitting a common cover $M$ are called \emph{strongly isospectral} if
the spectra of $q_{i\ast} A$ acting on $L^2(M_i, q_{i\ast} E) \cong L^2(M,E)^{H_i}$ are equal for any natural operator $A$ on $M$. Here, $q_i \colon M \twoheadrightarrow M_i$ are the corresponding covering maps, and an operator $A$ as above on $M$  is \emph{natural} if $G$ acts isometrically on the fibers of the bundle $E$ and $A$ commutes with the action of $G$. 
\end{definition} 

The Laplace operators acting on $p$-forms are natural for all $p$, so if $H_1$ and $H_2$ are weakly conjugate, then the spectra of all of these are equal (sometimes, ``strong isospectrality'' is used to mean that precisely these operators are isospectral, but we will follow Pesce's definition as above). 

\begin{example} The lens spaces in Example \ref{lens} are isospectral for the Laplacian on functions, but not on all $p$-forms, cf.\ \cite[Remark 3.8]{LMR}. 
\end{example}

The next Proposition \ref{mainpropdet} provides a spectral criterion that is equivalent to weak conjugacy, and is an immediate corollary of Proposition \ref{solo}. It is analogous to a  number theoretical result of Nagata \cite{Nagata}.  

\begin{proposition} \label{mainpropdet} 
Suppose $M$ is a connected smooth closed Riemannian manifold, $G$ a finite group of isometries of $M$ and $H_1$ and $H_2$ are two subgroups of fixed-point free isometries in $G$ with associated quotient manifolds $M_1:=H_1 \backslash M$ and $M_2:=H_2 \backslash M$. Then $H_1$ and $H_2$ are weakly conjugate if and only if the multiplicity of the zero eigenvalue in $\Sp_{M_i}(\Res_{H_i}^G \Ind_{H_j}^G \one)$ is independent of $i,j=1,2$. 
 \end{proposition}

\begin{proof}
The groups $H_1$ and $H_2$ are weakly conjugate precisely when there is an isomorphism of permutation representations $\Ind_{H_1}^G \one \cong \Ind_{H_2}^G \one$. 
By Proposition \ref{solo}, this is equivalent to the claim. 
\end{proof} 

One may vary the condition of twisted isospectrality using the equivalent conditions in Proposition \ref{solo}. For example, using Remark \ref{twoeq}, we deduce the following (adding some redundant information). 

\begin{corollary} \label{strong} 
If a diagram as in \textup{(\ref{sunadasetup})} is given, and the following twisted spectra agree: 
$$ \def\arraystretch{1.5} \left\{ \begin{array}{l} \Sp_{M_1}(\Res_{H_1}^G \Ind_{H_1}^G \one) = \Sp_{M_2}(\Res_{H_2}^G \Ind_{H_1}^G \one); \\
\Sp_{M_1}(\Res_{H_1}^G \Ind_{H_2}^G \one) = \Sp_{M_2}(\Res_{H_2}^G \Ind_{H_2}^G \one), \end{array} \right. $$
then the manifolds $M_1$ and $M_2$ are strongly isospectral. \qed
\end{corollary} 

\begin{remark} \label{mackey} 
Mackey's theorem describes how, for two subgroups $K_1$ and $K_2$ of a group $G$, a representation of the form $\Res_{K_2}^G \Ind_{K_1}^G \rho$ splits into irreducibles  (see, e.g., \cite[Prop.\ 22]{SerreRep}). In the situation of Proposition \ref{mainpropdet}, with $K_i \in \{H_1,H_2\}$, we find that $\Res_{K_2}^G \Ind_{K_1}^G \one$ splits as the direct sum of the permutation representations corresponding to the action of $K_2$ on the cosets of $s K_1 s^{-1} \cap K_2$ for $s \in K_2 \backslash G/ K_1$, and the occurring spectra are the (multiset-)union of the spectra corresponding to these representations; for example, 
$$ \Sp_{M_i} (\Res_{H_i}^G \Ind_{H_i}^G \one) = \bigcup_{s \in H_i \backslash G / H_i} \Sp_{M_i}( \Ind_{sH_is^{-1}\cap H_i}^{H_i} \one)$$
contains the usual Laplace spectra $\Sp_{M_i}(\Delta_{M_i})$ (setting $s$ to be the trivial double coset).  
\end{remark} 

\section{Representations with a unique monomial structure}
\label{sec:uniqmonstruc}

In this section, we recall a purely representation theoretical result from \cite{Lseries}. Since in that reference, the result was not cleanly separated from number theoretical results, we provide some details. 

\subsection{Monomial structures} 

\begin{definition} 
Suppose $\rho \colon G \rightarrow \Aut(V)$ is a representation, and $$V=\bigoplus\limits_{x \in \Omega} \mathcal L_x$$ is a decomposition of $V$ into one-dimensional spaces (``lines'') $\mathcal L_x$ for some index set $\Omega$. If the action of $G$ on $V$ permutes the lines $\mathcal L_x$, we say that the $G$-set $$L=\{ \mathcal L_x \colon x \in \Omega \}$$ is a  \emph{monomial structure} on $\rho$. 

Equivalently, in a basis having precisely one element from each line $\mathcal L_x$, the action of any $g \in G$ is given by a matrix having exactly one non-zero entry in each row and column. Note that, contrary to the case of permutation matrices, the non-zero entry in the matrix need not be $1$. 

An \emph{isomorphism of monomial structures} $L$ and $L'$ on two representation of the same group $G$ is an isomorphism of $L$ and $L'$ as $G$-sets. 
\end{definition} 

\begin{example} \label{exgset} An induced representation $\Ind_H^G \chi$ of a linear character $\chi \in \check{H}$ admits (by definition) a monomial structure where $\Omega=\{g_1,\dots,g_n\}$ is such that $g_i H$ are the different cosets of $H$ in $G$, and $\mathcal L_x = \C \cdot xH$. The corresponding matrices have as non-zero entries $n$-th roots of unity if $\chi$ is a character of order $n$. We call this monomial structure the \emph{standard monomial structure} on $\Ind_H^G \chi$. This standard monomial structure is isomorphic to $G/H$ as $G$-set. 
\end{example} 

\begin{definition} A linear character $\Xi$ on a subgroup $H$ of a group $G$ is called \emph{$G$-solitary} if $\Ind_H^G \Xi$ has a unique monomial structure up to isomorphism. 
\end{definition} 

\begin{lemma} \label{lem:solitarycharconjgroups}
Let $G$ denote a group with two subgroups $H_1$ and $H_2$, and suppose $\Xi \in \check{H_1}$ is a $G$-solitary linear character. There exists a linear character $\chi \in \check{H_2}$ for which there is an isomorphism of representations $\Ind_{H_1}^G \Xi \cong \Ind_{H_2}^G \chi$ if and only if $H_1$ and $H_2$ are conjugate subgroups of $G$.
\end{lemma} 

\begin{proof} 
In this situation, $\Ind_{H_2}^G \chi$ carries two monomial structures: the standard one and the one induced from the standard one on $\Ind_{H_1}^G \Xi$ through the isomorphism of representations. Hence these monomial structures have to be isomorphic. But as $G$-sets, they are $G/H_1$ and $G/H_2$, respectively (see Example \ref{exgset}). By Proposition \ref{propweakstrong}(\ref{conj-prop}), this means precisely that $H_1$ and $H_2$ are conjugate in $G$. 
\end{proof} 

\subsection{Wreath product construction} 

\begin{definition} \label{wreath} Let $G$ denote a finite group and $H$ a subgroup of index $n:=[G:H]$ with cosets  $$\{g_1H=H,g_2H,\dots,g_nH\}$$ of cardinality $n$. For a prime number $\ell$, let $C={\Z}/{\ell}{\Z}$ denote the cyclic group with $\ell$ elements, and let $$\tilde G := C^n \rtimes G$$ denote the \emph{wreath product}; this is by definition the semidirect product where $G$ acts on the $n$ copies of $C$ by permuting the coordinates in the same way as $G$ permutes the cosets $g_iH$. In coordinates, this means that if we let $e_1,\dots,e_n$ denote the standard basis vectors of $C^n$, and, as before, define the permutation $i \mapsto g(i)$ of $\{1,\dots,n\}$ by $g g_i H = g_{g(i)}H$, then the semidirect product is
defined by the action
\begin{equation} \label{eq:Phidef}
G \stackrel{\Phi}{\rightarrow} \Aut(C^n) \colon g \mapsto \Phi(g) = \left[
\sum_{j=1}^n k_j e_j \mapsto \sum_{j=1}^n k_j e_{g(j)} \right]
\end{equation}
where $k_j \in {\Z}/{\ell}{\Z}$. This is the (left) action of $g \in G$ on $C^n$ given by $$C^n \ni (k_1,\dots,k_n) \mapsto (k_{g^{-1}(1)},\dots,k_{g^{-1}(n)}) \in C^n.$$ 
Define $$\tilde H:= C^n \rtimes H$$ to be the subgroup of $\tilde G$ corresponding to $H$.
The cosets of $\tilde H$ in $\tilde G$ are of the form
$$ \{ \tilde g_1 \tilde H = \tilde H, \tilde g_2 \tilde H, \dots, \tilde g_n \tilde H \}, $$
where for $g_i \in G$, we have a corresponding element $\tilde g_i:=(0,g_i) \in \tilde G$.
\end{definition} 

\begin{remark} \label{indfl} 
Recall that $\Ind_{H}^G \one$ is the $\Z[G]$-module corresponding to the permutation representation of $G$ acting on the $G$-cosets of $H$. Thus, if we identify $C$ with the additive group of the finite field $\F_\ell$, the action of $G$ on $C^n \cong \F_\ell^n$ corresponds to the $\F_\ell[G]$-module $(\Ind_G^H \one) \otimes_{\Z} \F_\ell$. 
\end{remark}

\begin{proposition}[Bart de Smit, {\cite[\S 10]{Lseries}}] \label{solitary} 
For all $\ell \geq 3$, there exists a $\tilde G$-solitary character of order $\ell$ on $\tilde H$. 
\end{proposition}

\begin{proof} 
Define $\Xi$ by 
\begin{equation} \label{defxi} \Xi \colon \tilde H \rightarrow \C^*  \colon (k_1,\dots,k_n,g) \mapsto e^{2 \pi i k_1/\ell}. \end{equation} 
Let $L =\{ \mathcal L_x\}$ and $L' = \{ \mathcal L'_x\}$ denote two monomial structures on $\rho:=\Ind_{\tilde H}^{\tilde G} \Xi$, where $L$ is the standard one (see Example \ref{exgset}). The action of $G \leq \tilde G$ on $L$ is that of $G$ on $G/H$ and (after rearranging) the action of $C^n \leq \tilde G$ is given by $$(k_1,\dots,k_n)\cdot \mathcal L_j = e^{2 \pi i k_j/\ell} \cdot \mathcal L_j,$$
where we used the simplified notation $\mathcal L_j := \mathcal L_{g_j \tilde H}$.
The character $\psi$ of $\rho$ can be computed using as basis any set of vectors from the lines in $L$ or $L'$. From the above, $$|\psi((1,0,\dots,0))|=|e^{2 \pi i/\ell}+\underbrace{1+\dots+1}_{n-1}| > n-2, $$
where the last inequality is strict since $\ell \geq 3$. On the other hand, computing the same trace using a basis from $L'$, we get a sum of some number, say, $m$, of $\ell$-th roots of unity, where $m$ is the number of lines in $L'$ that are mapped to itself by $(1,0,\dots,0)$. If there is a line not mapped to itself (a zero diagonal entry in the corresponding matrix), then there are at least two (since every row/column has precisely two non-zero entries), so $m=n$ or $m \leq n-2$. In the latter case, 
$|\psi((1,0,\dots,0))|\leq n-2,$
which is impossible. Since $C^n$ is generated by $G$-conjugates of $(1,0,\dots,0)$, we find that $C^n$ fixes all lines in $L'$. Hence $L' \subseteq L$, but since $|L|=|L'|=[\tilde G:\tilde H]$, we have $L=L'$. 
\end{proof} 

\subsection{Application to manifolds} 

We deduce the following intermediate result. 

\begin{corollary} \label{art} 
Suppose we have a diagram \textup{(\ref{sunadasetup})}. Let $C:={\Z}/\ell{\Z}$ denote a cyclic group of prime order $\ell \geq 3$. Let $\tilde G$ and $\tilde H_1$ denote the wreath products as in Definition \textup{\ref{wreath}} \textup{(}with $H = H_1$\textup{)} and $\tilde H_2 := C^n \rtimes H_2$ \textup{(}with the same action defined via the $H_1$-cosets\textup{)}, and assume that there exists a diagram of Riemannian coverings
\begin{equation} \label{leftpartextend}
  \xymatrix@C=2mm@R=4mm{ & M' \ar[dd]^{C^n} \ar@/_1pc/[dddl]_{\tilde H_1} \ar@/^2.5pc/[dddd]^{\tilde G} \\ &  \\ & M \ar[dl]_<<{H_1} \ar[dd]^-G  \\ M_1\ar[dr]_{p_1} &   \\ & M_0  } 
\end{equation} 
Then $M_1$ and $M_2$ are equivalent Riemannian covers of $M_0$ if and only if for a $\tilde G$-solitary character $\Xi$ on $\tilde H_1$ and for some linear character $\chi$ on $\tilde H_2$, the multiplicity of zero is equal in the two spectra $$ \Sp_{M_1}(\bar \Xi \otimes \Res_{\tilde H_1}^{\tilde G} \Ind_{\tilde H_1}^{\tilde G} \Xi) \mbox{ and } \Sp_{M_2}(\bar \chi \otimes \Res_{\tilde H_2}^{\tilde G} \Ind_{\tilde H_1}^{\tilde G} \Xi)$$ and in the two spectra
$$\Sp_{M_1}(\bar \Xi \otimes \Res_{\tilde H_1}^{\tilde G} \Ind_{\tilde H_2}^{\tilde G} \chi) \mbox{ and } 
\Sp_{M_2}(\bar \chi \otimes \Res_{\tilde H_2}^{\tilde G} \Ind_{\tilde H_2}^{\tilde G} \chi).$$
\end{corollary}

\begin{proof} First of all, since $\ell \geq 3$, a $\tilde G$-solitary character $\Xi$ on $H_1$ exists, by Proposition \ref{solitary}. 
By Proposition \ref{solo}, the equalities of multiplicities of zero is equivalent to $\Ind_{\tilde H_1}^{\tilde G} \Xi \cong  \Ind_{\tilde H_2}^{\tilde G} \chi$. Since $\Xi$ is $\tilde G$-solitary, we conclude by Lemma \ref{lem:solitarycharconjgroups} that $\tilde H_1$ and $\tilde H_2$ are conjugate in $\tilde G$. As $C^n$ is normal in $\tilde H_2$ with quotient $H_2$, we find that $\tilde H_2 \backslash M' = H_2 \backslash M = M_2$ and hence this conjugacy defines an isometry from $M_1$ to $M_2$ that is the identity on $M_0$. 
\end{proof} 

Since $\chi$ runs over linear characters of $\tilde H_2$, the ``less abelian'' the extension is, the less spectra need to be compared. A more precise statement is the following. 

\begin{proposition} \label{numberofchecks} In the setup of Corollary \textup{\ref{art}}, the dimension of the representations of which the spectra are being compared is the index $[G:H_2]$. Furthermore, the number of spectral equalities to be checked in Corollary \textup{\ref{art}} by using all possible linear characters on $\tilde H_2$ is bounded above by $2 \ell \cdot |H^{\mathrm{ab}}_2|$. 
\end{proposition}

In Corollary \ref{art} and Proposition \ref{numberofchecks}, one may interchange the roles of $H_1$ and $H_2$, which could lead to tighter results. 

\begin{proof} The dimension of the representations we are considering, as induced representations, is the index $[\tilde G: \tilde H_2]=[G:H_2]$. 

The spectral criterion in the proposition requires testing of 2 equalities of spectra for each linear character on $\tilde H_2$, so there are at most $2 | \tilde{H}^{\mathrm{ab}}_2 | $ equalities to be checked. 

The commutator subgroup of a wreath product $\tilde H_2 = C^n \rtimes H_2$ is computed in \cite[Cor.\ 4.9]{Meldrum}, and we find that in our case, with $\Omega =\{g_1,\dots,g_n\}$ a set of representatives for the cosets, 
$$ |[\tilde H_2, \tilde H_2]| = |[H_2,H_2]| \cdot |\{ f \colon \Omega \rightarrow C \colon \sum_{y \in \Omega} f(y) = 0 \}|; $$
where, with $|C|=\ell$, the second factor is $\ell^{|\Omega|-1}$. Hence we find 
$ |\tilde H_2^{\mathrm{ab}}| = |H_2^{\mathrm{ab}}| \cdot \ell,$
and the result follows. 
\end{proof} 

\begin{remark} \label{pint}  Pintonello \cite[Theorem 3.2.2]{Pintonello} has shown that for $\ell=2$, there does not always exist a solitary character as in Proposition \ref{solitary}. He also proved that in this case, there exists such a quadratic character if one additionally assumes weak conjugacy (i.e., $\Ind^{\tilde G}_{\tilde H_1} \one \cong \Ind^{\tilde G}_{\tilde H_2} \one$) \cite[Theorem 2.3.1]{Pintonello}. Again using Proposition \ref{solo} to reformulate this extra assumption spectrally, we find that in this case, the number of equalities to check is at most $2+4|H_2^{\mathrm{ab}}|$.
\end{remark}

\begin{remark} By Lemma \ref{mult}, the multiplicity of zero in the spectrum $\Sp_M(\rho)$ can be computed purely representation theoretically as the multiplicity of the trivial representation in $\rho$, which is in principle possible by Mackey theory (cf.\ Remark \ref{mackey}), but this would be going in reverse (from spectra to group theory instead of the other way around). 
Knowing the group $G$ and its subgroups $H_1$ and $H_2$, Riemannian equivalence of $M_1$ and $M_2$ over $M_0$ can be checked by a finite computation, verifying that $H_1$ and $H_2$ are conjugate in $G$. Corollary \ref{art} translates this into a spectral statement (in the special setup where the group $\tilde G$ is realised as indicated there).
\end{remark} 

In the next sections, we study under which circumstances we have a cover as in Corollary \ref{art}, i.e., we deal with the realisation problem for the wreath product as isometry group of a cover, given an isometric free action of $G$ on a closed manifold $M$. This is analogous to the inverse problem of Galois theory, realising the wreath product as Galois group of a number field. In manifolds, some condition is necessary on $M$ for such an extension to be possible at all.

\section{Construction of suitable covers and proof of the main theorem} 
\label{sec:covers}
\label{sec:prelimhom} 

\subsection{Fundamental group and first homology} We first fix some notations and constructions. Let $M$ denote a connected closed oriented smooth Riemannian manifold. 
Fixing a point $x \in M$, the universal covering $\tilde M$ is described as the set of homotopy classes $[w]$ of paths $w: [0,1] \to M$ with $w(0) = x$. This provides a projection map $$\Pi: (\tilde M,\tilde x) \to (M,x),\ \Pi([w]) = w(1),$$ where $\tilde x \in \tilde M$ represents the homotopy class of the constant path at $x$. If we equip $\tilde M$ with the pull-back of the Riemannian metric on $M$ then the group $\Gamma$ acts isometrically by deck transformations on $\tilde M$, and $M$ is identified with the quotient $\Gamma \backslash \tilde M$. 

Let $*$ denotes path-concatenation read from left to right, that is, $[a] * [b]$ is the homotopy class of the path obtained by first traversing $a$ and then $b$. Letting $\tilde x$ denote the constant path at $x$, we have an identification 
$$\Phi_{\tilde x}: \Gamma \to \pi_1(M,x)$$ via the map $\Phi_{\tilde x}(\gamma) = \gamma(\tilde x)$, with $\gamma(\tilde x)$ representing a homotopy class of a closed loop starting and ending at $x$.  More generally, any homotopy class of a path $[w]$ in $\tilde M$ induces a map $\Phi_{[w]}: \Gamma \to \pi_1(M,w(1))$ via
\begin{equation} \label{wphi} \Phi_{[w]}(\gamma) = [w^{-1}] * \Phi_{\tilde x}(\gamma) * [w]. \end{equation} 

We denote the first homology group of $M$ (with integer coefficients) by $\Ho_1(M) = \Ho_1(M,\Z)$. The universal coefficient theorem for homology \cite[\S 3.A]{Hatcher} implies that for any field $K$, we have an isomorphism $$\Ho_1(M,K) = \Ho_1(M) \otimes_{\Z} K.$$ Also, since $M$ is a connected manifold, it is path-connected, so that we have a Hurewicz homomorphism inducing an identification
$$\Ho_1(M) = \pi_1(M,x)^{\ab} \cong \Gamma^{\ab}.$$ 
The map is given by considering a (homotopy class of a) loop as a concatenation of oriented $1$-cells and mapping it to the (homology class of the) signed sum of those cells. 

Composing the maps, we have a homomorphism $\Psi_0: \Gamma \to \Ho_1(M,\F_\ell)$ given as  the composition of $\Phi_{[w]}$ with the abelianisation map and the Hurewicz isomorphism, followed by reduction modulo $\ell$. 
\begin{equation} \label{eq:Phicomp} 
\xymatrix{ \Gamma \ar[r]^(.36){\Phi_{[w]}} \ar@/_1.5pc/[rrr]_{\Psi_0} & \pi_1(M,w(1)) \ar[r]^{\cdot^{\ab}} &  \Ho_1(M,\Z) \ar[r]^(.48){\otimes \F_\ell}  & \Ho_1(M,\F_\ell),}
\end{equation}
Since by \eqref{wphi} the homotopy classes of loops $\Phi_{[w]}$ for different $w$ are freely homotopic, the composed map $\Psi_0$ is independent of the choice of $w$, as notation indicates. The standard choice is $w=\tilde x$, but we will naturally encounter others.   

\subsection{First homology and Galois covers} \label{fhgc} Suppose now that $G$ is a finite group of isometries acting on $M$, and let $q \colon M \rightarrow M_0$ denote the orbifold quotient, with $\Gamma_0$ the covering group of $ \Pi_0 \colon \tilde M \rightarrow M_0$. Let $\Gamma$ be the normal subgroup of $\Gamma_0$ corresponding to $\Pi \colon \tilde M \rightarrow M$, so that  there is a short exact sequence of groups \begin{equation} \label{sesg} 1 \rightarrow \Gamma \rightarrow \Gamma_0 \xrightarrow{F} G \rightarrow 1 \end{equation} giving an isomorphism $G \cong \Gamma_0 / \Gamma$. 
The setup is summarised in diagram \eqref{qq}. 
\begin{equation} \label{qq}
  \xymatrix@C=6mm@R=6mm
  { \tilde M \ar@/_2pc/[dd]_{\Gamma_0}^{\Pi_0} \ar[d]^--{\Gamma}_{\Pi} \\  M \ar[d]^--G_q  \\  M_0  } 
\end{equation} 

\begin{definition} 
For $\gamma_0 \in \Gamma_0$, let $\conj_{\gamma_0}: \Gamma \to \Gamma$ be conjugation by $\gamma_0$, that is
$$ \conj_{\gamma_0}(\gamma) = \gamma_0 \gamma \gamma_0^{-1}. $$
\end{definition} 

\begin{remark} \label{outer} If $g \in G \cong \Gamma_0/\Gamma$ satisfies $g = \gamma_0 \Gamma$, this represents the usual map $$G \rightarrow \mathrm{Out}(\Gamma) \colon g \mapsto \conj_{\gamma_0}$$ induced by the exact sequence \eqref{sesg}, where $\mathrm{Out}(\Gamma) = \mathrm{Aut}(\Gamma) / \mathrm{Inn}(\Gamma)$ is the quotient of the group of automorphisms of $\Gamma$ by the group  $\mathrm{Inn}(\Gamma) = \{ \conj_\gamma \colon \gamma \in \Gamma\}$ of inner automorphisms. 

In this setup, the action of the isometry $g$ on $M$ corresponds to the action of $\conj_{\gamma_0^{-1}}$ on $\Gamma$, so that there is a commuting diagram 
\begin{equation} \label{isomconj} 
  \xymatrix@C=6mm@R=6mm
  { & \tilde M \ar[dl]_{\Gamma} \ar[dr]^{\gamma_0^{-1}\Gamma \gamma_0} & \\  M \ar[rr]^g & &   M} 
\end{equation} 
\end{remark}

The action of $G$ on $M$ by isometries induces a linear action of $G$ on $\Ho_1(M,\F_\ell)$, providing an $\F_\ell[G]$-module structure on this homology group. The following lemma describes the relation between this action and the above outer conjugation on $\Gamma$: they commute by the Hurewicz map. The argument is similar to the one used in the proof of Hopf's formula \cite[(5.3)]{BrownK}. 

\begin{lemma} \label{lem:homology_and_fundgroup}
Let $\gamma_0 \in \Gamma_0$ and $g \in G \cong \Gamma_0/\Gamma$ such that $g = \gamma_0 \Gamma$. Then the following diagram commutes:
\begin{equation} \label{homology_and_fundgroup}
\xymatrix@C=6mm@R=6mm{\Gamma \ar[d]_--{\Psi_0}\ar[r]^-{\conj_{\gamma_0}} & \Gamma \ar[d]^--{\Psi_0} \\
\Ho_1(M,\F_\ell) \ar[r]^-{g \cdot} & \Ho_1(M,\F_\ell)
}
\end{equation}
where $\Psi_0$ is as in \textup{\eqref{eq:Phicomp}} and the bottom line indicates the action of $g \in G$ on the first homology group. 
\end{lemma}

\begin{proof} 
The vertical maps are given by picking a base point $x \in M$, considering $\gamma \in \Gamma$ as a homotopy class of a closed loop in $M$ based at $x$ via $\Phi_{\tilde x}$, rewriting $\gamma$ as a concatenation of oriented $1$-cells, and mapping these to the corresponding sum of $1$-cells in homology. The crucial observation that makes the proof work is that we can decompose into $1$-cells, not just $1$-cycles, and the image is independent of the choice of base point. In this way, the closed loop corresponding to $\conj_{\gamma_0}(\gamma)$ decomposes as a (left-to-right) concatenation of the $1$-cells $$(\gamma_0: x \rightarrow gx) \ast ( g  \cdot \gamma: gx \rightarrow gx) \ast (\gamma_0^{-1}: gx \rightarrow x),$$ that is mapped by $\Psi_0$ to the sum of homology classes $[\gamma_0] + g\cdot [\gamma] - [\gamma_0] = g \cdot [\gamma],$ proving the commutativity of the diagram.
\end{proof} 

\begin{remark} In fact, there is a larger commutative diagram:
\begin{equation} \label{homology_and_fundgroup_larger}
\xymatrix@C=6mm@R=6mm{\Gamma \ar[d]_{\Phi_{\tilde x}}\ar[r]^-{\conj_{\gamma_0}}  \ar@/_3pc/[dd]_{\Psi_0} & \Gamma \ar[d]^{\Phi_{\gamma_0 \tilde x}}  \ar@/^3pc/[dd]^{\Psi_0} \\
\pi_1(M,x) \ar[d]\ar[r]^-{g \cdot} & \pi_1(M,gx) \ar[d] \\
\Ho_1(M,\F_\ell) \ar[r]^-{g \cdot} & \Ho_1(M,\F_\ell)
}
\end{equation} 
The detailed proof goes as follows: we have argued before that both vertical maps (from top to bottom) on the left and right hand side of \eqref{homology_and_fundgroup_larger} agree and are equal to
$\Psi_0$. 
Naturality of the Hurewicz isomorphism guarantees commutativity of the lower square in \eqref{homology_and_fundgroup_larger} and it therefore suffices  to prove commutativity of the upper square in \eqref{homology_and_fundgroup_larger}. 

Elements $g \in G$ are isometries $g: M \to M$ and 
the groups $\Gamma_0$ and $\Gamma$ acting by isometries on $\tilde M$ can be described via deck transformations as follows 
\begin{align}
& \Gamma := \{ I: \tilde M \to \tilde M\, \text{isometry} \mid
F(I([w])) = F([w])\, \forall\, [w] \in \tilde M \}, \nonumber \\
& \Gamma_0 := \{ I: \tilde M \to \tilde M\, \text{isometry} \mid
\exists\, g \in G: F(I([w]) = g F([w])\, \forall\, [w] \in \tilde M \}. \label{eq:G0}
\end{align}
In this description, the map $F : \Gamma \to G$ from \eqref{sesg} is given by 
mapping $I \in \Gamma_0$ to the
(uniquely determined) corresponding element $g \in G$ in \eqref{eq:G0}. 

Let $\pi_1(M,x,y)$ denote the set of homotopy classes of paths in $M$ from $x$ to $y$. We identify the elements in $\Gamma_0$ with homotopy classes of paths starting at $x$ via the following bijective map:
\begin{eqnarray}
\pi_1(M,x,gx) &\to& F^{-1}(g) \subset \Gamma_0, \label{eq:top-useful} \\
{[a]} &\mapsto&
 I_a
 \colon [w] \mapsto [a] * [gw]
. \nonumber
\end{eqnarray}
It can be easily checked that 
$ I_a^{-1}([w]) = [g^{-1}a^{-1}] * [g^{-1}w]. $
To prove commutativity of the upper square of diagram \eqref{homology_and_fundgroup_larger}, we go around the square both ways.
\begin{itemize}[leftmargin=1.4em]
\item \emph{Computing  $\Phi_{\gamma_0 \tilde x}(\conj_{\gamma_0}(\gamma))$.} We first describe the conjugation action in terms of concatenation. Using \eqref{eq:top-useful},
we identify $\gamma_0$ with a map $I_a: \tilde M \to \tilde M$ for a path $a$ satisfying $a(0)=x$ and $a(1)=gx$. Similarly, we identify $\gamma$ with a map $I_c: \tilde M \to \tilde M$ for a path $c$ satisfying
$c(0)=c(1)=x$. For any $\tilde y =[w] \in \tilde M$ with a path $w$ in $M$ starting at $w(0)=x$, we have
$$
\conj_{\gamma_0}(\gamma)(\tilde y) = I_a(I_c(I_a^{-1}([w]))) =
[a] * [gc] * [a^{-1}] * [w]
$$
and, in particular,  $$\conj_{\gamma_0}(\gamma)(\tilde x) = [a] * [gc] * [a^{-1}].$$
Now we evaluate $\Phi_{\gamma_0 \tilde x}(\conj_{\gamma_0}(\gamma))$. We first note that
$$ \gamma_0 \tilde x = I_a(\tilde x) = [a] * (g \tilde x) = [a], $$
since $g \tilde x$ is the homotopy class of the constant path at $gx$ in $M$. This implies
\begin{align*} \Phi_{\gamma_0 \tilde x}(\conj_{\gamma_0}(\gamma)) & = [a^{-1}] * \Phi_{\tilde x}(\conj_{\gamma_0}(\gamma)) * [a] \\ &= [a^{-1}] * \conj_{\gamma_0}(\gamma)(\tilde x) * [a] \\ &= [gc]. \end{align*}
\item \emph{Computing $g \cdot \Phi_{\tilde x}(\gamma)$.} 
We have
$ g \cdot \Phi_{\tilde x}(\gamma) = g \cdot \gamma(\tilde x) = g \cdot I_c(\tilde x) = g \cdot ([c] * \tilde x) = [g c]. $
\end{itemize} 
Since the results of the two computations agree, the proof is finished. \qed
\end{remark} 

\subsection{Realisability of the wreath product} \label{constrmp} The following result gives an exact topological criterion for realisation of the wreath product, in terms of the $\F_\ell[G]$-module structure of $\Ho_1(M,\F_\ell)$.

\begin{proposition} \label{maininterprop0}
Suppose that we have a diagram of Riemannian coverings 
\begin{equation} \label{leftparthalfprop}
  \xymatrix@C=2mm@R=4mm{ & M \ar[dl]_{H_1}^{q_1} \ar[dd]^G  \\ M_1\ar[dr]_{p_1} &   \\ & M_0  } 
\end{equation} 
with $M_0$ a developable orbifold and $M,M_1$ manifolds. Fix a prime $\ell$, and set $C={\Z}/{\ell}{\Z}$ and consider the wreath products $\tilde G$ and $\tilde H_1$ as in Definition \textup{\ref{wreath}} \textup{(}with $H = H_1$\textup{)}.  Let $\{g_1 H_1 = H_1, g_2 H_1, \dots, g_n H_1\}$ denote the cosets in 
$G/H_1$. For any $g \in G$, define the permutation $i \mapsto g(i)$ on $\{1,\dots,n\}$  and the element $h_{g,i} \in H_1$ via $g g_i = g_{g(i)} h_{g,i}$. 

Define the $\F_\ell[G]$-module $\mathcal N$ as $\mathcal N:=(\Ind_{H_1}^G \one) \otimes_{\Z} \F_\ell$ \textup{(}cf.\ Remark \textup{\ref{indfl}}\textup{)}. 
 Then the 
diagram \eqref{leftparthalfprop} can be extended to a diagram of the form \eqref{leftpartextend} if and only if $\mathcal N$ is a $\F_\ell[G]$-quotient module of $\Ho_1(M,\F_\ell)$.
 \end{proposition}

\begin{proof}
We fix the following for the duration of the proof. 
\begin{itemize}
\item Set $\mathcal M:=\Ho_1(M,\F_\ell)$.
\item Let $e_i = (0,\dots,0,1,0,\dots,0) \in C^n \cong \F_\ell^n$ denote the standard ``basis vectors''. 
\item As in Remark \ref{indfl}, we identify $\mathcal N \cong C^n \cong \F_\ell^n$ as $\F_\ell[G]$-modules, where the action of $G$ on $C^n$ given by the permutation representation of the cosets implemented by the map \eqref{eq:Phidef} used to define the wreath product, i.e.,  $\Phi \colon G \rightarrow \Aut(C^n) \colon g \mapsto (e_i \mapsto e_{g(i)})$. 
\item $\tilde M$ is the universal covering of $M$, $M_0 = \Gamma_0 \backslash \tilde M$, $M_1 = \Gamma_1 \backslash \tilde M$ and $M = \Gamma \backslash \tilde M$. 
\item $F$ is the map from \eqref{sesg}.
\end{itemize} 

In one direction, assume the extended cover exists, corresponding to a quotient of $\Gamma$. Now $M' \rightarrow M$ is a Galois cover with group the elementary abelian $\ell$-group $C^n$. Using the universal property of abelianisation, the corresponding map $\Gamma \rightarrow C^n$, whose kernel we denote by $\Gamma'$, factors through to a map of $\F_\ell$-vector spaces
$$ \varphi \colon \mathcal M \twoheadrightarrow C^n \cong \mathcal N. $$
We claim that this is a map of $\F_\ell[G]$-modules. For each $g \in G$ consider the corresponding element $\tilde g=(0,g) \in \tilde G$. Let $\gamma_0 \in \Gamma_0$ denote any lift of $\tilde g$ (i.e., $\tilde g = \gamma_0 \Gamma' \in \tilde G=\Gamma_0/\Gamma'$). By Lemma \ref{lem:homology_and_fundgroup} applied to both $M'$ and $M$, $\tilde g$ acts on $\Gamma'$ and $g$ acts on $\Gamma$ as $\conj_{\gamma_0}$. Hence $\tilde g$ acts on $\Gamma/\Gamma' = C^n$ as conjugation by $\gamma_0 \Gamma' = \tilde g$. Now in the semidirect product $\tilde G$, the action of $\tilde g$ on $c \in C^n$ is given by conjugation $c \mapsto \tilde g c \tilde g^{-1}$, which corresponds to the action of $G$ on $C^n$ via $\Phi$, i.e., gives the $\F_\ell[G]$-module structure $\mathcal N$ to $C^n$. Therefore, $\varphi(g m) = g \varphi(m)$, as desired. 

\medskip

Conversely, suppose there exists a map $ \varphi \colon \mathcal M \twoheadrightarrow \mathcal N $ of $\F_\ell[G]$-modules. 
Precomposing with $\Psi_0$ as in \eqref{eq:Phicomp}, we get a map $\Psi \colon \Gamma \rightarrow C^n$, and we set \begin{equation} \label{herepsi} \Gamma' := \ker \Psi \vartriangleleft \Gamma\mbox{ and } M':=\Gamma' \backslash M.\end{equation}  
With this map we can also extend the commutative diagram
\eqref{homology_and_fundgroup} as follows:
\begin{equation} \label{extended_homology_and_fundgroup}
\xymatrix@C=6mm@R=6mm{\Gamma \ar@/_3pc/[dd]_{\Psi}\ar[d]_-{\Psi_0} \ar[r]^-{\conj_{\gamma_0}}
 & \Gamma \ar@/^3pc/[dd]^{\Psi}\ar[d]^-{\Psi_0} \\
\Ho_1(M,\F_\ell) \ar[d]_{\varphi}\ar[r]^-{g \cdot} & \Ho_1(M,\F_\ell) \ar[d]^{\varphi} \\
C^n \ar[r]^-{\Phi(g)}& C^n
}
\end{equation}
where the commutativity of the bottom square is guaranteed precisely by our assumption that $\varphi$ is a map of $\F_\ell[G]$-modules. 

It remains to verify that the manifold $M'$ fits into a diagram \eqref{leftpartextend}. For this, we need to prove that  $\Gamma'$ is normal in $\Gamma_0$ (hence also normal in $\Gamma$ and $\Gamma_1$) such that there are induced group isomorphisms $\Gamma_0 / \Gamma' \cong \tilde G$, $\Gamma_1 / \Gamma' \cong \tilde H_1$ and $\Gamma / \Gamma' \cong C^n$, where $\tilde G = C^n \rtimes G$ is the wreath product introduced in Definition \ref{wreath} (with $H = H_1$).  For this, we show that the group $\Gamma'$ has the following properties (i)--(v). 
\begin{enumerate}[leftmargin=1.4em]
\item \emph{$\Gamma'$ is normal in $\Gamma_0$;} indeed, commutativity of diagram \eqref{extended_homology_and_fundgroup} implies that if $\gamma_0 \in \Gamma_0$ with $F(\gamma_0)=g$ and $\Psi(
\gamma')=0$ (i.e., $\gamma' \in \Gamma'$), then $\Psi(\conj_{\gamma_0}(\gamma'))=\Phi(g)\Psi(\gamma')=0$ (i.e., $\gamma_0 \gamma' \gamma_0^{-1} \in \Gamma')$.
\end{enumerate} 
Fix a set-theoretic section $G \rightarrow \Gamma_0 \colon g \mapsto \bar g$ of the map $F$, i.e., for any $g \in G$ fix any $\bar g \in \Gamma_0$ such that $F(\bar g)=g$.  
Fix an element $c \in \Gamma$ with $\Psi(c) = e_1$ (which exists since $\Psi$ is surjective), and for $i=1,\dots,n$, let $c_i := \bar g_i c \bar g_i^{-1} \in \Gamma_0.$ Notice that 
\begin{equation} \label{eq:Psici}
\Psi(c_i) = \Phi(g_i)e_1 = e_i.
\end{equation} 
Then we have the following properties: 
\begin{enumerate}[resume, leftmargin=1.7em]
\item \emph{$c_i c_j \Gamma' = c_j c_i \Gamma'$ for all $1 \leq i,j \leq n$.} This follows since $\Gamma/\Gamma' \cong C^n$ is commutative. 
\item \emph{The cosets of $\Gamma'$ in $\Gamma_0$ are represented by $\bar g c_1^{k_1} \cdots c_n^{k_n} \Gamma'$ with $g \in G$ and $0 \leq k_1,\dots,k_n < \ell$.}
From \eqref{eq:Psici}
we see that $c_1^{k_1} \cdots c_n^{k_n} \Gamma'$ are the cosets of $\Gamma/\Gamma'$, and this combines with   
$G \cong \Gamma_0/\Gamma$.
\item \emph{$\bar g c_i \Gamma' = c_{g(i)}\bar g \Gamma'$ for all $g \in G$.} 
Indeed, $$
\Psi(\bar g c_i \bar g^{-1}) \stackrel{\eqref{extended_homology_and_fundgroup}}{=} \Phi(g)\Psi(c_i) \stackrel{\eqref{eq:Psici}}{=} \Phi(g)(e_i) 
\stackrel{\eqref{eq:Phidef}}{=} e_{g(i)} \stackrel{\eqref{eq:Psici}}{=} \Psi(c_{g(i)}). 
$$
It follows that $\bar g c_i \bar g^{-1} \Gamma' = c_{g(i)} \Gamma'$ and, since $\Gamma'$ is normal in $\Gamma_0$, we can interchange left and right cosets and multiply on the right with $\bar g$ to find the result. 
\end{enumerate}

\noindent By (iii) and (iv), the cosets of $\Gamma_0/\Gamma'$ are given by
$ \bar g c_1^{k_1} \cdots c_n^{k_n} \Gamma' = 
c_{g(1)}^{k_1} \cdots c_{g(n)}^{k_n} \bar g \Gamma' $
with $g \in G$ and  $0 \leq k_1,\dots,k_n < \ell$.
Define $\hat \Psi: \Gamma_0 \to \tilde G = C^n \rtimes G$ by 
$$ \hat \Psi(c_1^{k_1} \cdots c_n^{k_n} \bar g \Gamma') = 
(k_1,\dots,k_n,g). $$

\begin{enumerate}[resume, leftmargin=1.5em]
\item Using the commutativity property in (ii), we have for $k_i,k_j' \in \{0,1,\dots,\ell-1\}$,
and $g,g' \in G$, 
 $$ \left( c_1^{k_1} \cdots c_n^{k_n} \bar g \Gamma' \right) \left( 
  c_1^{k_1'} \cdots c_n^{k_n'} \bar g' \Gamma' \right) = \left( c_1^{k_1} \cdots c_n^{k_n}
  c_{g(1)}^{k_1'} \cdots c_{g(n)}^{k_n'} \right) (\bar g \cdot \bar g') \Gamma', $$
which follows immediately from (i) and (iv).
\end{enumerate}

\begin{enumerate}[resume, leftmargin=1.5em]
\item \emph{Via a modification of the section $G \to \Gamma_0$, $g \mapsto \bar g$, the map $\hat \Psi$ becomes a surjective group homomorphism with kernel $\Gamma'$.} This follows immediately from (v) if we can choose the section in such a way that $(\bar g \cdot \bar g') \Gamma' = \overline{g\cdot g'} \Gamma'$. Consider the short exact sequence
\begin{equation} \label{eq:ses} 
\xymatrix@C=6mm@R=3mm{0 \ar[r]  & \Gamma / \Gamma' \ar[r] \ar@{}[d]|*=0[@]{\cong} &  \Gamma_0 / \Gamma' \ar[r]^{\alpha} &  \Gamma_0 / \Gamma \ar@{}[d]|*=0[@]{\cong}\ar[r] &  1, \\
& C^n & & G &
}
\end{equation}
with the canonical maps. Note that $\Gamma / \Gamma' \cong C^n$ is a $G$-module via the action
$$ g \cdot \left( c_1^{k_1} \cdot c_n^{k_n} \Gamma' \right) = c_{g(1)}^{k_1} \cdot c_{g(n)}^{k_n} \Gamma'. $$
Since the order of $G$ is coprime to that of $C^n$, we have  $\Ho^1(G,\Gamma/\Gamma') = \Ho^2(G,\Gamma/\Gamma') = 0$ \cite[IV 2.3, 3.12, 3.13]{BrownK} and hence \eqref{eq:ses} splits; let $\jmath$ denote a group homomorphism $\jmath: G \to \Gamma_0 / \Gamma'$ such that
$\alpha \circ \jmath = {\rm{id}}_G$. Redefine the section
$G \to \Gamma_0$ in such a way that $\bar g \Gamma' = \jmath(g)$. Then we have 
$$ (\bar g \cdot \bar g') \Gamma' = \bar g \Gamma' \cdot \bar g' \Gamma' = \jmath(g) \cdot \jmath(g') = \jmath(gg') = \overline{g\cdot g'} \Gamma' $$
and $\hat \Psi: \Gamma_0 \to \tilde G$ is a surjective group homomorphism with kernel $\Gamma'$. 
\end{enumerate} 
  
\medskip

\noindent It follows that $\hat \Psi$ induces an isomorphism $\Gamma_0/\Gamma' \cong \tilde G$, and we have already 
seen the isomorphism $\Gamma/\Gamma' \cong C^n$ in the proof of (ii). To show that $\Gamma_1/\Gamma' \cong \tilde H_1$ and finish the proof, note that $c_1^{k_1} \cdots c_n^{k_n} \bar h \Gamma'$ with $k_i \in \{0,\dots,\ell-1\}$ and $h \in H_1$ are the cosets of $\Gamma_1/\Gamma'$ and that the quotient $\Gamma_1 / \Gamma'$ is isomorphic to the subgroup $\tilde H_1$ of $\tilde G$ under the restriction of the  isomorphism induced by $\hat \Psi$.
\end{proof}

\subsection{Main result} We can now prove the main result of this paper.  

\begin{theorem} \label{maindetail} 
Suppose $M_1$ and $M_2$ are two connected closed oriented smooth Riemannian manifolds such that there is a diagram  
\eqref{m0} 
of finite covers of a developable Riemannian orbifold $M_0$. 
Then
\begin{enumerate}
\item The diagram \eqref{m0}  may be extended to a diagram of finite coverings
\eqref{sunadasetup} 
where $M$ is a connected closed smooth Riemannian manifold $M$ with $q_1 \colon M \twoheadrightarrow M_1:=H_1 \backslash M$, $q_2 \colon M \twoheadrightarrow  M_2:=H_2 \backslash M$ and $q \colon M  \twoheadrightarrow M_0:=G\backslash M$ Galois covers.
\end{enumerate} 
Suppose furthermore that there exists a prime number $\ell \geq 3$ such that 
\begin{quote} \textup{($\ast$)} $(\Ind_{H_1}^G \one) \otimes_{\Z} \F_\ell$ is an $\F_\ell[G]$-quotient module of $\Ho_1(M,\F_\ell)$. \end{quote} Then  
\begin{enumerate}[resume] 
\item There exists a diagram of Riemannian coverings 
\eqref{leftpartextend},  
where $C = \Z / \ell \Z$, $\tilde G = C^n \rtimes G$ is the wreath product where $G$ permutes the copies of $C$ in the same way as it permutes the cosets of $H_1$ in $G$, and $\tilde H_i = C^n \rtimes H_i$ are subgroups of $\tilde G$ corresponding to the groups $H_i$, $i=1,2$. 
\item Consider the linear character $$\Xi \colon \tilde H_1 \rightarrow \C^* \colon (k_1,\dots,k_n,h_1) \rightarrow e^{2 \pi i k_1}/\ell$$ with $(k_1,\dots,k_n) \in C^n$ and  $h_1 \in H_1$. Then the manifolds $M_1$ and $M_2$ are equivalent Riemannian covers of $M_0$
if and only if there exists a linear character $\chi \colon \tilde H_2 \rightarrow \Un_1(\C)$ such that the multiplicity of zero in 
the following two pairs of spectra of twisted Laplacians on $M_1$ and $M_2$ coincide: 
$$ \Sp_{M_1}(\bar \Xi \otimes \Res_{\tilde H_1}^{\tilde G} \Ind_{\tilde H_1}^{\tilde G} \Xi) \mbox{ and } \Sp_{M_2}(\bar \chi \otimes \Res_{\tilde H_2}^{\tilde G} \Ind_{\tilde H_1}^{\tilde G} \Xi) $$ 
and
$$ \Sp_{M_1}(\bar \Xi \otimes \Res_{\tilde H_1}^{\tilde G} \Ind_{\tilde H_2}^{\tilde G} \chi) \mbox{ and } 
\Sp_{M_2}(\bar \chi \otimes \Res_{\tilde H_2}^{\tilde G} \Ind_{\tilde H_2}^{\tilde G} \chi). $$
There are $\ell |H_2^{\ab}|$ linear characters $\chi$ on $\tilde H_2$, and the dimension of the representations involved is the index $[G:H_2]$.
\end{enumerate}
\end{theorem}  

\begin{proof} 
Part (i) holds by Proposition \ref{fromdiagramtodiagram}. 
Part (ii) is shown in Proposition \ref{maininterprop0}. Then (iii) holds by Corollary \ref{art} and Proposition \ref{numberofchecks}, using that the character $\Xi$ given in (iii), is the one constructed in the proof of Proposition \ref{solitary} (cf.\ formula (\ref{defxi})), and is $\tilde G$-solitary on $\tilde H_1$. 
\end{proof} 

\begin{remark} 
Using Proposition \ref{fixm0}, the condition on the multiplicity of zero in the spectrum of the indicated twisted Laplacians in Theorem \ref{maindetail} may be replaced by an equality of their spectral zeta functions, if we assume in addition that $M_1$ and $M_2$ are isospectral, i.e., $\zeta_{M_1, \Delta_{M_1}} = \zeta_{M_2, \Delta_{M_2}}$.  
\end{remark} 

Condition \textnormal{\textup{($\ast$)}} in the main theorem can be varied, as we will see in the next two sections. This will also produce a geometric realisation of $M'$ and a set of examples where the condition holds.

\section{Geometric construction of the covering manifold} \label{geocon} 

In this section, we provide a geometric construction of $M'$, using composita and fiber products, for which we need to assume that $\ell$ is a prime number \emph{coprime to $|G|$}.

\subsection{From quotient to submodule} \label{astastsection} If $\ell$ is a prime number coprime to $|G|$, by Maschke's theorem, any short exact sequence of $\F_\ell[G]$-modules splits, and condition \textup{($\ast$)} from Theorem \ref{maindetail} is equivalent to 
\begin{quote} \textup{($\ast \ast$)} \emph{$(\Ind_{H_1}^G \one) \otimes_{\Z} \F_\ell$ is an $\F_\ell[G]$-submodule of $\Ho_1(M,\F_\ell)$.} \end{quote} 

\subsection{Homology of a quotient as coinvariants} We recall the following tool from invariant theory, see, e.g.\ \cite[II \S 2]{BrownK}. If $R$ is a commutative ring (for us, $R$ is $\Z, \Q$ or $\F_\ell$), $H$ a finite group, and $\mathcal M$ is a (left) $R[H]$-module, its \emph{coinvariants} are defined as 
the $R$-module $\mathcal M_H:= \mathcal M / I \mathcal M$ where $I$ is the kernel of the augmentation map $R[H] \rightarrow R \colon \sum k_h h \mapsto \sum k_h$. An explicit description is given by $$I \mathcal M = \langle h(x)-x \colon h \in H, x \in \mathcal M \rangle$$
(by linearly, it suffices to let $x$ run over a set of generators of $\mathcal M$).  
Denote the projection map by \begin{equation} \label{transmap} \trans_R \colon \mathcal M \rightarrow \mathcal M_H = \mathcal M / I \mathcal M. \end{equation} When $R$ is clear from the context, we will leave it out of the notation and simply write $\trans$ for this map. 

This map is particularly easy if 
$ \mathcal M = \bigoplus R[H] x_i$ is \emph{free} as an $R[H]$-module with generators $x_i$; then $\mathcal M_H = \bigoplus R x_i$ with the obvious map, i.e., 
\begin{equation} \label{transmapfree} \trans_R \colon \bigoplus R[H] x_i \rightarrow \bigoplus Rx_i \colon \sum\limits_i \sum\limits_h k_h h x_i  \mapsto \sum\limits_i \left(\sum\limits_h k_h\right) x_i,\end{equation} 
 cf.\ \cite[(2.3)]{BrownK}. 
 
 One may use ``transfer'' to prove the following (the case of a free action is also in \cite[II.(2.4)]{BrownK}). 

\begin{lemma}[{\cite[III.2.4]{Bredon}}] \label{lem:coinvariants} If $H$ is a finite group of isometries of a closed smooth manifold $M$ with quotient map $$q \colon M \rightarrow H \backslash M,$$ and the order of $H$ is coprime to the characteristic of the field $K$, then the first $K$-homology of the quotient, $\Ho_1(H \backslash M,K)$, is isomorphic to the coinvariants $\Ho_1(M,K)_H$ of the first $K$-homology of $M$, and under this identification, the map $q_*$ that $q$ induces on the first homology groups is the map $\trans_K$ from \textup{(\ref{transmap})},  i.e., we have a diagram
\[ \pushQED{\qed}  \raisebox{\dimexpr\depth-\fboxsep}{\xymatrix@C=9mm@R=2mm{ & \Ho_1(H\backslash M,K) \ar[dd]^{\cong}  \\ \Ho_1(M,K) \ar[ru]^{q_*} \ar[rd]^{\trans_K} & \\ & \Ho_1(M,K)_H}} \qedhere \popQED \]

\end{lemma} 

\subsection{Geometric construction} We refer back to the situation of diagram \eqref{leftparthalfprop}, and keep our assumption that $\F_\ell$ is a field of order coprime to $|G|$. 
By condition \textup{($\ast \ast$)}, we have a decomposition of $\F_\ell[G]$-modules $$\Ho_1(M,\F_\ell) = \mathcal N \oplus V \cong \bigoplus \F_\ell \omega_i \oplus V$$ (for some $\F_\ell[G]$-submodule $V$), where the $G$-action on $\mathcal N$ is given in terms of the permutation of cosets as $g\omega_i=\omega_{g(i)}$, with the convention that $i=1$ corresponds to the trivial $H_1$-coset in $G$.  We also let $$V':=\bigoplus_{i \geq 2} \F_\ell \omega_i$$ denote the vector space complement of $\F_\ell \omega_1$ in $\mathcal N$. 

The quotient map $q_1 \colon M \rightarrow M_1 = H_1 \backslash M$ induces a surjective map $$q_{1*} \colon \Ho_1(M,\F_\ell) \rightarrow \Ho_1(M_1,\F_\ell), $$
and we define $\omega_1':=q_{1*}(\omega_1)$. 

Let $\Gamma_1$ denote the subgroup $\Gamma_1 \leq \Gamma_0$ for which $M_1 = \Gamma_1 \backslash \tilde M$. 

\begin{lemma}  Suppose $\ell$ is coprime to $|G|$ and condition \textup{($\ast$) (equivalently, ($\ast \ast$))} holds.  Then we have a well-defined and commutative diagram:
\begin{equation} \label{eq:Psichi0}
\xymatrix@C=6mm@R=8mm{\Gamma \ar@/^1.5pc/[rr]^{\Psi}
\ar@{->}[d]^-{\iota}\ar@{->>}[r]
& \Ho_1(M,\F_\ell) \ar[d]^{q_{1*}}\ar@{->>}[r]_-{\varphi} & C^n \ar[d]^-{r_1} \\
\Gamma_1  \ar@/_1.5pc/[rr]_{\chi_0}
\ar@{->>}[r] & \Ho_1(M_1,\F_\ell) \ar@{->>}[r]^-{\varphi_0} &
C
}
\end{equation}
where 
\begin{itemize} 
\item $\iota$ is the embedding of $\Gamma$ in $\Gamma_1$;
\item  $r_1 \colon C^n \to C$, $(k_1,\dots,k_n) \mapsto k_1$ is projection onto the first coordinate; 
\item $\varphi_0$ is defined by 
\begin{align*} \varphi_0 \colon \Ho_1(M_1,\F_\ell) \xrightarrow{\cong} 
\F_\ell \omega_1' \oplus W &\rightarrow \F_\ell \cong C \\
k_1 \omega'_1 + w &\mapsto k_1 \qquad (k_1 \in \F_\ell, w \in W).
\end{align*} 
with $W:=q_{1*}(V \oplus V')$ a complementary vector space to $\F_\ell \omega_1'$ in  $\Ho_1(M_1,\F_\ell)$. 
\end{itemize} 
\end{lemma} 

\begin{proof} To see that this is well-defined and the right square commutes, we need that $\omega_1'$ is linear independent of $W=q_{1*}(V \oplus V')$; so suppose that there are $a_1, a_2 \in \F_\ell$ such that $a_1 \omega_1' + a_2 q_{1*}(v) = 0$ for some $v \in V \oplus V'$. This means that \begin{equation} \label{a1a2}  a_1 \omega_1 + a_2 v \in \ker(q_{1*}). \end{equation} By Lemma \ref{lem:coinvariants}, the kernel of $q_{1*}$ is equal to the kernel of $\trans_{\F_\ell}$, and by definition this kernel is spanned by elements $h_1(\omega_i)-\omega_i$ ($i=1,\dots,n$) and $h_1(v)-v$ for $v \in V$ and $h_1 \in H_1$. Now 
\begin{itemize} 
\item for any $h_1 \in H_1 \leq G$, $ h_1(\omega_i)-\omega_i=\omega_{h_1(i)}-\omega_i$; if $i=1$, this element is zero, since that index corresponds to the trivial conjugacy class of $H_1$ in $G$, whereas if $i \neq 1$, this element belongs to $V'$, since then also $h_1(i)>1$; \item since $\mathcal N \oplus V$ is a decomposition as $\F_\ell[G]$-modules, $h_1(v)-v \in V$ for all $v \in V$ and all $h_1 \in H_1$.  
\end{itemize} 
 It follows that $\ker(q_{1*}) \subseteq V \oplus V'$, and by (\ref{a1a2}), $a_1 \omega_1 \in V \oplus V'$. Since $\omega_1$ is linearly independent from $V \oplus V'$, we conclude that $a_1=0$, as desired. This guarantees that if 
 $ \omega = \sum k_i \omega_i + v \in  \Ho_1(M,\F_\ell)$ with $v \in V$ (so $\varphi(\omega) = (k_1,\dots,k_n)$), then 
$ q_{1*}(\omega) = k_1 \omega'_1 + w \in \Ho_1(M_1,\F_\ell)$ with $w \in W$, so 
\[ \varphi_0(q_{1*}(\omega)) = k_1 = r_1(\varphi(\omega)).\qedhere \]
\end{proof}

Just like we defined $\Gamma=\ker \Psi$ in \eqref{herepsi}, we now set 
\begin{equation} \label{herechi} \Gamma_1' := \ker \chi_0 \vartriangleleft \Gamma_1 \mbox{ \ and \ } M_1':= \Gamma_1' \backslash \tilde M \mbox{\ with covering map \ } q_1' \colon M_1' \rightarrow M_1. \end{equation} 
The following lemma describes the relationship between the group $\Gamma=\ker \Psi$ used in Section \ref{sec:covers}, and $\Gamma'_1:=\ker \chi_0$, the group used in this section. 

\begin{lemma} \label{dagger}  Suppose $\ell$ is coprime to $|G|$ and condition \textup{($\ast$) (equivalently, ($\ast \ast$))} holds. The group $\Gamma'=\ker \Psi$ can be expressed in terms of the group $\Gamma_1'=\ker \chi_0$ and a set $\{\bar g_1,\dots,\bar g_n\}$ of lifts of $\{g_1,\dots,g_n\}$ to $\Gamma_0$, as $\Gamma' = \Gamma'_{\mathrm{new}}$, where
\begin{equation} \label{newgamma} \Gamma'_{\mathrm{new}} :=  \bigcap_{i=1}^n \bar g_i \Gamma_1' \bar g_i^{-1} \cap \Gamma= \bigcap_{i=1}^n (\Gamma \cap \bar g_i \Gamma_1' \bar g_i^{-1}) = \bigcap_{i=1}^n \bar g_i (\Gamma \cap \Gamma_1') \bar g_i^{-1}.\end{equation} 
\end{lemma} 

\begin{proof} 
The equalities in (\ref{newgamma}) follow since $\Gamma$ is normal in $\Gamma_0$. It remains to prove $\Gamma'_{\mathrm{new}} = \ker \Psi$. Notice that 
it follows from diagram \eqref{eq:Psichi0} that
\begin{equation} \label{eq:G1'capG}
\Gamma_1' \cap \Gamma =
 \ker \chi_0 \cap \Gamma = \{\gamma \in \Gamma \mid r_1 \circ \Psi(\gamma) = 0 \} = 
\Psi^{-1}( \{0\} \times C^{n-1} ).
\end{equation}
Since $\Psi$ is surjective, this implies
$ \Psi(\Gamma_1' \cap \Gamma) = \{0\} \times C^{n-1}. $
Since by definition 
$$ \Phi(g_i)(\{0\} \times C^{n-1}) = C^{i-1} \times \{0\} \times C^{n-i}, $$
from diagram \eqref{extended_homology_and_fundgroup}, we conclude that
\begin{equation} \label{subways} \Psi(\bar g_i (\Gamma_1' \cap \Gamma) \bar g_i^{-1}) = \Phi(g_i) \Psi(\Gamma_1' \cap \Gamma) = C^{i-1} \times \{0\} \times C^{n-i},  \end{equation} 
and therefore $$\Psi(\Gamma'_{\mathrm{new}}) \subseteq \bigcap\limits_i C^{i-1} \times \{0\} \times C^{n-i} = \{0\}, $$ so $\Gamma'_{\mathrm{new}} \subseteq \ker \Psi$. 

To prove the reverse inclusion, assume that $\Psi(\gamma)=0$ for some $\gamma \in \Gamma$. Then by diagram \eqref{extended_homology_and_fundgroup} we also have $\Psi(\gamma_0^{-1} \gamma \gamma_0) = 0$ for any $\gamma_0 \in \Gamma_0$, so $$\gamma_0^{-1} \gamma \gamma_0 \in \Psi^{-1}(0) \subseteq \Psi^{-1}( \{0\} \times C^{n-1} ) \stackrel{\eqref{eq:G1'capG}}{=} \Gamma_1' \cap \Gamma.$$ 
Therefore 
$\gamma \in \gamma_0 (\Gamma_1' \cap \Gamma) \gamma_0^{-1}$ for all $\gamma_0$, showing that $\gamma \in \Gamma'_{\mathrm{new}}$, so  $\ker \Psi 
\subseteq \Gamma'_{\mathrm{new}}$. 
\end{proof}

\begin{remark}
Standard expressions for the kernel of the restriction and induction of representations (see, e.g., \cite[Lemma 5.11]{Isaacs}) allow one to give a representation-theoretic description of $\Gamma'_{\mathrm{new}}$. Namely, let $\tilde \chi_0$ denote the linear character 
on $\Gamma_1$ given by $\tilde \chi_0(\gamma)= e^{2\pi i \chi_0(\gamma)/\ell}$ where $\chi_0$ is as in diagram \eqref{eq:Psichi0}.   Then, with $\ker \tilde \chi_0 = \ker \chi_0 =  \Gamma_1'$, we have 
$$ \ker \Res_\Gamma^{\Gamma_0} \Ind_{\Gamma_1}^{\Gamma_0} \tilde \chi_0  = \Gamma \cap \ker \Ind_{\Gamma_1}^{\Gamma_0} \tilde \chi_0  = \Gamma \cap \bigcap_{\gamma_0 \in \Gamma_0} \gamma_0 \ker(\tilde \chi_0) \gamma_0^{-1} = \Gamma'_{\mathrm{new}}. $$
\end{remark}

\medskip

\noindent We now perform the following 2-step geometric construction: 
\begin{enumerate}[leftmargin=1.7em, label=(\alph*)]
\item For $g \in G$,  ``twist'' the cover $q_1 \colon M \rightarrow M_1$ by defining $q_1^g \colon M \rightarrow M_1$ by $x \mapsto q_1(g^{-1}x)$, and set 
$$M_g'':= M_1' \times_{M_1,q_1^g} M; $$
corresponding to the following diagram:
\begin{equation} \label{diag:Mg''}
  \xymatrix@C=5mm@R=4mm{ & M''_g \ar[dl] \ar[dr] & &  \\ M_1' \ar[dr]_{C}^{q_1'} & & M \ar[dl]_-{q_1^g}\ar[r]^{g^{-1} \cdot}  & M \ar[dll]^{q_1} \\ & M_1 & }
\end{equation} 
By slight abuse, the two different $M$ in the diagram are in fact identical, but the maps are different. 
\item  Iteratively construct the fiber product  
\begin{equation} \label{fpmp} M'_{\mathrm{new}}  := M''_{g_1} \times_M M''_{g_2} \times_M \dots \times_M M''_{g_n}, \end{equation} 
where $\{g_1,\dots,g_n\}$ is the chosen set of representatives for $G/H_1$; this is presented in the following diagram:
\begin{equation} \label{subcoverdiag} 
  \xymatrix@C=5mm@R=5mm{ & & M'_{\mathrm{new}}  \ar[dll] \ar[dl] \ar[drr] & & \\ M''_{g_1} \ar[drr]_C & M''_{g_2} \ar[dr]^(.55)C & & \cdots & M''_{g_n} \ar[dll]^C \\ & & M & & }
\end{equation} 
\end{enumerate} 
We will prove that this manifold $M'_{\mathrm{new}} $ is the same as $M'$, the one constructed in the previous section. 

\begin{proposition} \label{propgeomact} Suppose $\ell$ is coprime to $|G|$ and condition \textup{($\ast$) (equivalently, ($\ast \ast$))} holds. 
\begin{enumerate} 
\item The fiber product $M'_{\mathrm{new}}$ in \eqref{fpmp} is represented as
\begin{equation} \label{eq:M'}
M'_{\mathrm{new}}  = \left\{ (x_1,\dots,x_n,x) \in M_1' \times \cdots M_1' \times M \mid
q_1'(x_i) = q_1(g_i^{-1}x), \ i =1,\dots,n \right\}, 
\end{equation}
and in these coordinates, the projection $M'_{\mathrm{new}}  \to M''_{g_i}$ is given
by
$$ M'_{\mathrm{new}}  \ni (x_1,x_2,\dots,x_n,x) \mapsto (x_i,x) \in M''_{g_i}. $$
$M'_{\mathrm{new}}$ is a connected manifold and corresponds to the subgroup $\Gamma'_{\mathrm{new}}$, so that in fact $M'_{\mathrm{new}} = M'$.
\item Geometrically, the action of $\tilde G$ on $M'_{\mathrm{new}} $ is expressed as follows in the coordinates used in \eqref{eq:M'}: there exists an isometry $\iota  \colon M'_{\mathrm{new}} \rightarrow M'_{\mathrm{new}} $ that conjugates the action of $\tilde G$ into
\begin{itemize}
\item $\underline{c}=(c_i) \in C^n \le \tilde G$ acts  componentwise on each factor $M''_{g_i}$, i.e., 
\begin{equation} \label{coaxc} \iota^{-1} \underline{c} \iota \cdot (x_1,x_2,\dots,x_n,x) = (c_1 x_1,\dots,c_n x_n,x); \end{equation}  
\item $g \in G \le \tilde G$ acts on $M'_{\mathrm{new}} $ by 
\begin{equation} \label{coax} \iota^{-1} g \iota \cdot (x_1,x_2,\dots,x_n,x) = (x_{g^{-1}(1)},x_{g^{-1}(2)},\dots,x_{g^{-1}(n)},gx), \end{equation} 
where $g^{-1}(i)$ is defined, as before, via $g^{-1} g_i \in g_{g^{-1}(i)}H_1$. Colloquially, this means that, up to an isometry, in diagram \eqref{subcoverdiag}, $g$ act naturally on the ``base'' manifold $M$, while the points in the various $M''_{g_j}$ above a given point in $M$ are permuted across these different manifolds in the same way as $g^{-1}$ permutes the cosets $G/H_1$. 
\end{itemize}
\end{enumerate}
\end{proposition} 

\begin{proof} 
We follow the steps of the construction. Since the group homomorphism $\chi_0 \colon \Gamma_1 \rightarrow C$ in diagram \eqref{eq:Psichi0} is surjective, $\Gamma_1' := \ker \chi_0 \vartriangleleft \Gamma_1$
is of index $\ell$ in $\Gamma_1$, and $ q_1' \colon M_1' \rightarrow M_1$ is a $C$-Galois cover. 
\begin{enumerate}[leftmargin=1.7em, label=(\alph*)] 
\item 
Since $M_1$ is a manifold, the compositum is described as 
$$ M_g'' = \{ (x_1,x) \in M'_1 \times M \colon q_1'(x_1) = q_1(g^{-1} x) \}. $$ 
Since the degrees of the covers $q_1' \colon M_1' \rightarrow M_1$ and $q_1 \colon M \rightarrow M_1$ are coprime, Lemma \ref{coprime} implies that $M''_g$ is connected and equal to the compositum. 
As in \eqref{isomconj}, the action of $g^{-1}$ on $M_0$ and $M_1$ corresponds to the action on $\Gamma_0$ and the subgroup $\Gamma_1'$ by conjugation with $\bar g^{-1}$, where $\bar g$ is a lift of $g$ to $\Gamma_0$.  Hence the corresponding group is the intersection $\Gamma_g'':=\Gamma \cap \bar g \Gamma_1' \bar g^{-1}$, i.e., $M_g''=\Gamma_g'' \backslash \tilde M$.   By Lemma \ref{lindis} and coprimality of the degree, the covering $M_g'' \rightarrow M$ is $C$-Galois.

\item Since $M$ is a manifold, the underlying set of the fiber product is indeed the set theoretic fiber product in \eqref{eq:M'}. 
We next argue that 
$M'_{\mathrm{new}} $ is connected, agrees with the compositum, and indeed corresponds to the group  
$\Gamma'_{\mathrm{new}}$ (and hence $\Gamma'$) in \eqref{newgamma}, i.e., $M'_{\mathrm{new}}=M'$. This will finish the proof of (i). 
To see the connectedness,  
we use induction with respect to the number of factors. So suppose we have already proven that $M_{g_1}'' \times_M \dots M_{g_{N-1}}'' \rightarrow M$ is a connected $C^{N-1}$-cover corresponding to the group $\bigcap\limits_{i=1}^{N-1} \bar g_i (\Gamma \cap \Gamma_1') \bar g_i^{-1}$. By Lemma \ref{coprime}, the product with the next factor $M_{g_N}''$ is connected if and only if 
\begin{equation} \label{subways2}  \Gamma = \langle  \bigcap_{i=1}^{N-1} \bar g_i (\Gamma \cap \Gamma_1') \bar g_i^{-1}, \bar g_N (\Gamma \cap \Gamma_1') \bar g_N^{-1} \rangle. \end{equation} 
To prove this, we notice that is is true after applying $\Psi$, using \eqref{subways}: the image of left hand side is $C^n$, and the image of the right hand side is the subgroup of $C^n$ spanned by $\{0\}^{N-1} \times C^{n-N}$ and  $C^{N-1} \times \{0\} \times C^{n-N}$, which equals the whole of $C^n$. Hence equation \eqref{subways2} is true up to $\ker \Psi$, and from Lemma \ref{dagger}, it follows that $\ker \Psi$ is contained in both the left hand side and the right hand side of the equality, proving that \eqref{subways2} holds on the nose. 
 \end{enumerate}
 To prove (ii), note that $M'_{\mathrm{new}} \to M$ is a $C^n$-Galois cover by Lemma \ref{lindis}, with one copy of $C$ acting componentwise on each factor $M''_{g_i}$, and this is the same as the action of $C^n$ on $M'$. The claim about the action of $g \in G \leq \tilde G$ can be proven as follows: the action of $G$ on $M'$ is given by considering $G$ as a subgroup of $\tilde G$, and as such it acts by isometries on $M_{\mathrm{new}} = M'$. We know that, in the geometric representation \eqref{eq:M'} for $M'_{\mathrm{new}}$, 
 $$g \underline{x} = \underline{y} \mbox{ with } \underline{x}=(x_1,\dots,x_n,x), \underline{y}= (y_1,\dots,y_n,y)$$ for some \emph{unique} $y_i \in M_1'$ and $y \in M$ with $q_1'(y_i) \stackrel{\mathrm{(I)}}{=} q_1(g_i^{-1}y).$ We only need to determine what $y_i$ and $y$ are. Since the action of $G$ on $M$ is as given, $y\stackrel{\mathrm{(II)}}{=}gx$. Recall also that $g^{-1} g_i  = g_{g^{-1}(i)}h_{g,i}$ for some $h_{g,i} \in H_1$. In particular, with $q_1 \colon M \rightarrow M_1$ the covering with group $H_1$, for $x \in M$ we have $q_1((g^{-1}g_i)^{-1} x)  \stackrel{\mathrm{(III)}}{=}q_1 (g_{g^{-1}(i)}^{-1} x).$ We collect this information to compute 
 \[ q_1'(y_i)  \stackrel{\mathrm{(I)}}{=} q_1(g_i^{-1}y) \stackrel{\mathrm{(II)}}{=}q_1(g_i^{-1} g x) = q_1((g^{-1}g_i)^{-1} x) = q_1 (g_{g^{-1}(i)}^{-1} x) \stackrel{\mathrm{(III)}}{=}q_1'(x_{g^{-1}(i)}). \]
Since $q_1' \colon M_1' \rightarrow M_1$ is a $C$-cover, this shows that $y_i = c_i x_{g^{-1}(i)}$ for some $\underline{c} = (c_i) \in C^n$, that a priori depends on $g$ and $\underline{x}$, i.e., it is a map $$\underline{c} \colon G \times M'  \rightarrow C^n.$$
Let us now prove that it does not depend on any of these parameters. Note that we write the group operation on $C^n$ multiplicatively. 

First of all, denote the dependency on $\underline{x}$ by $\underline c(\underline{x})$. Let $d(\cdot,\cdot)$ denote the distance on a manifold induced from the Riemannian metric. Since $C$ acts properly discontinuously on $M_1'$ there is a $\delta>0$ such that, for any two elements $c,c' \in C$ and $x \in M_1'$, if $d(cx,c'x)<\delta$, then $c'=c$. 
If $\underline{x}'$ is at distance $\varepsilon$ from $\underline{x}$ in $M'$, then so is $g \underline{x}$ from $g \underline{x}'$, and hence so is $c_i(\underline{x}) x_{g^{-1}(i)}$ from $c_i(\underline{x}') x'_{g^{-1}(i)}$ for all $i$.  Hence 
\begin{align*} d(c_i(\underline{x})&x_{g^{-1}(i)}, c_i(\underline{x}')x_{g^{-1}(i)}) \\  &\leq d(c_i(\underline{x})x_{g^{-1}(i)}, c_i(\underline{x}')x'_{g^{-1}(i)}) + d(c_i(\underline{x}')x'_{g^{-1}(i)}, c_i(\underline{x}')x_{g^{-1}(i)}) \\ &= d(c_i(\underline{x})x_{g^{-1}(i)}, c_i(\underline{x}')x'_{g^{-1}(i)}) + d(x'_{g^{-1}(i)}, x_{g^{-1}(i)}) \leq 2 \varepsilon, \end{align*} (the equality in the above formula holds since $c_i(\underline{x}')$ is an isometry) and thus 
$c_i(\underline{x}) = c_i(\underline{x}')$ as soon as $\underline{x}$ and $\underline{x}'$ are at distance $<\delta/2$. We conclude that $\underline c(\underline{x})$ is locally constant in $\underline{x}$, and since $M'$ is connected, $\underline{c}$ is actually independent of $\underline{x}$. so that we have a map  \begin{equation} \label{coc} \underline{c} \colon G \rightarrow C^n. \end{equation}   
Now denote the dependence on $g$ by $\underline{c}(g)$; we observe that for two elements $g,h \in G$, 
\begin{align*}  (c_i(gh) x_{(gh)^{-1}(i)}, gh x) & = gh \underline{x} = g(c_i(h) x_{h^{-1}(i)}, hx) \\ &= (c_i(g) c_{g^{-1}(i)}(h) x_{h^{-1} g^{-1}(i)}, gh x), \end{align*}
so $\underline{c}(gh) = \underline{c}(g) \underline{c}(h)^g$, where the action of $g$ on $\underline{c}=(c_i)$ is given by $\underline{c}^g:=(c_{g^{-1}(i)})$.  This shows that the map $\underline{c}$ in equation (\ref{coc}) is a cocycle from $G$ to $C^n$, and the corresponding first group cohomology class lies in $\Ho^1(G,C^n)$. Since $|G|$ and $|C^n|=\ell^n$ are coprime, the latter cohomology group is zero \cite[III.(10.1)]{BrownK}, proving that $\underline{c}$ is a coboundary, i.e., there exists $\underline{v} \in C^n$ (independent of $g$) such that $\underline{c}(g) = \underline{v}^{-1} \underline{v}^g = (v_i^{-1} v_{g^{-1}(i)})$. 
Consider the isometry $$\iota \colon M'_{\mathrm{new}} \rightarrow M'_{\mathrm{new}}  \colon \underline{x} = (x_i,x) \mapsto (v_i^{-1} x_i, x). $$
Now $$\iota^{-1} g \iota(\underline{x}) = \iota^{-1} (c_i(g) v_{g^{-1}(i)}^{-1} x_{g^{-1}(i)}, gx) = \iota^{-1} (v_i^{-1} x_{g^{-1}(i)}, gx) = (x_{g^{-1}(i)},gx),$$ 
as was claimed. Note also that conjugating by $\iota$ commutes with the action of $C^n$, so it does not change that action.  
  \end{proof} 
  
  \begin{remark} 
 Assuming condition $(\ast \ast)$, one can now give the following alternative construction of diagram \eqref{leftpartextend} used in the main theorem \ref{maindetail}: perform the above two step construction of $M'_{\mathrm{new}}$ and \emph{define} an action of $\tilde G$ on $M'_{\mathrm{new}}$ using the right hand side of equations \eqref{coaxc} and \eqref{coax}.  
  \end{remark} 

\begin{remark} The action of $\tilde G$ on $M_{\rm{new}}'$ ties up with the group theoretical construction from the previous section, as follows.
The group $\tilde G \cong \Gamma_0 / \Gamma'$ acts naturally on $M' = \Gamma' \backslash \tilde M$ via 
\begin{equation} \label{eq:action}
(\gamma_0 \Gamma') \cdot (\Gamma'\tilde x) = \Gamma' \cdot (\gamma_0 \tilde x).
\end{equation}
The explicit identification between $M'$ and $M_{\rm{new}}'$ is given by the map
\begin{equation} \label{eq:iso} 
M' = \Gamma' \backslash \tilde M \ni \Gamma' \tilde x \mapsto (\Gamma_1' \bar g_1^{-1}\tilde x,\dots,\Gamma_1'\bar g_n^{-1} \tilde x,\Gamma \tilde x) =: (x_1,\dots,x_n,x) \in M_{\rm{new}}', 
\end{equation}
where $\{ g_i H_1 \}$ represent the cosets of $H_1$ in $G$ and
$G \to \Gamma_0: g \mapsto \bar g$ is a section such that we have $\bar e_G = e_{\Gamma_0}$, $\overline{g^{-1}} = \bar g^{-1}$ and $\bar g \Gamma' = \jmath(g)$ with the homomorphism $\jmath: G \to \Gamma_0 / \Gamma'$ representing the splitting of \eqref{eq:ses}. The action of $\tilde G \cong \Gamma_0/\Gamma'$, transferred from $M'$ to $M_{\rm{new}}'$ is then
\begin{equation} \label{eq:action2}
(\gamma_0 \Gamma') \cdot (\Gamma_1' \bar g_1^{-1}\tilde x,\dots,\Gamma_1'\bar g_n^{-1} \tilde x, \Gamma \tilde x) = (\Gamma_1'\bar g_1^{-1} \gamma_0 \tilde x,\dots,\Gamma_1'\bar g_n^{-1}\gamma_0 \tilde x, \Gamma \gamma_0 \tilde x). 
\end{equation}

Let $c \in \Gamma$ be an element satisfying $\Psi(c) = e_1$
and set $c_i := \bar g_i c \bar g_i^{-1} \in \Gamma$, as in Section \ref{constrmp}. Utilising the diagrams \eqref{extended_homology_and_fundgroup} and \eqref{eq:Psichi0}, we see that $c_i \in \bar g_j \Gamma_1' \bar g_j^{-1}$ for all $j \neq i$. Thus,  \eqref{eq:action2} implies 
\begin{eqnarray}
(c_i \Gamma') \cdot (x_1,\dots,x_n,x) &=& (\Gamma_1'\bar g_1^{-1} c_i \tilde x,\dots,\Gamma_1'\bar g_n^{-1} c_i \tilde x, \Gamma c_i \tilde x) \nonumber \\
&=& (\Gamma_1'\bar g_1^{-1} \tilde x, \dots, \underbrace{\Gamma_1' c \bar g_i^{-1} \tilde x}_{\text{$i$-th entry}}, \Gamma_1'\bar g_n^{-1} \tilde x, \Gamma \tilde x).\label{eq:c-act} 
\end{eqnarray}
Let $g \in G$; by definition, we have $\bar g^{-1} \bar g_i
\Gamma_1' = \bar g_{g^{-1}(i)} c^{- k_i(g)} \Gamma_1'$ for some $k_i(g)$ modulo $\ell$. This implies that
\begin{eqnarray*} 
(\bar g \Gamma') \cdot (x_1,\dots,x_n,x) &\stackrel{\eqref{eq:action2}}{=}& (\Gamma_1' \bar g_1^{-1} \bar g \tilde x, \dots, \Gamma_1' \bar g_n^{-1} \bar g \tilde x, \Gamma \bar g \tilde x) \\
&=& (\Gamma'_1 c^{k_1(g)} \bar g_{g^{-1}(1)} \tilde x,\dots,\Gamma'_1 c^{k_n(g)} \bar g_{g^{-1}(n)} \tilde x, \Gamma \bar g \tilde x) \\
&\stackrel{\eqref{eq:c-act}}{=}& (c_1^{k_1(g)} \cdots c_n^{k_n(g)}\Gamma') \cdot 
(x_{g^{-1}(1)},\dots,x_{g^{-1}(n)},g \cdot x).
\end{eqnarray*}
 Now $g \mapsto (c_i^{k_i(g)}\Gamma')_{i=1}^n$ is a cocycle from $G$ to $C^n = \Gamma/\Gamma'$, and since $\Ho^1(G,C^n)=0$, there exists $(m_1,\dots,m_n)$ with $k_i(g) = m_{g^{-1}(i)}-m_i$ (modulo $\ell$). 
Using the commutativity of the elements $c_i \Gamma'$, this implies that if we set 
$c_0 := \prod_{i=1}^n c_i^{m_i}$, then 
$$ (c_0 \jmath(g) c_0^{-1}) \cdot (x_1,\dots,x_n,x) = (x_{g^{-1}(1)},\dots,x_{g^{-1}(n)},g \cdot x) $$
for all $g \in G$ and all $(x_1,\dots,x_n,x) \in M_{\rm{new}}'$. 
This shows that a copy of $G$ in $\tilde G \cong \Gamma_0 / \Gamma'$, namely $c_0 \jmath(G) c_0^{-1}$, acts on $M_{\rm{new}}'$ via permutation of the first $n$ entries. In other words, it is possible to conjugate the subgroup $G$ in $\tilde G$ to realise the specific action \eqref{coax} on $M_{\rm{new}}'$.
\end{remark} 

\subsection{Universal property of the wreath product} \label{wreathuniversal} The appearance of the wreath product in our constructions becomes less of a surprise given the following universal property, showing that the minimal Galois cover that ``contains'' a $G$-cover and a $C$-cover as in our situation arises from this wreath product (the analogous result in the theory of field extensions is well-known, compare \cite[13.7]{FJ}). 

\begin{proposition} 
Let $G$ and $C$ denote finite groups with $C$ cyclic of prime order $\ell$ not dividing the order of $G$. Suppose that we are given Riemannian manifolds $M,M_1,M_1'$ and a developable Riemannian orbifold $M_0$ such that $M \rightarrow M_0$ is $G$-Galois with subcover $M_1 \rightarrow M_0$, and $M_1' \rightarrow M_1$ is $C$-Galois. If $N \rightarrow M_0$ is a Galois cover of minimal degree admitting Riemannian covers $N \rightarrow M$ and $N \rightarrow M_1'$, then the Galois group $G'$ of $N$ over $M_0$ is the wreath product $\tilde G:=C^n \rtimes G$, where $n$ is the degree of the cover $M_1 \rightarrow M_0$ \textup{(}see Figure \eqref{universalwreath}.\textup{)}
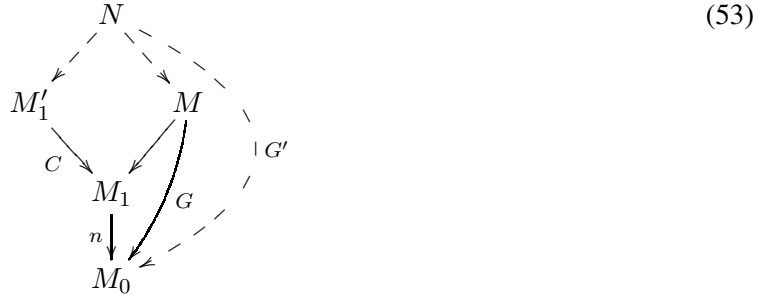
\begin{equation} \label{universalwreath}
  \xymatrix@C=3mm@R=6mm{ & N \ar@{-->}[dr] \ar@{-->}[dl] \ar@{-->}@/^4.5pc/[ddd]^{G'} & \\ M_1' \ar[dr]_C & & M \ar[dl] \ar@/^0.5pc/[ddl]^-G  \\ & M_1\ar[d]_n &   \\ & M_0 &  } 
\end{equation} 
\end{proposition} 

\begin{proof} Writing the manifolds $M_0,M,M_1,M_1',N$ as quotients of he universal cover $\tilde M_0$ of $M_0$ by the respectively group $\Gamma_0,\Gamma, \Gamma_1, \Gamma_1', \Gamma_N$, the defining properties of $N$ imply that it is the normal closure of the compositum of $M_1'$ and $M$ over $M_0$, and hence $$\Gamma_N = \bigcap_{\gamma_0 \in \Gamma_0} \gamma_0 ( \Gamma \cap \Gamma_1') \gamma_0^{-1}.$$ 
First of all, for $g \in G$, choose one element $\bar g \in \Gamma_0$ that maps to $g \in \Gamma_0/\Gamma \cong G$. 
We claim that $$\Gamma_N = \bigcap_{i=1}^n \Gamma_{g_i} \mbox{ where } \Gamma_{g_i} := \bar g_i ( \Gamma \cap \Gamma_1') {\bar g_i}^{-1}, $$
for $\{g_i\}$ a set of coset representatives for $H_1$ in $G$. 
Indeed, for any $\gamma_0 \in \Gamma_0$ we can write
$\gamma_0 = \bar g_i \gamma_1$ for some $i \in \{1,\dots,n\}$ and some $\gamma_1 \in \Gamma_1$, since the cosets of
$H_1 \cong \Gamma_1/\Gamma$ in $G \cong \Gamma_0/\Gamma$ are $g_1 H_1,\dots, g_n H_1$ and the cosets of $\Gamma_1$ in $\Gamma_0$ are therefore $\bar g_1 \Gamma_1,\dots, \bar g_n \Gamma_1$. The statement now follows  from the fact that $\gamma_0 \in \Gamma_0$ must lie in one of these cosets $\bar g_i \Gamma_1$; since both $\Gamma$ and
$\Gamma_1'$ are normal in $\Gamma_1$, we have
\begin{align*}  \gamma_0 (\Gamma \cap \Gamma_1' ) \gamma_0^{-1} & =
(\bar g_i \gamma_1) (\Gamma \cap \Gamma_1') (\bar g_i \gamma_1)^{-1} = \bar g_i \gamma_1 (\Gamma \cap \Gamma_1') \gamma_1^{-1} \bar g_i^{-1} \\ &= \bar g_i (\Gamma \cap \Gamma_1') \gamma_1 \gamma_1^{-1} \bar g_i^{-1} =
\bar g_i (\Gamma \cap \Gamma_1') \bar g_i^{-1}. \end{align*}

Now since $\Gamma$ is normal in $\Gamma_0$, $\Gamma \geq \Gamma_N$, and we find an exact sequence $$1 \rightarrow \Gamma/\Gamma_N \rightarrow \Gamma_0 / \Gamma_N \rightarrow \Gamma_0 / \Gamma \cong G \rightarrow 1.$$ The natural map $\varphi \colon \Gamma \rightarrow \prod\limits_{i=1}^n \Gamma / \Gamma_{g_i}$ has  kernel  $\bigcap\limits_{i=1}^n \Gamma_{g_i} = \Gamma_N$. 
Next, $\Gamma / \Gamma_{g_i} \cong C$ since the index is the prime number $\ell$. Finally, we claim that $\varphi$ is surjective. For this, it suffices to find for every $i$ an element $\gamma_i \in \Gamma$ with $$\varphi(\gamma_i) = e_i=(0,\dots,0,1,0,\dots,0) \in C^n.$$  Since $C$ is cyclic of prime order, every non-zero element is a generator, and it suffices to choose $\gamma_i \in  ( \bigcap\limits_{j \neq i}\Gamma_{g_j}) \setminus \Gamma_{g_i}$. This is possible since the reasoning in the first paragraph of this proof shows that the latter set is non-empty. In the end, we find a sequence 
$$ 1 \rightarrow C^n \rightarrow \Gamma_0/\Gamma_N \rightarrow G \rightarrow 1 $$
where $G$ acts on $C^n$ by permuting the factors like it permutes the cosets of $H_1$, and this finishes the proof. 
\end{proof}

\begin{remark} 
In our setup, the universality property says the following: if we search for the ``easiest possible'' twisted Laplace operator on $M_1$, meaning associated to the Laplace operator on some prime order cyclic cover of $M_1$, we necessarily arrive at a diagram of the form \eqref{leftpartextend}. 
\end{remark}

\section{Homological wideness} \label{homwide} 
\subsection{The notion of homological wideness} We now introduce an easier topological notion than conditions  \textup{($\ast$)} from Theorem \ref{maindetail} and  \textup{($\ast \ast$)} from Subsection \ref{astastsection},  that involves only the group $G$ and not the subgroup $H_1$. 

\begin{definition} 
Suppose $G$ is a finite group $G$ acting (freely or not) on a closed connected (topological) manifold $M$. Let $K$ denote a field. We say the action of $G$ is \emph{$K$-homologically wide} if the \emph{$K$\hyp{}homology representation $h_K=h$} of $G$, given by the induced action on the first homology group  \begin{equation} \label{psi} h_K \colon G \rightarrow \Aut(\Ho_1(M,K)) \end{equation} contains the regular representation of $G$. 
\end{definition} 

The regular representation is cyclic (any element $g \in G$ is a cyclic vector, i.e., the vector space span of the orbit $G\cdot g$ spans the entire representation space), hence another way to formulate homological wideness is as follows:  \emph{the action of $G$ on $M$ is homologically wide if and only if there exists a class $\omega \in \Ho_1(M,K)$ such that the orbit $G \cdot \omega$ spans a vector space of dimension $|G|$ inside $\Ho_1(M,K)$.}
Indeed, if the regular representation is contained in the homological representation, just take a cyclic vector for that subrepresentation. Conversely, if such $\omega$ exists, then all $g \cdot \omega$ for $g \in G$ are linearly independent, and hence span a copy of the regular representation. 

\begin{lemma} \label{hwstars} 
If the action of $G$ on $M$ is $\F_\ell$-homologically wide, then condition \textup{($\ast \ast$)}, and hence condition  \textup{($\ast$)} for $\ell$ coprime to $|G|$, holds for any subgroup $H_1$ in $G$. 
\end{lemma} 

\begin{proof} 
It suffices to show that for any $H_1 \leq G$, $\Ind_{H_1}^G \one$ is a subrepresentation of the regular representation. For that, it suffices to prove that the multiplicity of any irreducible $G$-representation $\rho$ in $\Ind_{H_1}^G \one$ is less than or equal to $\dim \rho$, the multiplicity of $\rho$ in the regular representation. By Frobenius reciprocity, we compute that the multiplicity of $\rho$ in $\Ind_{H_1}^G \one$ is
$$ \langle \rho, \Ind_{H_1}^G \one \rangle = \langle \Res_{H_1}^G \rho, \one \rangle = \frac{1}{|H_1|} \sum_{h_1 \in H_1} \tr(\rho(h_1)) \leq \frac{1}{|H_1|} |H_1| \dim \rho, $$
since the trace of $\rho(h_1)$ is a sum of $\dim \rho$ roots of unity (as $\rho(h_1)$ is of finite order).  
\end{proof}

\subsection{The notion of $\Q$-homological wideness} We first relate $\Q$-homological wideness (that has a transparent geometric meaning in terms of cycles) to $\F_\ell$-homological wideness (that is used in the proof of the main result). If the action of $G$ on $M$ is $\Q$\hyp{}homologically wide, there exists a non-torsion homology class $\omega \in \Ho_1(M,\Q) = \Ho_1(M,\Z)\otimes \Q$ such that $\{g \omega \}$ is linearly independent over $\Q$. Fix an integer $N$ such that $\omega':=N\omega \in \Ho_1(M,\Z)$; then $\{ g \omega'\}$ is a set of $\Z$-independent non-torsion homology classes for $M$. These classes will remain linearly independent modulo infinitely many $\ell$. In fact, we can use representation theory to say more. 

\begin{lemma} \label{fromQtoell}
If the action of $G$ on $M$ is $\Q$\hyp{}homologically wide, then it is $\F_\ell$\hyp{}homologically wide for all $\ell$ coprime to $|G|$. 
\end{lemma} 

\begin{proof}This follows from the basic theory of modular representations in ``good'' characteristics. In this proof, we write ``$R\mathrm{-mod}$'' for the category of finitely generated modules over a ring $R$.
 
If $\mathcal M$ is a $\Q[G]$-module, then it is a $\Q_\ell[G]$-module (with $\Q_\ell$ the field of $\ell$-adic numbers). 
We let $K \supseteq \Q_\ell$ denote a splitting field for all irreducible representations of $G$; it suffices to assume that $K$ contains all $m$-th roots of unity where $m$ runs over all orders of elements of $G$. Let $R \supseteq \Z_\ell$ denote the ring of integers of $K$ and $\mathfrak m$ its maximal ideal with residue field $k:=R/\mathfrak m$. 

Fixing any $R[G]$-lattice $\mathcal M'$ in $\mathcal M$, we have a reduction map modulo $\mathfrak m$, producing a $k[G]$-module $\overline{\mathcal M'} = \mathcal M' \otimes k = \mathcal M'/\mathfrak m \mathcal M'$. The \emph{decomposition map} $$ d \colon K_0(K[G]\mathrm{-mod}) \rightarrow K_0(k[G]\mathrm{-mod} ) \colon [\mathcal M] \rightarrow [\overline{\mathcal M'}] $$ 
is an isomorphism  for $\ell$ coprime to $|G|$ and an effective map (i.e., positive integral combinations map to positive combinations) (see, e.g., \cite[\S 15.5]{SerreRep}; another formulation says that if $\ell$ is coprime to $|G|$, the Brauer character of the reduction of a $G$-representation modulo $\mathfrak m$ equals the character of the original representation, see, e.g., \cite[Thm.\ 15.8]{Isaacs}). 
By assumption, any irreducible ($K$-)representation of $G$ occurs as direct summand in $\Ho_1(M,\Z) \otimes K$, and hence also every $k$-irreducible representation occurs as direct summand in $\Ho_1(M,\Z) \otimes k$. Since the regular representation $\Q[G]$ is defined over $\Q$, we also find the regular representation $\F_\ell[G]$ as direct summand in $\Ho_1(M,\Z) \otimes \F_\ell$. 
\end{proof} 

\begin{example} For $G={\Z}/{2{\Z}} = \langle \left( \begin{smallmatrix} 1 & 0 \\ 3 & -1 \end{smallmatrix} \right) \rangle$ acting on $\mathcal M = \Z^2$, $(1,0)$ is a cyclic vector over $\Q$ but not over $\F_3$ (so $\ell=3$ is excluded by the reasoning before the lemma). The lemma does not state that an integral cyclic vector is a cyclic vector modulo $\ell$; just that if one exists, then one exists modulo $\ell$, as long as $\ell$ is coprime to $|G|$. In the example, $(1,1)$ is a cyclic vector over both $\Q$ and $\F_3$. 
\end{example} 

\section{Examples of homologically wide actions} \label{exhomwide} 

We now turn to study various examples of (non-)homologically wide actions. 

\subsection{Surfaces} In dimension $2$, the situation for a fixed-point free action is clear because the homology representation can be computed using the Lefschetz fixed point theorem \cite[Theorem 2C.3]{Hatcher}. 

\begin{proposition} \label{prop:homwidecritriemsurf}
A fixed-point free action of a non-trivial finite group $G$ on a closed orientable surface $M$ is $\Q$\hyp{}homologically wide if and only if $M$ is hyperbolic (i.e., has negative Euler characteristic). In particular, the property is independent of $G$.  
\end{proposition} 

\begin{proof} 
We will compute the character $\chi_h$ of the rational homology representation $$h=h_{\Q} \colon G \rightarrow \Ho_1(M,\Q).$$ First of all, since we assume any $g \neq e$ has no fixed points, the map $g \colon M \rightarrow M$ has Lefschetz number
$$ \tr(g_* | \Ho_0(M,\Q)) -  \tr(g_* | \Ho_1(M,\Q)) +  \tr(g_* | \Ho_2(M,\Q)) = 0.$$ Since the action of $g_*$ on $\Ho_0(M,\Q) \cong \Q$ and $\Ho_2(M,\Q) \cong \Q$ is trivial (induced by the action of $g$ on the space of connected components of $M$, respectively the $2$-cells, i.e., the one-element sets), both outer terms in this expression are $1$, and the 
middle term is $\chi_h(g)$ by definition, so we find that 
$ \chi_h(g) = 2 $ for $g \neq e$. For $g=e$, on the other hand, we get directly from the definition of the character that $ \chi_h(e) =  \tr(e_* | \Ho_1(M,\Q)) =  \dim \Ho_1(M,\Q)=b_1(M) = 2-\chi_M. $ We conclude that 
$$ \chi_h(g) = \left\{ \begin{array}{ll} 2 & \mbox{if } g \neq e, \\ 2 - \chi_M & \mbox{if } g=e. \end{array} \right. $$
On the other hand, the character of the regular representation is 
$$\chi_{{G,\mathrm{reg}}} = \left\{ \begin{array}{ll} 0 & \mbox{if } g \neq e, \\ |G| & \mbox{if } g=e. \end{array} \right. $$
We can match these expressions, and since representations are isomorphic if and only if their characters are equal, we find 
$$ h = 2 \cdot \one_G - \frac{\chi_M}{|G|}  \cdot \rho_{G,\mathrm{reg}}, $$ and 
hence the regular representation occurs inside $h$ if and only if $\chi_{M}<0$.  
\end{proof} 

The above proposition above has no (constant) curvature assumption. In constant curvature but with more general actions, we have the following. 

\begin{proposition} \label{prop:homwidecritriemorb}
Any action of a finite group $G$ by (not necessarily fixed-point free) conformal automorphisms on a closed  Riemann surface $M$ is $\Q$\hyp{}homologically wide if $\chi_{G \backslash M}<0$.  
\end{proposition} 

\begin{proof} 
We rely on the computation of the character of $h$ in this more general ramified setting by Broughton in \cite[Prop.\ 2(iii)]{Broughton}, using in addition (taking into account the holomorphic structure) the Eichler trace formula. Suppose that the $G$-cover is branched above $t$ points. For each branch point, choose a lift to the cover and let $C_i \leq G$ denote the (cyclic) stabiliser of that lift (the stabilisers of any lift of a given point are conjugate in $G$).  Then   
\begin{align*} 
 h 
& = 2 \cdot \one - \chi_{G \backslash M} \cdot \rho_{G,\mathrm{reg}} + \sum_{i=1}^t  (\rho_{G,\mathrm{reg}}- \Ind_{C_i}^G \one ) \\ 
& 
= 2 \cdot \one - \chi_{G \backslash M} \cdot \rho_{G,\mathrm{reg}} + \sum_{i=1}^t   \Ind_{C_i}^G (\rho_{C_i,\mathrm{reg}}-\one ), 
\end{align*} 
where we have used that the induced representation of the regular representation of $C_i$ to $G$ is the regular representation of $G$. 
Since $\one$ occurs in $\rho_{C_i,\mathrm{reg}}$, the representations occurring in the sum are not virtual (i.e., every irreducible representation of $G$ occurs in it with non-negative multiplicity) and therefore $h$ contains $\rho_{G,\mathrm{reg}}$ as soon as $\chi_{G \backslash M}<0$. 
\end{proof} 

\begin{remark} 
In the ``non-orbifold quotient'' setting of Proposition \ref{prop:homwidecritriemsurf}, by the Riemann-Hurwitz formula, $\chi_M<0$ if and only if $\chi_{G\backslash M}<0$. This is no longer true in the setting of Proposition \ref{prop:homwidecritriemorb}, when we only have  $\chi_{G \backslash M}<0 \Rightarrow  \chi_{M}<0$ but not the other way around. 
\end{remark}

The following is a detailed version of Corollary \ref{maincor}. 

\begin{corollary} \label{maincordetail} Let $M_1,M_2$ be two
  commensurable non-arithmetic closed Riemann surfaces. Then they
  admit a diagram \textup{(\ref{m0})} and, assuming the corresponding
  orbifold $M_0$ satisfies $\chi_{M_0} < 0$, isometry of $M_1$ and
  $M_2$ can be checked by computing the multiplicity of zero in
at most $$4 (\chi_{M_1}\chi_{M_2}/(\chi^{\mathrm{orb}}_{M_0})^2)!)^2 $$ twisted Laplace spectra, where $\chi^{\mathrm{orb}}_{M_0}$ is the orbifold Euler characteristic given by \begin{equation} \label{deforbeuler} \chi^{\mathrm{orb}}_{M_0}:=\chi_{M_0} - \sum (1-1/n_i), \end{equation}  with $n_i$ the order of the stabiliser group at the orbifold points.  
\end{corollary}

\begin{proof}
By Proposition \ref{2to1}, hyperbolic non-arithmetic commensurable closed Riemann surfaces automatically admit a diagram of the form (\ref{m0}), and by Proposition \ref{prop:homwidecritriemorb}, every group action is $\Q$\hyp{}homologically wide since we assume $\chi_{M_0}<0$. Therefore, Theorem \ref{main} applies.
To find a prime number $\ell$ coprime to $|G|$, we can always choose $\ell>|G|$, and by Bertrand's postulate, we can find such a prime $\ell \leq 2|G|$. Hence we can make the bound  in Theorem \ref{main} weaker by 
$ 2 \ell |H_2^{\mathrm{ab}}| \leq 4 |G|^2$. 
To express this entirely in terms of the original diagram, we set $d_i$ to be the degree of $M_i \rightarrow M_0$. Notice that the compositum $M_1 \bullet_{M_0} M_2$ is of degree at most $d_1 d_2$ over $M_0$, and the degree of the normal closure of the compositum is of degree at most $(d_1d_2)!$ over $M_0$ (see (\ref{factorial})). Hence $|G| \leq (d_1 d_2)!$. Now $d_i = \chi_{M_i}/\chi^{\mathrm{orb}}_{M_0}$ where $\chi^{\mathrm{orb}}_{M_0}$ is the orbifold Euler characteristic given by $\chi_{M_0} - \sum (1-1/n_i)$ for $n_i$ the order of the stabiliser group at the orbifold points \cite[5.1.3]{Choi}.  We find an upper bound of at most 
$$ 4 ((d_1 d_2)!)^2 \leq 4 (\chi_{M_1}\chi_{M_2}/(\chi^{\mathrm{orb}}_{M_0})^2)!)^2 $$
for the number of equalities of multiplicities that needs to be checked. 
\end{proof}

\begin{remark} 
If $M_1$ and $M_2$ as in Corollary \ref{maincordetail} are isospectral, by Weyl's law, they have the same volume. Since $d_i = \mathrm{vol}(M_i)/\mathrm{vol}(M_0)$, we can then assume that $d_1=d_2$ and $\chi_{M_1} = \chi_{M_2}$. 
\end{remark} 

In Section \ref{obstructions}, one finds some detailed examples of surfaces with less crude bounds on the required number of equalities. 

\subsection{Using the virtual Lefschetz character} 

The above arguments in dimension $2$ are based on very precise information given by fixed point formulae. These admit a generalisation  to a setup as in diagram (\ref{sunadasetup}) with $M$ of arbitrary dimension, where they can sometimes be used to deduce some information about condition $(\ast \ast)$ from Subsection \ref{astastsection}, more specifically,  whether $\Ind_H^G \one$ is a subrepresentation of the homology representation (over $\Q$). For this, we use that our manifolds are closed, and thus admit a regular triangulation, which allows us to apply the work of Curtis \cite{Curtis}. Consider the virtual Lefschetz character of $G$ given as 
$ \Lambda(g)= -h_{\Q}(g) + \sum\limits_{i \neq 1} (-1)^i \tr(g_* | \Ho_i(M,\Q)).$
Then by \cite[Prop.\ 1.6]{Curtis}, we have
\begin{equation} \label{curtis} \langle \Ind_H^G \one, \Lambda \rangle = \chi(M_1), \end{equation} 
the Euler characteristic of $M_1$. The formula provides no information for $3$-dimensional manifolds, since then $\chi(M_1)=0$ and $\Lambda=0$. 
The next proposition provides an example of a result that can be deduced from such methods. 
\begin{proposition}  if $M$ is of dimension $4$ and $\chi(M_1)\leq 0$, then $ \Ind_H^G \one$ and $h_{\Q}$ have at least one irreducible representation in common.
\end{proposition} 

\begin{proof} By Poincar\'e duality, $\Ho_3(M,\Q) = \Hom(\Ho_1(M,\Q),\Q) \cong \Ho_1(M,\Q)$, hence $$\Lambda = 2 \cdot \one + \Ho_2(M,\Q) - 2 h_{\Q},$$ and we conclude from (\ref{curtis}) that 
\[ 2 \langle \Ind_H^G \one, h_{\Q} \rangle = 2 + \langle \Ind_H^G \one, \Ho_2(M,\Q) \rangle - \chi(M_1) \geq 2  - \chi(M_1) > 0. \qedhere \] 
\end{proof} 

Such results do not suffice to completely verify whether condition $(\ast\ast)$ holds under general topological conditions in higher dimension (i.e., only referring to the vector space structure of $\Ho_1(M,\Q)$, and not to its $\Q[G]$-module structure), and indeed, in the next few subsections we will see examples showing that this is not possible. 

\subsection{Manifolds of dimension $\geq 3$} 
In case of $3$-manifolds, the picture can vary widely: it is possible to construct a class of closed $3$-manifolds with $\Q$\hyp{}homologically wide group actions, but also hyperbolic $3$-manifolds with large isometry group for which only the trivial group action is $\Q$\hyp{}homologically wide.  This upgrades to similar results in higher dimensions. 
We use glueing and surgery constructions that are based on ideas from work of Cooper and Long \cite{CooperLong} for topological manifolds. 

\begin{proposition} \label{hwdim>3} 
Suppose $N$ is a smooth compact connected $3$-manifold with a free smooth action by a finite group $G$, and let $\gamma$ denote a smooth simple closed curve in $N$ such that the orbit $G\cdot \gamma$ consists of $|G|$ disjoint smooth simple closed curves. Let $$X=N-\mathcal{N}(G \cdot \gamma)$$ denote the open manifold $N$ with an open regular neighbourhood of the $G$-orbit of $\gamma$ removed, and let $M$ denote the double of the manifold $X$. Define a Riemannian metric on $M$ as the pullback of any Riemannian metric on the quotient manifold $G \backslash M$; then if $M'$ is any $(n-3)$-dimensional closed smooth connected Riemannian manifold with trivial $G$-action, $M \times M'$ is an $n$-dimensional closed smooth connected Riemannian manifold with a free isometric $G$-action that is $\Q$\hyp{}homologically wide. 
\end{proposition} 

\begin{proof} Since \cite{CooperLong} concerns topological manifolds, we start by observing that the double of a \emph{smooth} manifold has a smooth structure compatible with the embedding of the original manifold, but composed with a diffeomorphism on one of the copies, see, e.g., \cite[VI.5]{Kosinski}. Hence $M$ is a smooth compact connected manifold on which the finite group $G$ acts smoothly (and properly) without fixed points. Therefore, the quotient $G \backslash M$ is a smooth compact connected manifold, too, and the quotient map $M \rightarrow G \backslash M$ is a smooth covering map. Choose a Riemannian structure on the quotient $G \backslash M$ such that it becomes a closed Riemannian manifold, and make $M$ into a closed Riemannian manifold by giving it the pullback Riemannian structure. Now $G$ acts on $M$ by fixed-point free Riemannian isometries. 

Let $\iota \colon \partial X \hookrightarrow X$ denote the embedding of the boundary. 
The Mayer-Vietoris sequence for rational homology contains a short exact sequence
\begin{equation} \label{ses} \Ho_1(\partial X, \Q) \rightarrow  2 \cdot \Ho_1(X, \Q) \rightarrow \Ho_1(M, \Q). \end{equation} 
In \cite[Lemma 2.1]{CooperLong}, duality is used to show that $$\mathrm{im}(\iota_* \colon \Ho_1(\partial X, \Q) \rightarrow \Ho_1(X, \Q))=\rho_{G,\mathrm{reg}}$$ as $\Q[G]$\hyp{}modules, so that we can rewrite (\ref{ses}) as a sequence of $\Q[G]$\hyp{}modules 
$$ 0 \rightarrow \rho_{G,\mathrm{reg}} \rightarrow 2 \cdot \Ho_1(X, \Q) \rightarrow \Ho_1(M, \Q).$$ 
Decomposing all $\Q[G]$\hyp{}modules into irreducible representations, we notice that every irreducible constituent of $ 2 \cdot \Ho_1(X, \Q) $ occurs with even multiplicity. Since it has $\rho_{G,\mathrm{reg}}$ as subrepresentation and the regular representation is the sum of all irreducible representations of $G$, we find that all irreducible representations of $G$ occur with even multiplicity in $2 \cdot \Ho_1(X, \Q)$, so that we can split off a $\Q[G]$\hyp{}module $V$ with  $$ 0 \rightarrow \rho_{G,\mathrm{reg}} \rightarrow 2 \rho_{G,\mathrm{reg}} \oplus V \rightarrow \Ho_1(M, \Q).$$ 
We conclude that $\rho_{G,\mathrm{reg}}$ is a subrepresentation of $\Ho_1(M, \Q)$. 

Now $G$ acts trivially on $M'$, so $G$ acts by isometries on the cartesian product $M \times M'$. By the K\"unneth formula, $M \times M'$ has first homology group $$\Ho_1(M \times M', \Q) = \Ho_1(M,\Q) \oplus \Ho_1(M',\Q),$$ so 
$\rho_{G,\mathrm{reg}}$ is also a subrepresentation of $\Ho_1(M \times M', \Q)$, and the action of $G$ on $M \times M'$ is $\Q$\hyp{}homologically wide. 
\end{proof} 

In the other direction, Cooper and Long have shown that $$\Ho_1(M, \Q) = \rho_{G,\mathrm{reg}} \oplus (\rho_{G,\mathrm{reg}}-1)$$ if $\iota_*$ is surjective, and that through Dehn surgeries, it is possible to ``remove'' the canonical $\Q[G]$\hyp{}modules $\rho_{G,\mathrm{reg}}$ and $\rho_{G,\mathrm{reg}}-1$ from the homology representation to arrive at a rational homology $3$-sphere with a $G$-action. Further surgery along an embedded hyperbolic knot allows one to construct such a \emph{hyperbolic} (i.e., constant $-1$ curvature) manifold \cite[Theorem 2.6]{CooperLong}. By making the same amendments as at the start of the proof of Proposition \ref{hwdim>3} to change topological manifolds into smooth and Riemannian ones, the result is the following. 

\begin{proposition}\label{clex}
For any finite non-trivial group $G$, there exists a hyperbolic rational homology $3$\hyp{}sphere $M$ with a free action of $G$ by isometries on $M$; in particular, the action of $G$ on $M$ is not $\Q$\hyp{}homologically wide. \qed 
\end{proposition} 

We conclude that in dimension $3$ homological wideness is unrelated to hyperbolicity (in marked contrast to the case of dimension $2$). 

\begin{corollary} \label{nhwdim>3}
For any finite non-trivial group $G$, and any dimension $n\geq 3$, there exists an $n$-dimensional closed connected Riemannian manifold $M'$ with a free action of $G$ by isometries on $M'$ for which the action of $G$ on $M'$ is not $\Q$\hyp{}homologically wide. 
\end{corollary} 

\begin{proof} 
If $M$ is as in Proposition \ref{clex}, and we let $G$ act trivially on the $(n-3)$-dimensional sphere $S^{n-3}$, set $M' = M \times S^{n-3}$. By the K\"unneth formula, $\Ho_1(M',\Q) = 0$ for $n \neq 4$ and $\Ho_1(M',\Q) = \Q$ for $n=4$, so it is impossible for non-trivial $G$ to act homologically wide on $M'$. 
\end{proof} 

\begin{remark}
Bartel and Page have shown that  there exists a closed hyperbolic $3$\hyp{}manifold $M$ with a free action of any given finite group $G$ by isometries on $M$ such that additionally, $\Ho_1(M,\Q)$ is any given $\Q[G]$\hyp{}module \cite{BP}. 
\end{remark}

\subsection{Locally symmetric spaces of rank $\geq 2$} \label{lssr}  Let $\mathbf{G}$ denote a connected semisimple Lie group with trivial center, $\mathbf{K}$ a maximal compact subgroup of $\mathbf{G}$, and $\Gamma$ a discrete subgroup of $\mathbf{G}$ such that $\Gamma \backslash \mathbf{G}$ is compact. Consider the locally symmetric Riemannian manifold  $M:= \Gamma \backslash \mathbf{G} / \mathbf{K}$. If all factors of $\mathbf{G}$ have real rank $\geq 2$, then $\Gamma$ has Kazhdan's property (T), and hence $\Ho_1(M,\Q)= \Gamma^{\mathrm{ab}} \otimes_{\Z} \Q = \{0\}$ (see, e.g., \cite[Cor.\ 1.3.6]{PropertyT}). This shows the following. 

\begin{proposition}  Only the trivial group can have a $\Q$\hyp{}homologically wide action on a locally symmetric space of rank $\geq 2$. \qed
\end{proposition} 

\subsection{Locally symmetric spaces of rank $1$} \label{hwls1} 
On the other hand (keeping the notations of Subsection \ref{lssr}), if $\mathbf{G}$ has rank $1$, the first Betti number of $M$ can be expressed in terms of representation theory via a formula of Matsushima's \cite{Matsushima}; more precisely, a sum of multiplicities of specific representation occurring in the representation $R_\Gamma$ of $\mathbf{G}$ by right multiplication on $L^2(\Gamma \backslash \mathbf{G})$. If $\mathbf{G}=\mathrm{SO}(n,1)$ for $n \geq 3$, there is a unique representation $J_1$ in that sum and $b_1(M)$ equals the multiplicity of the representation $J_1$ in $R_\Gamma$. Here, $J_1$ is the unique unitary irreducible representation with non-zero Lie algebra cohomology. Except for $n=3$, $J_1$ is not in the discrete or principal series (\cite[Thm.\ V.5; Rem.\ V.8; Prop.\ V.6]{Delorme}; \cite[Lemma 4.4]{HW} or \cite[VII.4.9]{BW}). For $n=3$, $J_1$ is the principal series representation of $\mathrm{PSL}(2,\C)$ on $L^2(\C)$ with Gelfand-Graev-Vilenkin parameters $(2,0)$, given explicitly as $$ J_1\left(\left(\begin{smallmatrix}a & b \\ c & d \end{smallmatrix} \right)\right) (f)(z):= (cz+d)^2 f\left(\frac{az+b}{cz+d}\right).$$ 

\begin{proposition} Let $M=\Gamma \backslash \mathbb{H}^n$ denote a closed hyperbolic $n$-manifold ($n \geq 3$), corresponding to a cocompact discrete subgroup $\Gamma$ in $\mathrm{SO}(n,1)$. If $G$ is a finite group acting $\Q$\hyp{}homologically widely on $M$, then 
$$ |G| \leq \langle J_1, R_\Gamma\rangle, $$
the multiplicity of the representation $J_1$ described above in the $\mathrm{SO}(n,1)$-representation $R_\Gamma$ given by right multiplication on $L^2(\Gamma \backslash \mathrm{SO}(n,1))$. \qed
\end{proposition}

\subsection{Hyperbolic manifolds; formulation in terms of uniform lattices} \label{sechyp}  If $M=\Gamma \backslash \mathbb{H}^n$ is a compact connected hyperbolic manifold of dimension $n \geq 3$ with finite full isometry group $\mathrm{Isom}(M)$, Mostow rigidity implies that $\mathrm{Isom}(M) \cong \mathrm{Out}(\Gamma)$, the outer automorphism group of $\Gamma$; indeed, $M$ is an Eilenberg-MacLane $K(\Gamma,1)$, and hence $\mathrm{Out}(\Gamma)$ is isomorphic to the group of homotopy self-equivalences up to free homotopy; but by Mostow rigidity, every homotopy equivalence is homotopic to an isometry \cite[Thm.\ 24.1']{Mostow}. Hence homological wideness of the action of a subgroup $G \hookrightarrow \mathrm{Isom}(M)$ on $M$ can be formulated in purely group theoretical terms.

\begin{proposition} The action of a finite group $G$ of isometries on a compact connected hyperbolic manifold $M=\Gamma \backslash \mathbb{H}^n$ of dimension $n \geq 3$ is $K$\hyp{}homologically wide if and only if  the representation 
$$ G \hookrightarrow \mathrm{Out}(\Gamma) \rightarrow \mathrm{Aut}(\Gamma^{\mathrm{ab}} \otimes_{\Z} K) \cong \mathrm{GL}(b_1(\Gamma),K) $$
contains the regular representation. \qed
\end{proposition}  
Recall again that Proposition \ref{clex} gives an example where this representation is trivial for $K=\Q$. Belolipetsky and Lubotzky \cite{BL} have shown that, given any finite group $G$, there exist infinitely many compact connected hyperbolic  manifolds with $G$ as isometry group. 

\begin{remark} Besson, Courtois and Gallot \cite[Th\'eor\`eme 9.1]{BCG} have proven that \emph{homeomorphic} oriented hyperbolic manifolds of the same dimension $n \geq 3$ with the same volume are isometric. In our situation, we start from a ``correspondence'' as in diagram (\ref{m0}) and a homeomorphism is not given. 
\end{remark}

\subsection{Using torsion homology for homological wideness} \label{torsionex} 
Instead of using $\Q$\hyp{}homological wideness, one may try to find suitable $\ell$ for which the action of $G$ on $\Ho_1(M)$ is $\F_\ell$\hyp{}homologically wide for specific $\ell$. For example, $\Ho_1(M)$ might be torsion, so that no non-trivial group acts $\Q$\hyp{}homologically wide, but nevertheless, $\Ho_1(M,\F_\ell)$ can contain $\F_\ell[G]$. 
We content ourselves with commenting on one example.  

\begin{example} \label{SWD} Let $M$ denote the Seifert--Weber dodecahedral space, a hyperbolic $3$-manifold with first Betti number zero; cf.\ \cite{SW}. By Mostow rigidity, $M$ is uniquely described by its fundamental group {\small $$ \langle a_1,\dots,a_6 | a_3^{-1} a_6 a_4^{-1} a_5 a_2, a_2^{-1} a_6 a_3^{-1} a_4 a_1, a_6 a_2^{-1} a_3 a_5 a_1^{-1}, a_2a_4 a_5^{-1}a_6 a_1^{-1}, a_3 a_4^{-1} a_6 a_5^{-1} a_1, a_4 a_2 a_5 a_3 a_1 \rangle.$$}  For the following facts, especially the computation of the homology representation, we refer to Mednykh \cite{Mednykh}: 
\begin{itemize} \item $\Ho_1(M,\Z) = \F_5^3$, 
admitting non-trivial maps to a cyclic group ${\Z}/\ell{\Z}$ for $\ell=5$.  
\item The full isometry group of $M$ is isomorphic to $S_5$. If we write generators as $r=(12)$ and $c=(12345)$, there is a faithful action on $\Ho_1(M,\F_\ell)$ through matrices in $\mathrm{GL}(3,5)$ given as 
$$ r = \left( \begin{smallmatrix} 4 & 2 & 4 \\ 0 & 0 & 2 \\ 0 & 3 & 0 
\end{smallmatrix}\right), \ c = \left( \begin{smallmatrix} 0 & 1 & 0 \\ 0 & 0 & 1 \\ 1 & 2 & 3 
\end{smallmatrix}\right).$$
\end{itemize} 
The isometry group $G=\langle r \rangle \cong {\Z}/{2{\Z}}$ of $M$ has a cyclic vector $(1,1,0)$ in $\Ho_1(M,\F_\ell)$, so the action of $G$ on $M$ is $\F_5$\hyp{}homologically wide (but not $\Q$\hyp{}homologically wide). Similarly, the isometry group $G=\langle crc^{-1}r \rangle \cong {\Z}/{3{\Z}}$ of $M$ has cyclic vector $(1,0,0)$. 
\end{example}

\begin{remark} 
The mod-$5$ homology representation $\rho \colon S_5 \rightarrow \mathrm{GL}(3,5)$ in Example \ref{SWD} is irreducible, and can be identified with 
\begin{equation} \label{isomrepSWD} \rho \cong   \mathrm{sgn} \otimes \psi, \end{equation} 
 where $\mathrm{sgn} \colon S_5 \rightarrow \F_5^*$ is the linear character given by the sign of a permutation (modulo $5$), and $\psi$ is the $3$-dimensional irreducible representation constructed as follows as composition factor of the standard permutation representation of $S_5$. The group $S_5$ acts on $V:=\F_5^5$ by permuting the standard basis vectors; consider the quotient  $W=V/L$  by the $S_5$-invariant line  $L:=\F_5 \cdot (1,1,1,1,1)$ spanned by the all-one vector, and consider the $S_5$-invariant hyperplane $U:=\{ [w=(w_1,\dots,w_5)] \in W \colon \sum w_i = 0 \}$ in $W$ (this makes sense since we work modulo $5$). Then the  natural induced action of $S_5$ on $U \cong \F_5^3$ is the $3$-dimensional faithful mod-$5$ representation $\psi$, that turns out to be irreducible. 
To see the isomorphism in (\ref{isomrepSWD}), we compute that the values of the Brauer character of $\rho$ on the conjugacy classes of elements of order coprime to $5$ (given by cycle type) are as in the table below. %\ref{brauer}. 
This equals the Brauer character of $\mathrm{sgn} \otimes \psi$, so $\rho$ has the same semi-simplification, but since $\mathrm{sgn} \otimes \psi$ is irreducible, it is actually isomorphic to $\rho$, that is then also automatically irreducible. 
\begin{table}[h]
\begin{center} 
\begin{tabular}{l||c|c|c|c|c|c} 
conjugacy class & $()$ & $(ab)$ & $(abc)$ & $(abcd)$ & $(ab) (cd)$ & $(ab) (cde)$ \\
%\hline 
matrix element & $1$ & $r$ & $rc^{-1}rc$ & $rc$ & $c^{-1}rc^2rc$ & $rc^{-1}rcrc^{-1}$ \\
%\hline
character value & $3$ & $-1$ & $0$ & $1$ & $-1$ & $2$ 
\end{tabular} 
\end{center}
%\caption{Brauer character of the mod-$5$ homology representation for the Seifert--Weber dodecahedral space} \label{brauer}
\end{table}    
\end{remark} 

\begin{remark}
If $M$ is a closed hyperbolic $3$\hyp{}manifold with large $\dim_{\F_\ell} \Ho_1(M,\F_\ell)$, there are sometimes explicit lower bounds on the volume of $M$, expounded in works of Culler and Shalen. For example, if for some prime $\ell$, $\dim_{\F_\ell} \Ho_1(M,\F_\ell) \geq 5$ then $\mathrm{vol}(M)>0.35$ \cite{CS}. 
\end{remark}

\begin{remark} The scope of the method of using $\ell$-torsion in $\Ho_1(M)$ in establishing homological wideness (or the weaker conditions we outlined) is unclear, although some heuristics can be set up by considering the size of the ($\ell$-)torsion subgroup. Bader, Gelander and Sauer \cite[Theorem 1.8]{BGS} have constructed, for any $\alpha \geq 0$, sequences $\{M_m\}$ of (non-arithmetic) closed hyperbolic rational homology $3$-spheres, converging in Benjamini--Schramm topology, for which \begin{equation} \label{growthtorsion} \# \Ho_1(M_m)_{\mathrm{tors}} \sim e^{\alpha\, \mathrm{vol}(M_m)}. \end{equation} (A conjecture of Bergeron and Venkatesh \cite[Conjecture 1.3 for $\mathrm{SO}(3,1)$, and bottom of page 122]{BV} states that if $M=\Gamma \backslash \mathbb{H}^3$ is closed hyperbolic \emph{arithmetic} manifold and $M_n:=\Gamma_m \backslash \mathbb{H}^3$ for a chain of congruence subgroups of $\Gamma_m$ of $\Gamma$ with trivial intersection, then the growth in (\ref{growthtorsion}) holds with $\alpha=1/(6 \pi).$)

In our setup, we would need to understand not just the size, but also more about the decomposition of $\Ho_1(M_m,\F_\ell)$ as a $\Z[\mathrm{Out}(\Gamma_m)]$-module. \end{remark} 

\section{Homological wideness, ``class field theory'' for covers, and a number theoretical analogue} \label{homwideNT} 

In this section, we relate homological wideness to the existence of geodesics with certain splitting behaviour. This allows us to study an analogue of homological wideness in the theory of extensions of number fields.  

\subsection{Abelian class field theory applied to the cover $M' \rightarrow M$} \label{CFT}  We briefly revert to the setup of 
Section \ref{geocon}, now under the stronger assumption that the action of $G$ on $M$ is $\F_\ell$\hyp{}homologically wide. In this situation, the intermediate covers $M_{g_i}''$ can be characterised in terms of the splitting behaviour of certain geodesics, as we now explain. 

First of all, fix $\omega \in \Ho_1(M,\F_\ell)$ to be a cyclic vector for the $G$-action, so that by assumption $$\Ho_1(M,\F_\ell) = \F_\ell[G] \omega \oplus U$$ for some complementary $\F_\ell[G]$-module $U$. Recall also that we have chosen a set of representatives $\{g_1,\dots,g_n\}$ for the left cosets $G/H_1$, so $\{g_1^{-1},\dots,g_n^{-1}\}$ is a set of representatives for the right cosets $H_1 \backslash G$ (by Hall's marriage theorem, there even exists a set $\{g_j\}$ that simultaneously represents the left and right cosets, so alternatively we could use the same set of representatives, if chosen suitably.)  This allows us to decompose 
$\Ho_1(M,\F_\ell) = \F_\ell[G] \omega \oplus U = U' \oplus U$  as direct sum of $\F_\ell[H_1]$-modules, where $U'=\bigoplus_{j=1}^n \F_\ell[H_1] g_j^{-1} \omega$ is, by the assumption of homological wideness, a free $\F_\ell[H_1]$-module with basis $\omega_j:=g_j^{-1} \omega$. Thus, the quotient map to the $H_1$-coinvariants on $U$ is as described in (\ref{transmapfree}), and since this can be identified with the map $q_{1*}$ by Lemma \ref{lem:coinvariants}, we find an isomorphism of $\F_\ell$-vector spaces $\Ho_1(M_1,\F_\ell) = \bigoplus \F_\ell \omega_j'  \oplus U_{H_1},$ where $\omega_j':=q_{1*}(\omega_j)$; in particular, the $\omega_j'$ are linearly independent. With these identifications, the map $q_{1*}$ is given as in the following diagram.  
$$ \xymatrix@C=0mm@R=4mm{**[r] \Ho_1(M,\F_\ell) \ar[r]^(.7){q_{1*}} \ar@<8mm>[d]^(0.35){\cong} &  **[r] \Ho_1(M_1,\F_\ell) \ar@<8mm>[d]^(0.35){\cong} \\
**[r] \bigoplus\limits_{j=1}^n \F_\ell[H_1] \omega_j \oplus U \ar[r]^(0.7){\trans} & **[r]\bigoplus\limits_{j=1}^n \F_\ell \omega'_j  \oplus U_{H_1} \\ 
**[r] {\sum\limits_{j=1}^n \sum\limits_{h \in H_1}  k_{j,h} h \omega_j + u}  \ar@{|->}[r] & **[r] {\sum\limits_{j=1}^n \left(\sum\limits_{h \in H_1}  k_{j,h}\right) \omega_j' +\trans(u) 
}
}$$ 
with $k_{i,h} \in \F_\ell, u \in U$. 
Since every Riemannian covering of $M$ is (isomorphic to) a quotient of $\tilde M$ by a subgroup of $\Gamma$, every  \emph{abelian} cover of $M$ is (isomorphic to) a quotient of $[\Gamma,\Gamma]\backslash \tilde M$, and hence Galois groups of abelian covers of $M$ correspond to quotient groups of $\Ho_1(M) \cong \Gamma^{\ab}$. 
The coverings $\varpi_i \colon M_{g_i}'' \rightarrow M$ are abelian, and they correspond, by construction, to the surjective maps 
\begin{align*} \varphi_i \colon \Ho_1(M)  \xrightarrow{\otimes \F_\ell} \Ho_1(M,\F_\ell) \xrightarrow{q_{1*}=\trans}   \Ho_1(M_1, \F_\ell)  = \bigoplus_{j=1}^n &\F_\ell \omega_j' \oplus U_{H_1} \rightarrow C\cong \F_\ell,  \\
&\sum_{j=1}^n k_j \omega_j' + u  \mapsto k_i \mbox{ mod } \ell 
\end{align*}  
with $k_{j} \in \F_\ell, u \in U.$
We let $\{\Omega_j\}$ denote a set of linearly independent elements of $\Ho_1(M)$ that map to $\{\omega_j\}$ in $\Ho_1(M, \F_\ell)$. 

The analogue of abelian class field theory for manifolds (described by Sunada in \cite[\S 5]{SunadaSurvey}, compare \cite[\S 4]{Sunada}) allows us to distinguish the different covers $\varpi_i$, as follows. 

We consider geodesics in $M$ (smooth closed curves in $M$ locally of minimal length) as oriented cycles, forgetting the parametrisation. A \emph{prime geodesic} is a geodesic that is not a multiple of another geodesic. Let $I_M$ denote the free abelian group generated by the prime geodesics of $M$, and let $I_M \rightarrow \Ho_1(M)$ denote the map that associates to a prime geodesic the homology class of the corresponding closed loop. We denote the kernel of this map by $I_M^0$, the subgroup of elements that are homologous to zero. The map is surjective: choose any lift of an element in $\Ho_1(M) \cong \Gamma^{\mathrm{ab}}$ to $\Gamma$, and consider the free homotopy class of free loops in $M$ corresponding to its conjugacy class; by shortening, that free homotopy class contains a closed geodesic (E.~Cartan's theorem, Note IV in his ``Le\c{c}ons sur la g\'eom\'etrie des espaces de Riemann''; see, e.g. \cite[Ch.\ 12, Thm.\ 2.2]{docarmo}; as is written in that reference, the result of Cartan does not require negative curvature), and that geodesic maps to the given element in  $\Ho_1(M)$. We conclude that there is an isomorphism $I_M/I_M^0 \cong \Ho_1(M)$. 

If $\fp$ is a prime geodesic on $M$, let $(\fp|\varpi_i) \in \F_\ell$ denote a generator for the stabiliser of (any) lift of $\fp$ to $M_{g_i}''$ (since the cover $\varpi_i$ is abelian, this does not depend on the chosen lift: in general,  the stabilisers of different lifts are conjugate). 
By the orbit-stabiliser theorem, the number of prime geodesics in $M_{g_i}''$ above $\fp$ is given by $|\F_\ell|/\langle (\fp|\varpi_i) \rangle$. This number is either $\ell$ (the prime geodesic ``splits'', and $(\fp|\varpi_i)=0$) or $1$ (the prime geodesic is ``inert'', and $(\fp|\varpi_i) \neq 0$). 

A main result in abelian class field theory for manifolds, \cite[Prop.\ 7]{SunadaSurvey}, says, given that the cover $\varpi_i$ is abelian, the group homomorphism 
$ I_M \rightarrow \F_\ell$ given by $\fp \mapsto (\fp|\varpi_i)$
is surjective with kernel $I_M^0 \cdot \varpi_i(I_{M_{g_i}''})$, and we have the following commutative diagram 
\begin{equation*}
  \xymatrix{0 \ar[r] & \ker(\varphi_i)\ar[d]^-{\cong}  \ar[r] & \Ho_1(M)\ar[d]^-{\cong} \ar[r]^-{\varphi_i} & \F_\ell \ar@{=}[d]\ar[r] & 0 \\ 
1 \ar[r] &   \varpi_i(I_{M_{g_i}''})\ar[r] & I_M/I_M^0 \ar[r]^-{(\cdot,\varpi_i)} & \F_\ell\ar[r] & 0}
\end{equation*} 
Since $\ker(\varphi_i)$ consists of the $H_1$-orbit of all homology classes spanned by, on the one hand, the classes $\Omega_j$ with $j \neq i$ and, on the other hand, the classes in the complement $U$, we deduce from this diagram the following result. 

\begin{proposition} The prime geodesics of $M$ that are inert in the cover $\varpi_i \colon M_{g_i}'' \rightarrow M$ are precisely the prime geodesics in the $H_1$-orbit of all geodesics whose homology class lies in the one-dimensional subspace $\langle \Omega_i \rangle$ of $\Ho_1(M)$.  \qed
\end{proposition}
Since $g_i$ represent different conjugacy classes of $H_1$ in $G$,  the subspaces $\langle \Omega_i \rangle$ of $\Ho_1(M)$ are distinct, and hence so are the covers $\varpi_i \colon M_{g_i}'' \rightarrow M$ in the fiber product \eqref{subcoverdiag}.

\subsection{Homological wideness and geodesics} \label{Cheb} The question whether the action of $G$ on $M$ is $\Q$-homo\-logically wide may be approached by splitting it into two separate questions: \begin{enumerate} \item[(a)] Does there exist of a prime closed geodesic on $M$ whose $G$-orbit consists of $|G|$ distinct geodesics? \item[(b)] Do the loops corresponding to these geodesics become homologous in $\Ho_1(M,\Q)$? 
\end{enumerate} 

We have no general framework to deal with the second question (notice the example of the loop separating the two tori in the connected sum $T^2 \# T^2$; it is non-trivial in the fundamental group, but becomes trivial in the first homology group, since it bounds one of the tori). However, concerning the first question, we can say the following. 

\begin{proposition} \label{anosov} 
Suppose that $M$ and $M_0$ are negatively curved closed Riemannian manifolds and $M \rightarrow M_0$ is a $G$-Galois Riemannian covering. Then there exists infinitely many closed prime geodesics in $M_0$ that lift to $|G|$ distinct closed prime geodesics in $M$, that hence form one $G$-orbit of such geodesics. 
\end{proposition} 
\begin{proof} 
We use the following analytical argument (similar to the analytic argument that there exist split primes in number fields). 
Let $Z_{M_0}(s):= \prod_{\fp} (1-e^{-s \ell(\fp)})^{-1}$ where $\fp$ runs over closed prime geodesics $\fp$ in $M_0$ of length $\ell(\fp)$ (i.e., the Ruelle zeta function for the geodesic flow on $M$). 

Let $h>0$ denote the volume entropy of the universal covering of $M$ (which is also the volume entropy of $M_0$). Since $M$ is negatively curved, the geodesic flow is weak mixing Anosov, so it follows that $Z_M(s)$ converges for $\mathrm{Re}(s)>h$ but has a pole at $s=h$ \cite[Prop.\ 9]{PPPrime}. 

Since the covering $M \twoheadrightarrow M_0$ is Galois, a prime geodesic $\fp$ splits into $r_{\fp}$ distinct prime geodesics in $M$, which are all of the same length, say, $f_{\fp}$ times the length of $\fp$. Then $r_{\fp} f_{\fp} = |G|$ (see \cite[\S 5]{SunadaSurvey} or \cite[\S 4]{Sunada}). 

Suppose, by contradiction, that the set $S$ of geodesics of $M_0$ that split completely into $|G|$ distinct prime geodesics in $M$ is finite. Then $f_{\fp} \geq 2$ for all $\fp \notin S$. We find, with $\mathfrak P$ running over the closed prime geodesics of $M$, and for real $s>h$,  
\begin{align} \label{ZZ} Z_M(s) & = \prod_{\mathfrak P} \left(1-e^{-s \ell(\mathfrak P)}\right)^{-1} \leq \prod_{\fp} \left(1-e^{-s f_{\fp} \ell(\fp)}\right)^{-r_{\fp}} \nonumber \\ & \leq 
\left(Z_{M_0}(2s)\right)^{|G|} \prod_{\fp \in S} \left(1+e^{-s\ell(\fp)}\right)^{|G|}. \end{align}  
This implies that $Z_{M_0}(2s)$ converges at $s=h$, and hence also the right hand side of the inequality (\ref{ZZ}) converges at $s=h$; this contradiction shows that $S$ is infinite. 
\end{proof} 

\begin{remark} The result also follows from the more general Riemannian covering version of the Chebotarev density theorem due to Parry and Pollicott \cite[Theorem 3]{PP} applied to the trivial conjugacy class in $G$. 
Kojima \cite[Prop.\ 2]{Kojima} gave a different proof for the case of closed orientable hyperbolic $3$-manifolds admitting a totally geodesic embedding of a Riemann surface of genus $\geq 3$, using projective laminations. 
\end{remark} 

\subsection{An analogue of homological wideness for number fields} 
Suppose that $G$ acts on $M$ with (orbifold) quotient $M_0$, and denote, as usual, the fundamental group of $M$ by $\Gamma$ and that of $M_0$ by $\Gamma_0$. Then, as in Remark \ref{outer}, $G$ acts by outer automorphisms on $\Gamma$, and hence it acts by automorphisms on the abelianisation $\Gamma^{\mathrm{ab}}$; in this interpretation, the homology representation is given by 
$ h \colon G \rightarrow \Aut(\Gamma^{\mathrm{ab}} \otimes_{\Z} \Q). $ 
This has the following analogue in number theory: if $K/\Q$ is a finite Galois extension with Galois group $G$, and $G_K:=\mathrm{Gal}(\bar \Q/K)$, $G_{\Q}:= \mathrm{Gal}(\bar \Q/\Q)$, we have a short exact sequence $1 \rightarrow G_K \rightarrow G_{\Q} \rightarrow G \rightarrow 1$, and hence an action of $G$ by outer automorphisms on $G_K$, given by conjugation by any lift of an element of $G$ to $G_{\Q}$. This also induces a group representation 
$$ \mathbf{h}  \colon G \rightarrow \Aut(G_K^{\mathrm{ab}} \otimes_{\Z} \Q). $$
The analogue of $\Q$-homological wideness in this context is the following. 

\begin{proposition}  The representation $\mathbf{h}$ contains the regular representation $\Q[G]$.
\end{proposition}  
\begin{proof} There exists a prime number $p$ that is totally split in $K/\Q$ (which follows from Chebotarev's theorem, or easier manipulation with zeta functions much as in the proof of Proposition \ref{anosov}). Let $\mathfrak p$ denote any prime ideal of $K$ above such $p$. Consider the reciprocity map from class field theory $$\vartheta \colon \mathbf{A}_{K,f}^* \rightarrow G_K^{\mathrm{ab}}$$ from the finite idele group $\mathbf{A}_{K,f}^*$ of $K$, let $\pi_{\mathfrak{p}}$ denote any uniformiser in the $\mathfrak p$-completion of $K$, and let $$F:=\vartheta((1,\dots,1,\pi_{\mathfrak{p}},1,\dots,1))$$ denote a ``Frobenius'' of $\mathfrak p$; then for any $g \in G$, $$g(F)=\vartheta((1,\dots,1,\pi_{g(\mathfrak{p})},1,\dots,1)).$$ By the assumption of total splitting, all $g(\mathfrak p)$ for $g$ running through $G$ and for $\mathfrak p$ a fixed prime above the given $p$ are distinct. Since the kernel of $\vartheta$ is the closure of the diagonally embedded $K^*$ in $\mathbf{A}_{K,f}^*$, we see that $\{ g(F) \colon g \in G\}$ are distinct commuting elements of infinite order in $G_K^{\mathrm{ab}}$, and hence $F$ is a cyclic vector for $G$. 
\end{proof} 

\begin{remark} 
Compared to the case of manifolds,  in the number theory case, the group $G_K^{\mathrm{ab}} \otimes_{\Z} \Q$ is of infinite rank (and captures all ramified abelian extensions), whereas $\Ho_1(M,\Q)$ is always of finite rank (and captures abelian topological covers). For number fields,  subproblem (a) as in Remark \ref{Cheb} is answered affirmatively by a similar splitting theorem as Proposition \ref{anosov} for manifolds; and problem (b)---linear combinations of geodesics becoming homologous---does not occur at all, due to the specific nature of the reciprocity map. 
\end{remark} 

\begin{remark} 
Continuing along the lines of the previous remark, in \cite{dynsys} it was shown that isomorphism of number fields is equivalent to topological conjugacy of associated dynamical systems built from the reciprocity map. The analogous question for manifolds becomes: associate to a manifold $M$ the dynamical system given by the monoid generated by prime geodesics acting on $\Ho_1(M,\Z)$, where a prime geodesic acts by adding the homology class of the corresponding closed loop; then under what conditions is isometry of two manifolds $M_1$ and $M_2$ equivalent to the topological conjugacy of the corresponding dynamical systems, where the identification of the prime geodesics is length-preserving? 
\end{remark}

\section{Examples concerning the main result}  \label{obstructions} 

\subsection{Examples where Theorem \ref{main} does not apply} We first give some examples where the conditions of Theorem \ref{main} are not met. 

\begin{example}
Sch\"uth constructed isospectral, non-isometric simply connected manifolds $M_1$ and $M_2$ \cite{SchuethAnn}: these are non-isometric manifolds that are indistinguisheable by any twisted spectrum on functions, simply because there is nothing to twist by. 

From the perspective of our two conditions, this example already violates the first: if a diagram like (\ref{m0}) would exist, then by simple connectedness, $M_1$ and $M_2$ would both be isometric to $\tilde M_0$, and hence isometric, which they are not. 
\end{example}

\begin{example} 
The isospectral compact Riemann surfaces constructed by Vign\'eras \cite{Vigneras} are described by commensurable arithmetic lattices and by \cite[Proposition 3]{Chen} there does not exist a diagram (\ref{m0}), so our first condition is violated.  The original examples have Euler characteristic $-201600$ (corrected value from \cite{LV}) but Linowitz and Voight have constructed similar examples of Euler characteristic $-10$ and that this is the maximal Euler characteristic that can occur for the class of ``unitive'' torsion free Fuchsian groups; cf. \cite{LV}. 
\end{example}

\begin{example}
Ikeda's isospectral non-isometric lens spaces from Example \ref{lens} have $M=S^5$, and $M_1=L(11;1,2,3)$ and $M_2=L(11,1,2,4)$. In this example, $H_i \cong {\Z}/{11}{\Z}$. The action is not $\Q$\hyp{}homologically wide; actually, the only cyclic cover of $M_1$ is $M$ itself.  
\end{example}

\begin{example} 
We consider Milnor's example \cite{Milnor} where $M_i = \Gamma_i \backslash \R^{16}$ with $\Gamma_1 = E_8 \oplus E_8$ and $\Gamma_2 = E_{16}$. The group $\langle \Gamma_1, \Gamma_2 \rangle = \Gamma_1 + \Gamma_2$ is a lattice, and, as Chen \cite[\S 3]{Chen} observed, there is a diagram of the form (\ref{sunadasetup}) with $M:=(\Gamma_1 \cap \Gamma_2) \backslash \R^{16}$, $G= (\Gamma_1 + \Gamma_2) / (\Gamma_1 \cap \Gamma_2)$ and $H_i = \Gamma_i / (\Gamma_1 \cap \Gamma_2)$. One easily checks, by writing down explicit lattice bases (e.g., using Magma), 
% I used the following code in the MAGMA online calculator:
%L0:=LatticeWithBasis(8,[2,0,0,0,0,0,0,0, -1,1,0,0,0,0,0,0, 0,-1,1,0,0,0,0,0, 0,0,-1,1,0,0,0,0, 0,0,0,-1,1,0,0,0, 0,0,0,0,-1,1,0,0, 0,0,0,0,0,01,1,0, 1/2,1/2,1/2,1/2,1/2,1/2,1/2,1/2]);
%L1:=DirectSum(L0,L0);
%L2:=LatticeWithBasis(16,[
%2,0,0,0,0,0,0,0,0,0,0,0,0,0,0,0,  -1,1,0,0,0,0,0,0,0,0,0,0,0,0,0,0,  0,-1,1,0,0,0,0,0,0,0,0,0,0,0,0,0, 0,0,-1,1,0,0,0,0,0,0,0,0,0,0,0,0,  0,0,0,-1,1,0,0,0,0,0,0,0,0,0,0,0,  0,0,0,0,-1,1,0,0,0,0,0,0,0,0,0,0,  0,0,0,0,0,-1,1,0,0,0,0,0,0,0,0,0,  0,0,0,0,0,0,-1,1,0,0,0,0,0,0,0,0,  0,0,0,0,0,0,0,-1,1,0,0,0,0,0,0,0,  0,0,0,0,0,0,0,0,-1,1,0,0,0,0,0,0,  0,0,0,0,0,0,0,0,0,-1,1,0,0,0,0,0,  0,0,0,0,0,0,0,0,0,0,-1,1,0,0,0,0,  0,0,0,0,0,0,0,0,0,0,0,-1,1,0,0,0,  0,0,0,0,0,0,0,0,0,0,0,0,-1,1,0,0, 0,0,0,0,0,0,0,0,0,0,0,0,0,-1,1,0,   1/2,1/2,1/2,1/2,1/2,1/2,1/2,1/2,1/2,1/2,1/2,1/2,1/2,1/2,1/2,1/2]);
%L3:=L1 meet L2;
%Index(L2,L3);
%Index(L1,L3);
%G:=(L1+L2)/L3;
%G;
 that $G=H_1 \times H_2$ is the Klein four-group with $H_i \cong {\Z}/2{\Z}$. Since $G$ is abelian, the subgroups $H_1$ and $H_2$ cannot be weakly conjugate (since they are not equal as subgroups of $G$).
 A cover realising the wreath product cannot exist: the group $G$ acts non-trivially on the coset space $G/H_1=H_2$, so the wreath product $C^2 \rtimes G$ is non-commutative, and hence cannot occur as subgroup of the (abelian) fundamental group of $M$.   
 
Chen also observed that for these $M_1,M_2$, there exists another diagram of the form (\ref{sunadasetup}) in which the corresponding groups \emph{are} weakly conjugate, see \cite[Prop.\ 1]{Chen}. In that example, the group $G$ contains an element involving a non-trivial translation. One easily computes the corresponding lattices bases (e.g., again using Magma) to see that in that case, $H_i \cong ({\Z}/2{\Z})^{12}$, 
%Ls:=LatticeWithBasis(4,[2,0,0,0, 0,2,0,0, 0,0,0,2, 1,1,1,1]);
%Lt:=DirectSum(Ls,Ls);
%L:=DirectSum(Lt,Lt);
%H1:=L1/L;
%H2:=L2/L;
%H1;
%H2;
but again, the wreath product cannot be realised for the same reason as above (the corresponding wreath product is non-commutative but the fundamental group of $M$ is commutative). 
\end{example} 

\begin{example} Doyle and Rossetti \cite{DR} constructed two isospectral, non-isometric closed flat $3$\hyp{}manifolds, called ``Tetra'' and ``Didi''. These are commensurable, given as quotients of $\R^3/({\Z}^2 \times 2{\Z})$ by the action of the two non-isomorphic groups of order $4$, but there is no diagram (\ref{m0}) (if so, they would be strongly isospectral, but they cannot be isospectral on $1$-forms, since they have different first Betti numbers). 
\end{example}

\subsection{Examples from Sunada's construction} 

If $M_1$ and $M_2$ are isospectral via the Sunada construction, then a diagram of the form (\ref{sunadasetup}) exists by default (possibly with orbifold $M_0$). We discuss some ``small'' examples of closed surfaces (where a group $G$ is realised by choosing a compact hyperbolic Riemann surface $M_0$ whose genus is larger than or equal to the number of generators of $G$, so that there is a surjection $\pi_1(M_0) \twoheadrightarrow G$). These examples satisfy the requirements of Theorem A and illustrate numerically that, whereas our auxiliary construction involves manifolds and groups of relatively large order / negative Euler characteristic, the dimension of the representations and the number of required twists can be rather small (dictated by the degrees of the coverings $M \rightarrow M_i$). 

\begin{example} Sunada lists three examples in \cite[\S 1, Example 1--3]{Sunada}, for which we indicate in the first three rows of Table \ref{tablesunadaex} the dimensions of the representations by which one needs to twist, as well as how many equalities of multiplicities of zero suffice in Theorem \ref{main}. This is straightforward, except for the last case, named ``Komatsu''; here, we need to find a prime $\ell$ coprime to $|G|=(p^3)!$, and if we observe that all prime divisors of $|G|$ are $<p^3$ and use Bertrands postulate/Chebyshev's theorem that there is a prime between $p^3$ and $2p^3-2$, we can certainly find $\ell \leq 2p^3-3$. We also
use that $|H_i^{\ab}| \leq p^2$; for that, note the commutator identities
\[ 
\left[\left( \begin{smallmatrix} 1 & 1 &  \\  & 1 &  \\  &  & 1 
\end{smallmatrix}\right), \left( \begin{smallmatrix} 1 &  &  \\  & 1 & 1 \\  &  & 1 
\end{smallmatrix}\right)\right]=\left( \begin{smallmatrix} 1 &  & 1  \\  & 1 &  \\  &  & 1 
\end{smallmatrix}\right), \ \ 
 \left[\left( \begin{smallmatrix} 1 & 1 &  \\  & 1 &  \\  &  & 1 
\end{smallmatrix}\right), \left( \begin{smallmatrix} 1 &  & 1  \\  & 1 &  \\  &  & 1 
\end{smallmatrix}\right)\right]=
 \left[\left( \begin{smallmatrix} 1 &  & 1 \\  & 1 &  \\  &  & 1 
\end{smallmatrix}\right), \left( \begin{smallmatrix} 1 &  &  \\  & 1 & 1 \\  &  & 1 
\end{smallmatrix}\right)\right]=
\left( \begin{smallmatrix} 1 &  &   \\  & 1 &  \\  &  & 1 
\end{smallmatrix}\right), 
 \]
which imply that the abelianisation of the group of upper triangular matrices in $\mathrm{GL}(3,\F_p)$ with all diagonal entries equal to $1$ is isomorphic to 
$({\Z}/{p\Z})^2$. 
\end{example}

We now discuss in some more detail follow-up examples of Brooks--Tse and Barden--Kang that have the smallest known genus (in variable or constant curvature). The results are summarised in the fourth and fifth row of Table \ref{tablesunadaex}. 

\begin{example} 
Brooks and Tse (see \cite{BT} and \cite{BAMM}) constructed closed surfaces of Euler characteristic $-4$ (genus $3$) that are isospectral but not isometric for well-chosen metrics whose curvature is not constant. Here, $G=\mathrm{PSL}(2,7)=\mathrm{(P)SL}(3,2)$ is the unique simple group of order $168=2^3 \cdot 3 \cdot 7$ \cite[6.14(4) \& 6.15]{Huppert}, the automorphism group of the Fano plane $\mathbf{P}^2(\mathbf{F}_2)$, and $H_1,H_2$ are index 7 subgroups given as $3 \times 3$ matrices in $\mathrm{SL}(3,2)$ in which the first column, respectively the first row, is $(1,0,0)$ (stabilisers of a point and a dual hyperplane in $\mathbf{P}^2(\mathbf{F}_2)$). 

In this case, $M_0$ is an orbifold sphere with $3$ singular points of order $7$, and $\chi_{M} = - 2^5 \cdot 3$. The smallest possible $\ell$ that can be chosen is $\ell=5$, and then $\tilde G$ is of order $5^7 \cdot 168 \approx 13 \cdot 10^6$ and $\chi_{M'} = -2^5 \cdot 3 \cdot 5^7$. The dimension of the required representations in Proposition \ref{art}, on the other hand, is relatively low: $7$. Recall that $S_4$ has a unique normal subgroup isomorphic to the Klein four-group (generated by products of two two-cycles) and quotient isomorphic to $S_3 \cong \mathrm{GL}(2,2)$ \cite[5.2]{Huppert}. Since $H_i \cong ({\Z}/2{\Z} \times {\Z}/2{\Z}) \rtimes \mathrm{GL}(2,2) \cong S_4$ has commutator subgroup $A_4$, the number of spectral equalities that needs to be checked is only $2\cdot5 \cdot |{H^{\mathrm{ab}}_2}| =20$ (cf.\ Proposition \ref{numberofchecks}). 
\end{example} 

\begin{example}
In the same reference, Brooks and Tse (see \cite{BT} and \cite{BAMM}) used a different representation of the same group $G$ and the same subgroups $H_i$ to construct  isospectral non-isometric compact Riemann surfaces (surfaces of \emph{constant negative curvature $-1$}) of Euler characteristic $-6$ (genus $4$). Here, $M_0$ is an orbifold torus with a single singular point of order $7$. Then $\chi_{M'} = - 2^4 \cdot 3^2 \cdot 5^7$, but one needs to check only $20$ equalities of spectral multiplicities of $7$-dimensional representations. 
\end{example} 

\begin{example} 
A similar construction of such surfaces (not of constant curvature) of Euler characteristic $-2$ (genus $2$) by Barden and Kang \cite{BK} has the largest currently known Euler characteristic. In their case, $G$ has order $96$, and $H_1$ and $H_2$ are of index $12$ in $G$, both isomorphic to ${\Z}/{2}{\Z} \times {\Z}/{4}{\Z}$. We can choose $\ell=5$ and $\tilde G$ of order $5^{12} \cdot 96 \approx 2 \cdot 10^{10}$. With $\chi_{M_i}=-2$ for $i=1,2$, we have $\chi_{M} = -16$ and the dimension of the required representations is $12$. In this case, one needs to check $80$ equalities of multiplicities of zero in various spectra. 
\end{example}

\begin{example} 
We consider an example where the order of $|G|$ is odd (the situation can be realised using Riemann surfaces of sufficiently high genus, as above). Note that $G$ is forcedly solvable, by the Feit--Thompson theorem. Let $p$ denote an odd prime number. As in Guralnick \cite[Example 4.1]{Guralnick}, consider the group $G$ of order $p^5$ given as the semidirect product $$G = A \rtimes H, \mbox{ where } A=({\Z}/{p^2{\Z}} \times {\Z}/{p{\Z}}) \mbox{ and } H=({\Z}/{p{\Z}} \times {\Z}/{p{\Z}})$$ 
with the action $a \mapsto a^h$ of $h\in H$ on $a\in A$ determined by 
$$ (1,0)^{(1,0)} = (1,1), (0,1)^{(1,0)} = (p,1), (1,0)^{(0,1)} = (p+1,0), (0,1)^{(0,1)} = (0,1). $$  
Now $G$ has two subgroups $H_i \cong {\Z}/{p{\Z}} \times {\Z}/{p{\Z}}$ that are not weakly conjugate; more precisely, in the above description, $$H_1= H = \langle ((0,0),(1,0)), ((0,0),(1,0))\rangle \mbox{\, and } H_2=\langle ((0,0),(1,0)), ((p,0),(0,1))\rangle.$$ Using the smallest prime $\ell \geq 3$ coprime to $|G|$ in the main theorem, we need to check the following number of equalities: (a) if $p=3$, $90$ equalities, using $\ell=5$;  (b) if $p \neq 3$, $6p^2$ equalities, using $\ell=3$. In this case, a better result is possible using Pintonello's method from Remark \ref{pint}, where one can choose $\ell=2$ at the cost of adding extra identities, leading to $4p^2+2$ equalities to be checked (this number equals $38<90$ for $p=3$ and is always smaller than $6p^2$). Note that an earlier similar example with $|G|=p^6, |H_i|=p^2$ can be found in \cite[Ch.\ IV, Ex.\ 2, p.\ 63]{aldukair}. 
\end{example} 

\begin{center}
\begin{table}[t]
\begin{tabular}{llllll}
\hline
Name & $|G|$ & $|H_i|$ & smallest $\ell$ & dim.\ of reps.\ & number of equalities $\leq$ \\ 
\hline 
``Gerst'' & $32$ & $4$ & $3$ & $8$ & $24$ \\
``Ga{\ss}mann'' & $720$ & $4$ & $7$ & $180$ & $56$ \\
``Komatsu'' & $(p^3)!$ & $p^3$ & $ \leq 2p^3-3$ & $(p^3-1)!$ & $2 p^2(2p^3-3)$ \\
\hspace*{1em} e.g., $p=3$ & $\approx 10^{28}$ & $27$ &29& $\approx 4 \cdot 10^{26}$ & $522$ \\
Brooks-Tse & $168$ & $24$ & $5$ & $7$  & $20$ \\
Barden-Kang & $96$ & $8$ & $5$ & $12$ & $80$ \\
Guralnick & $p^5$ & $p^2$ & $2$  & $p^3$ & $4p^2+2$ \\
\hline\\
\end{tabular} 
\caption{Examples of Sunada triples (with $p$ an odd prime number) from  \cite[Example 1]{Sunada}, \cite{BT}, \cite{BK} and \cite[Example 4.1]{Guralnick}. We indicate the dimension of the representations involved in Theorem \ref{main}, as well as an upper bound on the number of equalities of spectra that need to be checked. In the last line, we choose $\ell=2$ and use the bound in Remark \ref{pint}.} \label{tablesunadaex} 
\end{table}
\end{center}

\subsection{Flat manifolds isospectral for all twists by linear characters} 

\begin{example} Miatello and Rossetti constructed non-isometric closed flat manifolds $M_1$ and $M_2$ that admit non-trivial twists but are twisted isospectral on functions and forms for all twists by linear characters (and hence also any representation that decomposes as a direct sum of such)
\cite[4.5]{Miatello}. Twisted Laplacians of such linear characters act on sections of flat line bundles. 

The manifolds are constructed as follows (our notations differ from \cite{Miatello}). Consider the group  of affine transformations $\R^4 \rtimes \mathrm{O}(4)$ of $\R^4$. Let $\tau_v$ denote the translation by $v \in \R^4$, let $$\Lambda:=\{\tau_v \colon v \in \Z^4\} \cong \Z^4,$$ and denote the standard basis vectors as $e_1,\dots,e_4$. Consider the two orthogonal matrices $$A:=\mathrm{diag}(1,1,-1,-1) \mbox{ and } A':=\mathrm{diag}(1,-1,1,-1)$$ and the vectors $$b_1=(e_2+e_4)/2, b_1'=e_3/2, b_2=e_2/2, b_2'=e_1/2.$$  The manifolds are $M_i = \Gamma_i \backslash \R^4$ with $\Gamma_i=\langle A \tau_{b_i}, A' \tau_{b'_i}, \Lambda \rangle$. The $\Gamma_i$ are Bieberbach groups fitting into an exact sequence
$$ 1 \rightarrow \Lambda \cong  \Z^4 \rightarrow \Gamma_i \rightarrow \langle A, A' \rangle \cong ({\Z}/{{2}{\Z}})^2  \rightarrow 1. $$

Concerning our two conditions, the situation in this example is as follows. If $A \in \mathrm{O}(4)$ is of order two and $b \in 1/2 \Z^4$, then $(A\tau_b)^2 = \tau_{b+Ab} \in \Lambda$, since $b+Ab \in \Z^4$; hence for any $\gamma \in \Gamma_i$, we have $\gamma^2 \in \Lambda \leq \Gamma_1 \cap \Gamma_2$. Therefore, with $\Gamma_0 := \langle \Gamma_1, \Gamma_2 \rangle$ we have $\Gamma_0 / \Gamma_i \cong ({\Z}/{{2}{\Z}})^2$, and a diagram such as (\ref{m0}) exists with $M_0:=\Gamma_0 \backslash \R^4$. Then we can set $M:=\Gamma \backslash \R^4$ with $\Gamma:=\Gamma_1 \cap \Gamma_2$ to get a diagram (\ref{sunadasetup}). From the presentation of $\Gamma_i$, one may compute the intersection to be  $\Gamma = \Lambda \cong \Z^4$, so $M$ is a $4$-torus. Since $G = \Gamma_0 / \Gamma$ is a group of exponent $2$ and order $16$, we find 
$G = H_1 \times H_2 \cong ({\Z}/{{2}{\Z}})^4$ with $H_i = ({\Z}/{{2}{\Z}})^2$. 

We now look at the second condition. Since $G/H_1 \cong H_2$ with left $G$-action induced by the multiplication in $H$, it follows that $\Ind_{H_1}^G \one \cong \Z[H_2]$ is the ($4$-dimensional) regular representation of $H_2$, i.e., the direct sum of the four linear characters of the Klein four-group. Note that, similarly, $\Ind_{H_2}^G \one \cong \Z[H_1]$, and hence $\Ind_{H_1}^G \one \cong \Ind_{H_2}^G \one,$ meaning that $H_1$ and $H_2$ are weakly conjugate subgroups of $G$.  

To analyse the homology representation, we note that 
$ \Ho_1(M) \cong \Gamma^{\mathrm{ab}} \cong \Z^4$. The group $G$ acts on this through conjugation by (outer) automorphisms, i.e., if $A\tau_b$ represents an element of $G$ (with $A \in \mathrm{O}(4)$ of order two, $b \in 1/2 \Z^2$), then it acts by 
$ \tau_{-b} A \tau_v A \tau_b = \tau_{Av}.$ We conclude that the representation of $G$ on $ \Ho_1(M) $ factors through the representation of $\langle A, A' \rangle \cong ({\Z}/{{2}{\Z}})^2$ in $\mathrm{GL}(4,\Z)$. Looking at characters, this is the regular representation of the Klein four-group. We conclude that 
$$ \Ho_1(M) \cong \Ind_{H_1}^G \one $$
as $G$-modules. 

We conclude that our two conditions are satisfied, and, using $\ell=2$ as in Remark \ref{pint}, we find that the manifolds $M_1$ and $M_2$ can be distinguished by $18$ equalities of twisted spectra for $4$-dimensional representations. These representation are then forcedly not all direct sums of linear characters.  

In a different direction, Gordon, Ouyang and Sch\"uth \cite{GOS} \cite{SchuethTAMS} have shown how to distinguish the manifolds in this example using a  \emph{non-flat} Hermitian \emph{line} bundle. 
\end{example}

\subsection{An example that does not arise from Sunada's construction} 

\begin{example}
We consider a finite group $G$ with two non-weakly conjugate subgroups $H_1$ and $H_2$. Let $M_0$ denote a compact  hyperbolic Riemann surface of genus larger or equal to the number of generators of $G$, and fix a surjective morphism $\pi_1(M_0) \twoheadrightarrow G$. This leads to a diagram of the form (\ref{sunadasetup}), and homological wideness is satisfied by Proposition \ref{prop:homwidecritriemsurf}. For example, set $G=S_4$ with $H_1 = \langle (1234) \rangle \cong {\Z}/4{\Z}$ and $H_2 = \langle (12)(34), (13)(24) \rangle \cong {\Z}/2{\Z} \times {\Z}/2{\Z}$. Since the cycle types in $H_1$ and $H_2$ are different, they are not weakly conjugate. We choose $M_0$ to be a genus two compact Riemann surface and let $\ell=5$; now we can distinguish the genus $7$ Riemann surfaces $M_1$ and $M_2$ using $40$ equalities of multiplicities of zero in the spectra of Laplacians twisted by $6$-dimensional representations. 
\end{example}

\section{Length spectrum} \label{length} 

\subsection{$L$-series} We assume that $M$ is a connected negatively curved oriented closed Riemannian manifold. Then the following properties hold: 
\begin{enumerate} 
\item each free homotopy class $[\gamma]$ contains a unique closed geodesic \cite[Thm.\ 3.8.14]{Klingenberg}. We write $\ell(\gamma)$ for the length of that geodesic. 
\item the geodesic flow is of Anosov type, and the topological entropy equals the \emph{volume entropy}, defined as $h_M:=\lim\limits_{R \rightarrow + \infty} (\log \mathrm{vol}(B(x,R)))/R, $ where $B(x,R)$ is a geodesic ball of radius $R$ in the universal cover $\tilde M$ of $M$ centered at some point $x \in \tilde M$ \cite{Manning}. \end{enumerate}
Recall some elements of the theory of prime geodesics (for general $M$ but an abelian cover, we treated this in Subsection \ref{CFT}). We call $[\gamma]$ \emph{prime} if the associated geodesic is not a multiple of another geodesic (in the sense of oriented cycles). Let $\varpi \colon M' \rightarrow M$ denote a $G$-cover for a finite group $G$. Notice that $h_{M'}=h_M$ by definition. Above each fixed prime geodesic of $M$ lie finitely many prime geodesics of $M'$, the set of which carries a transitive $G$-action. For a fixed $\gamma'$ mapping to $\gamma$, we let $(\varpi | \gamma)$ denote any element of $G$ that generates the stabiliser of $\gamma'$.  All such elements are conjugate in $G$ (in the abelian case, the element does not depend on the choice of $\gamma'$, cf.\ Subsection \ref{CFT}). 

Let $\rho \colon G \rightarrow \Un_N(\C)$ denote a representation. Since the determinant takes the same value on conjugate matrices, we have a well-defined  associated \emph{$L$-series} 
$$ L_M(\rho,s):= \prod_{[\gamma]} \det(1_N - \rho((\varpi|{\gamma})) e^{-s \ell(\gamma)} )^{-1},$$ where $1_N$ is the identity matrix of size $N \times N$. If $\rho_1$ and $\rho_2$ are two representations of $G$, then \cite[p,\ 135]{PP} \begin{equation} \label{artin} L_M(\rho_1 \oplus \rho_2) = L_M(\rho_1)L_M(\rho_2).\end{equation} 
Choosing $\rho=\one$, $L_M(\one,s)$ is related to analogues of the Selberg zeta function. Standard theory of Dirichlet series implies that knowledge of $L_M(\one,s)$ is equivalent to knowledge of the multiset of lengths $\{ \ell(\gamma) \}$. 
Parry and Pollicott and Adachi and Sunada have proven: 

\begin{lemma}[{\cite[Thm.\ 1 \& 2]{PP} and \cite[Thm.\ A]{ASJFA}}] \label{ext} The function $L_M(\rho,s)$ converges absolutely for $\mathrm{Re}(s)>h_M$ and can be analytically continued to an open set $D$ that contains the half-plane $\mathrm{Re(s)} \geq h_M$. For an irreducible representation $\rho$, $L(\rho,s)$ is holomorphic in $D$ unless $\rho=\one$. Furthermore, $L_M(\one,s)$ has a simple pole at $s=h_M$.\qed
\end{lemma} 

\subsection{Main result for the length spectrum} In this setup, the analogue of Lemma \ref{mult} is the following: 
\begin{lemma} \label{multL} Let $G$ be a finite group acting by fixed-point free isometries on a negatively curved Riemannian manifold $M'$ with quotient $M=G \backslash M'$. Set $h:=h_M=h_{M'}$. 
If $\rho$ is any unitary representation of $G$, then the multiplicity $\langle \rho, \one \rangle$ of the trivial representation in the decomposition of $\rho$ into irreducibles equals  $-\mathrm{ord}_{s=h} L_M(\rho,s)$, the order of the pole of $L_M(\rho,s)$ at $s=h$. 
\end{lemma} 

\begin{proof} 
Let $D$ denote the extended region of convergence as in Lemma \ref{ext}. Decompose $\rho = \bigoplus\limits_{i=1}^N \langle \rho_i,\rho\rangle \rho_i$ as a sum over irreducible representations $\rho_i$. Then by formula (\ref{artin}), we have $$L_M(\rho,s) = \prod_{i=1}^N L(\rho_i,s)^{\langle \rho_i,\rho\rangle},$$ and set $\rho_1 = \one$ for convenience.
Applying Lemma \ref{ext}, we find from this product decomposition that $\mathrm{ord}_{s=h} L_M(\rho,s) = \langle \one,\rho\rangle$.
\end{proof} 

We have the following further two analogues of results previously shown for the Laplace spectra: 
\begin{enumerate}
\item The analogue of Lemma \ref{SNT} for $L$-series states that \emph{if $M \rightarrow M_1 \rightarrow M_0$ is a tower of finite Riemannian coverings and $M \rightarrow M_0$ is Galois with group $G$, $M \rightarrow M_1$ with group $H$, and $\rho \colon H \rightarrow \Un_N(\C)$ a representation, then $$ L_{M_0}(\Ind_{H}^G \rho,s) = L_{M_1}(\rho,s), $$}see \cite[Remark 2 after Lemma 1]{SunadaNT}. 
\item The analogue of Proposition \ref{solo} for $L$-series states: suppose that we have diagram \eqref{sunadasetup} and set $h:=h_{M_1}=h_{M_2}=h_M$; then,  
\emph{for two linear characters $\chi_1 \in \check{H}_1$ and $\chi_2 \in \check{H}_2$, a representation isomorphism $\Ind_{H_1}^G \chi_1 \cong \Ind_{H_2}^G \chi_2$ is equivalent to 
\begin{equation} \label{Leq} \mathrm{ord}_{s=h} L_{M_i}(\bar \chi_i \otimes \Res_{H_i}^G \Ind_{H_j}^G \chi_j,s) \end{equation} being the same for the pairs $(i,j)$ given by $(1,1),(2,1)$ and $(1,2),(2,2)$}. The proof is essentially the same as that of Proposition \ref{solo}, but now using Lemma \ref{multL} instead of Lemma \ref{mult}. 
\end{enumerate}

We can adopt the reasoning in the proof of the main theorem to find: 

\begin{theorem} \label{mainL} Suppose we have a diagram \textup{(\ref{m0})} of negatively curved Riemannian manifolds with (common) volume entropy $h$, and suppose that the action of $G$ on $M$ in the extended diagram \textup{(\ref{sunadasetup})} is homologically wide. 
Then $M_1$ and $M_2$ are equivalent Riemannian covers of $M_0$ if and only if the pole orders at $s=h$ of a finite number of specific $L$-series of representations on $M_1$ and $M_2$ is equal.  

More specifically, using the notation of Theorem \ref{maindetail}, $M_1$ and $M_2$ are equivalent Riemannian covers of $M_0$
if and only if there exists a linear character $\chi \colon \tilde H_2 \rightarrow \C^*$ such that 
$$  \mathrm{ord}_{s=h} L_{M_1}(\bar \Xi \otimes \Res_{\tilde H_1}^{\tilde G} \Ind_{\tilde H_1}^{\tilde G} \Xi) =  \mathrm{ord}_{s=h} L_{M_2}(\bar \chi \otimes \Res_{\tilde H_2}^{\tilde G} \Ind_{\tilde H_1}^{\tilde G} \Xi) $$ 
and
$$  \mathrm{ord}_{s=h} L_{M_1}(\bar \Xi \otimes \Res_{\tilde H_1}^{\tilde G} \Ind_{\tilde H_2}^{\tilde G} \chi) = 
 \mathrm{ord}_{s=h} L_{M_2}(\bar \chi \otimes \Res_{\tilde H_2}^{\tilde G} \Ind_{\tilde H_2}^{\tilde G} \chi). $$
Again, there are  $\ell |H_2^{\ab}|$ linear characters $\chi$ on $\tilde H_2$, and the dimension of the representations involved is the index $[G:H_2]$. \qed 
\end{theorem}  

\begin{remark} In case of a surface of constant curvature $-1$, Theorem \ref{main} and Theorem \ref{mainL} are equivalent using a twisted version of the Selberg trace formula as in \cite[III.4.10]{Hejhal}.
\end{remark} 

\begin{remark} 
In \cite[Theorem 3.1]{Lseries}, it was proven that global function fields $K$ and $L$ of algebraic curves over finite fields are isomorphic if and only if there is a group isomorphism between their abelianised absolute Galois groups such that the corresponding $L$-series are equal. In \cite{Voloch}, it is shown that one may restrict to unramified characters by admitting extensions of the ground field. 

For negatively curved (e.g., hyperbolic) manifolds $M_1$ and $M_2$, the corresponding question is whether they are isometric if and only if there is an isomorphism $\psi \colon \Ho_1(M_1,\Z) \cong \Ho_1(M_2,\Z)$ such that $L_{M_2}(\chi,s)=L_{M_1}(\chi \circ \psi,s)$ for all $\chi \in \Hom(\pi_1(M_2),\C^*)= \Hom(\Ho_1(M_2,\Z),\C^*)$.
\end{remark}

\begin{remark} The geodesic length function 
defines the \emph{marked length spectrum} $[\gamma] \mapsto \ell(\gamma)$ as a function from conjugacy classes in $\pi_1(M)$ to $\R_{>0}$. Croke and Otal \cite{Croke} \cite{Otal} showed that this characterises the isometry class of $M$ in dimension two (in arbitrary dimension, this is an open conjecture of Burns and Katok). This motivates the question: what is the relation between the marked length spectrum and the information encoded in the $L$-series $L_M(\rho,s)$ where $\rho$ runs over all unitary representations of $\pi_1(M)$ (this information might be called the ``twisted length spectrum'')? 
\end{remark}

\bibliographystyle{amsplain} 
\providecommand{\bysame}{\leavevmode\hbox to3em{\hrulefill}\thinspace}
\providecommand{\MR}{\relax\ifhmode\unskip\space\fi MR }
\providecommand{\MRhref}[2]{%
  \href{http://www.ams.org/mathscinet-getitem?mr=#1}{#2}
}
\providecommand{\href}[2]{#2}

\end{document}